\documentclass{article}

\usepackage[english]{babel}
\usepackage[utf8]{inputenc}
\usepackage[T1]{fontenc}
\usepackage{amsmath}
\usepackage{amssymb}
\usepackage{amsthm}
\usepackage{mathtools}
\usepackage{dsfont}
\usepackage{bm}
\usepackage{physics}
\usepackage{caption}
\usepackage{hyperref}
\usepackage{xcolor}
\hypersetup{
    colorlinks,
    linkcolor={red!50!black},
    citecolor={blue!50!black},
    urlcolor={blue!80!black}
}
\usepackage{enumerate}

\theoremstyle{plain}
\newtheorem{definition}{Definition}[section]
\newtheorem{lemma}[definition]{Lemma}
\newtheorem{prop}[definition]{Proposition}
\newtheorem{theorem}[definition]{Theorem}
\newtheorem{corol}[definition]{Corollary}
\newtheorem{hypothesis}[definition]{Hypothesis}

\theoremstyle{remark}

\newtheorem{ex}[definition]{Example}
\newtheorem{remark}[definition]{Remark}

\usepackage{tikz}
\usetikzlibrary{calc}
\usetikzlibrary{ decorations.markings}

\usepackage{geometry}

\newcommand{\ind}{\mathds{1}}
\newcommand{\surf}{S}
\newcommand{\wordsA}{\mathcal{X}}
\newcommand{\wordsB}{\mathcal{Y}}
\newcommand{\N}{\mathbb{N}}
\newcommand{\E}{\mathds{E}}

\newcommand{\Real}{\mathbb{R}}
\newcommand{\C}{\mathbb{C}}

\newcommand{\ee}{\mathrm{e}}
\newcommand{\algA}{\mathcal{A}}
\newcommand{\algB}{\mathcal{B}}
\newcommand{\Unit}{\mathbb{U}}
\newcommand{\Weingarten}{\mathrm{Wg}}
\newcommand{\Sym}{\mathfrak{S}}

\newcommand{\partition}{\mathcal{P}}
\newcommand{\Inv}{\mathcal{I}}

\usepackage{listofitems}
\newcommand\cycle[2][\,]{%
  \readlist\thecycle{#2}%
  (\foreachitem\i\in\thecycle{\ifnum\icnt=1\else#1\fi\i})%
}

\DeclareMathOperator{\Cyc}{Cycles}

\DeclareMathOperator{\Id}{Id}

\DeclareMathOperator{\val}{val}
\DeclareMathOperator{\Supp}{Supp}

\newcommand{\mwalks}{\overrightarrow{w}}
\newcommand{\mhurwitz}{\overrightarrow{h}}

\newcommand{\mwset}{\overrightarrow{\mathcal{W}}}
\newcommand{\carte}{\mathcal{C}}
\newcommand{\W}{\mathcal{W}}
\newcommand{\tW}{\tilde{\mathcal{W}}}
\newcommand{\M}{\mathcal{M}}
\newcommand{\D}{\mathcal{D}}
\newcommand{\opP}{\mathcal{P}}

\author{Thomas Buc--d'Alch\'e
\thanks{UMPA UMR 5669, ENS de Lyon, CNRS; 46, all\'ee d’Italie 69007, Lyon, France;  email:thomas.buc-dalche@ens-lyon.fr}}

\title{Topological expansion of unitary integrals and maps}

\date{}

\begin{document}

\maketitle

\begin{abstract}
  In this article, we study integrals on the unitary group with respect
  to the Haar measure. We give a combinatorial interpretation in terms
  of maps of the asymptotic topological expansion, established
  previously by Guionnet and Novak. The maps we introduce -- the maps
  of unitary type -- satisfy Tutte-like equations. It allows us to
  show that in the perturbative regime, they describe the different
  orders of the asymptotic topological expansion. Furthermore, they
  generalize the monotone Hurwitz numbers.
\end{abstract}

\section{Introduction}

In the breakthrough article \cite{brezin_planar_1978}, Brézin and al.
used random matrix theory to address the problem of enumeration of
maps, graphs embedded in surfaces up to homeomorphisms. The
topological properties of Feynman diagrams had previously been shown
to be critical in the work of 't Hooft \cite{t_hooft_planar_1974}, thus
relating the combinatorics of maps to field theory (see also the
review article \cite{bessis_quantum_1980}). For instance, the planar
diagrams give the leading order in the expansion of physically
significant quantities.

The random matrix approach to the enumeration of maps pioneered by
Brézin and al. subsequently found many applications. Harer and Zagier
used the same approach to study the topological properties of the
moduli space of curves \cite{harer_euler_1986}. In the celebrated
article \cite{kontsevich_intersection_1992}, Kontsevitch used matrix
integrals to solve Witten's conjecture. See also
\cite{di_francesco_2d_1995} for a review of the application of random
matrix theory to combinatorial problems appearing in 2D gravity. More
generally, random matrices provide a powerful tool to address hard
combinatorial problems such as the problem of the enumeration of
Riemann surfaces, see the work of Eynard \cite{eynard_counting_2016}.
For another approach on the enumeration of maps, see for instance
\cite{bouttier_census_2002}.

In all the problems above, the matrix models used are related to the
Gaussian Unitary Ensemble (GUE). Let
$\dd M = \prod_{i}\dd M_{ii}\prod_{i < j}\dd \Re(M_{i j})\dd \Im(M_{i j})$
be the Lebesgue measure on the space of Hermitian matrices
$\mathcal{H}_{N}$ and $V$ be a polynomial called the potential. We
consider the measure
\begin{equation*}
  \begin{split}
    \mu^{N}_{\text{GUE}, V} = \frac{1}{Z_{\text{GUE}, V}^{N}}\ee^{-N\Tr V(M) - \frac{N}{2}\Tr M^{2}}\dd M,
  \end{split}
\end{equation*}
where the normalization constant is the partition function
\begin{equation*}
  \begin{split}
    Z_{\text{GUE}, V}^{N} = \int_{\mathcal{H}_{N}}\ee^{-N\Tr V(M) -\frac{N}{2}\Tr M^{2}}\dd M.
  \end{split}
\end{equation*}

Many relevant quantities, such as the partition function can be
expressed as a formal series of maps using Wick's formula (see \cite{zvonkin_matrix_1997} for
an introduction). For instance, for $V(M) = t q(M)$, with
$q(M) = M^{4}$ and $t \in \Real$, we get the formal expansion in the
dimension $N$ of the random matrix $M$,
\begin{equation*}
  \begin{split}
    \ln Z_{\text{GUE}, V}^{N}& = \sum_{g\geq 0}N^{2-2g}\sum_{n \geq 0}\frac{(-t)^{n}}{n!}\M^{\text{GUE}, (g)}_{q, n}\,,
  \end{split}
\end{equation*}
where $\M^{\text{GUE}, (g)}_{q, n}$ is the number of connected maps of
genus $g$ with $n$ vertices, all of them of degree 4. Notice that the
term of order $N^{2-2g}$ in this expansion is a generating series of
maps of genus $g$. We call such an expansion a (formal) \textbf{topological
  expansion}.

In general, the above equality holds in the sense of formal power
series, see \cite{eynard_formal_2011} for instance. By the above
equality, we mean that the derivatives with respect to $t$ at $t=0$ of
the left and right sides of the equation above coincide. In fact, the
series of maps on the right side may not converge in general.

We can replace this divergent series with an asymptotic expansion as
$N \to \infty$, where the equality holds up to an error of order $N^{-p}$,
for some integer $p$. Ercolani and McLaughlin obtained such an
expansion in a one-matrix model, for a potential whose coefficients
are close to zero \cite{ercolani_asymptotics_2003}. The case with several random matrices was
studied by Guionnet and Maurel-Segala
\cite{guionnet_combinatorial_2006, guionnet_second_2007}, and
Maurel-Segala \cite{maurel-segala_high_2006}. More complicated models involving not only a matrix
from the GUE but also deterministic matrices, sometimes called models
with external sources, have been studied, see \cite{brezin_random_2016}.

The multi-matrix models display much more variety. As for the
one-matrix models, they were first studied by physicists, see the
reviews \cite{gross_two_1991,di_francesco_2d_1995}. From an analytical
point of view, they are harder to solve than one-matrix models, see
for instance the works of Mehta \cite{mehta_method_1981}, and from a
combinatorial point of view, they allow to address a wealth of
combinatorial problems as they are related to the enumeration of
colored maps, see for instance \cite{guionnet_combinatorial_2006}.

In this article, we establish a similar link between integrals of
unitary matrices and the combinatorics of some maps. More precisely,
we introduce new maps, the maps of unitary type (Definition
\ref{def:maps_unitary_type}), that describe the topological expansion.
These maps allow us to relate the Weingarten calculus and the
Dyson-Schwinger equation -- two important ways to study unitary
integrals. In a particular case, the maps of unitary type are related
to the Hurwitz numbers. In this way, we generalize part of the results
obtained in \cite{goulden_monotone_2011}, that relate a particular integral, the HCIZ integral,
to Hurwitz numbers.

The Haar unitary matrices share the same unitary invariance as
matrices of the Ginibre ensemble. An expansion in terms of
non-crossing permutations for expectation of traces of words of
Ginibre matrices $G_{i}$ has been obtained in \cite{mingo_annular_2004}. This point of view
in terms of non-crossing permuations is similar to the interpretation
in terms of maps. For instance, the annuli considered in multi-annulus
permutations correspond to the vertices in a map. A genus expansion in
terms of maps has been obtained in \cite{dubach_words_2021}. They consider only some
pairings, \textit{admissible pairings}, which corresponds to the fact
that edges are oriented in the maps of unitary type. In particular,
the unitary invariance of the Ginibre ensemble implies that to have a
non-zero expectation, the words in Ginibre matrices considered must be
\textit{balanced}, i.e. contain as many $G_{i}$ as $G_{i}^{*}$. A similar
condition appears for Haar unitary matrices.

We introduce some notation. We consider matrices of dimension
$N \in \N^{*} = \{1, 2, 3, \ldots\}$. We denote by $\Tr A = \sum_{i=1}^{N}A_{ii}$
the trace of a matrix $A$, and by $\tr = \Tr/N$ the normalized trace.
Notice that both $\Tr$ and $\tr$ depend on the dimension $N$. The
conjugate transpose of a matrix $M$ is denoted by $M^{*}$. Let
$p \in \N^{*}$. For all $N \geq 1$, we fix $p$ deterministic matrices
$A_{1}^{N}, \ldots, A_{p}^{N}$ of size $N \times N$. The matrix $U^{N}$ will be
a unitary matrix of size $N \times N$, i.e. an element of the unitary group
$\Unit(N)$, and $(U^{N})^{*} = (U^{N})^{-1}$ will be its conjugate
transpose.

Let $\dd U^{N}$ be the Haar measure on the unitary group $\Unit(N)$,
and $V$ be a non-commutative polynomial in several variables, that
does not depend on $N$. The measure $\mu^{N}_{V}$ is given by
\begin{equation}\label{eq:measure}
  \begin{split}
    \dd\mu_{V}^{N}(U^{N}) &= \frac{1}{Z^{N}_{V}}\exp\left(N\Tr V\Big(U^{N}, (U^{N})^{*}, A_{1}^{N}, (A_{1}^{N})^{*}, \ldots, A_{p}^{N}, (A_{p}^{N})^{*}\Big)\right)\dd U^{N},
  \end{split}
\end{equation}
where the partition function $Z^{N}_{V}$ is
\begin{equation}
  \begin{split}
    Z^{N}_{V} &= \int_{\Unit(N)}\exp\left(N\Tr V\Big(U^{N}, (U^{N})^{*}, A_{1}^{N}, (A_{1}^{N})^{*}, \ldots, A_{p}^{N}, (A_{p}^{N})^{*}\Big)\right)\dd U^{N}.
  \end{split}
\end{equation}

We will evaluate all non-commutative polynomials at the matrices
$U^{N}, (U^{N})^{*}, A_{1}^{N}, (A_{1}^{N})^{*}, \ldots, A_{p}^{N}, (A_{p}^{N})^{*}$
and will omit writing this explicitly in the sequel, e.g. writing
$\Tr(V)$ to mean
$\Tr(V(U^{N}, (U^{N})^{*}, A_{1}^{N}, \ldots, A_{p}^{N}))$.

In Section \ref{sec:multimatrix}, we will consider measures of the
form
\begin{equation*}
  \begin{split}
    \frac{1}{Z^{N}_{V}}\exp(N\Tr V)\dd U^{N}_{1}\cdots \dd U^{N}_{n},
  \end{split}
\end{equation*}
where $V$ is a noncommutative polynomial that depends on
$U^{N}_{1}, \ldots, U^{N}_{n}$, all independent and Haar-distributed.

We will assume the two following hypotheses.
\begin{hypothesis}\label{hyp:real}
  For all $N \geq 1$ and for all
  $U_{1}, \ldots, U_{n} \in \Unit(N)^{n}$, $\Tr V$ is real.
\end{hypothesis}
\begin{hypothesis}\label{hyp:bound-infty}
  Assume
  \begin{equation*}
    \sup_{N \geq 1}\sup_{1 \leq i \leq p}\|A^{N}_{i}\| < \infty\,,
  \end{equation*}
  where $\|\cdot\|$ is the operator norm.
\end{hypothesis}

In most of the article, we will not assume
Hypothesis \ref{hyp:bound-infty}, but rather assume Hypothesis
\ref{hyp:bound}:
\begin{hypothesis}\label{hyp:bound}
  For all $N \geq 1$ and for all $1 \leq i \leq p$,
  $\|A^{N}_{i}\| \leq 1$, where $\|\cdot\|$ is the operator norm.
\end{hypothesis}
This will prove convenient, and will not change the main result,
Theorem \ref{thm:main} stated below. Stating Theorem \ref{thm:main}
with Hypothesis \ref{hyp:bound-infty} instead of \ref{hyp:bound}
corresponds to rescaling the coefficients of the polynomial $V$.

Hypothesis \ref{hyp:real} implies that the measure $\mu^{N}_{V}$ is a probability
measure, and in particular that $Z^{N}_{V} \in (0, +\infty)$. We write the
potential $V$ as a sum of monomials $q_{i}$ with complex coefficients
$z_{i}$, $V = \sum_{i}z_{i}q_{i}$. Thus, we will sometimes consider the
partition functions, cumulants, etc. as functions of
$\bm{z} = (z_{1}, z_{2}, \ldots )$. With this notation, the reality
conditions is
\begin{equation*}
  \sum_{i}z_{i}\Tr(q_{i}) = \sum_{i}\overline{z_{i}}\Tr(q_{i}^{*})\,.
\end{equation*}
Notice that for generic $q_{i}$'s, $\Tr V$ might be real for only
specific values of $\bm{z}$.

Notice that when considering the partition function with potential
$V = tAU^{N}B(U^{N})^{*}$, where $t\in \C$ and $A, B$ are self-adjoint matrices, we
recover the Harish-Chandra-Itzykson-Zuber (HCIZ) integral
\begin{equation*}
  \begin{split}
    Z^{N}_{V} = \int_{\Unit(N)}\exp(tN\Tr(AU^{N}B(U^{N})^{*}))\dd U^{N},
  \end{split}
\end{equation*}
which was first studied by Harish-Chandra
\cite{harish-chandra_differential_1957} and Itzykson and Zuber
\cite{itzykson_planar_1980}, and whose asymptotics have been since investigated,
see \cite{zinn-justin_integrals_2003, goulden_monotone_2014, guionnet_asymptotics_2015, novak_complex_2020}.

We will compute joint moments and cumulants (see Definition \ref{def:moment-cumulant}) of
the random variables $\Tr(P_{1}), \ldots, \Tr(P_{l})$ under $\mu^{N}_{V}$
(for $V$ small), where the $P_{i}$ are non-commutative polynomials. In
\cite{collins_asymptotics_2009}, the first-order asymptotics of partition functions was studied.
In \cite{guionnet_asymptotics_2015}, it has been shown that the joint cumulants admit an asymptotic
expansion as $N\to \infty$, when the coefficients of the potential $V$ are
small enough.

The goal of this article is to give a combinatorial interpretation of
the coefficients of this expansion. We show that unitary matrix
integrals enumerate a particular family of maps, which we call maps of
unitary type. They are introduced in Section \ref{sec:maps_unitary}, Definition
\ref{def:maps_unitary_type}. This interpretation links the
Dyson-Schwinger equation, which is satisfied by sums of maps of
unitary type, and the Weingarten calculus studied first by Weingarten
\cite{weingarten_asymptotic_1978}, and then by \cite{samuel_un_1980}, whose results were rediscovered and expanded
upon by Collins \cite{collins_moments_2003} and Collins and \'Sniady \cite{collins_integration_2006}. See \cite{collins_weingarten_2022} for a
review.

Expansions in terms of combinatorial objects have already been
introduced for unitary matrices. For instance, in the case of the HCIZ
integral, expansions for the free energy using double Hurwitz numbers
are computed in \cite{goulden_monotone_2011}. In
\cite{collins_asymptotics_2009}, the leading order of the expansion of
unitary integrals is expressed in terms of maps with ``dotted edges''.
However, to our knowledge, no interpretation of these expansions using
maps has been obtained at all orders for the unitary integrals we
consider. As an interesting particular case, when considering
alternated polynomials (see Definition \ref{def:alternated_poly}), the
combinatorics of maps of unitary type is related to triple Hurwitz
numbers.

In the case of the GUE, integrals of random matrices and enumeration
of maps are related by Wick's formula. In the case of unitary
matrices, Wick's formula is replaced by Weingarten's formula. In
Section \ref{sec:Weingarten}, we express joint moments of random variables
$\Tr(P_{i})$, for non-commutative polynomials $P_{i}$, using
Weingarten's formula. In the case where the potential $V = 0$, we can
express such moments as weighted sums of permutations. In Section \ref{sec:maps},
we recall a few notions on maps and introduce the maps of unitary
type, which are our main combinatorial tools. This allows us to deduce
a topological expansion for the joint cumulants in the case of no
potential (i.e. $V = 0$). To address the general case $V \neq 0$, we
introduce generating series of maps of unitary type of the form
\begin{equation*}
  \M^{(g), N}_{V, l}(P_{1}, \ldots, P_{l}) =\sum_{\bm{n}\in\N^{k}}\frac{\bm{z}^{\bm{n}}}{\bm{n}!}\times \sum w_{N}(\carte, \bm{n}, V, P_{1}, \ldots, P_{l}),
\end{equation*}
where the second sum is on a set of connected maps $\carte$ of unitary
type (see Definition \ref{def:maps_unitary_type}) of genus $g$ which
depends on $V, P_{1}, \ldots, P_{l}, \bm{n}$. The term
$w_{N}(\carte, \bm{n}, V, P_{1}, \ldots, P_{l})$ is a weight which depends
on the size $N$, $\carte$, $\bm{n}$ and on the polynomials
$V, P_{1}, \ldots, P_{l}$. See Definition \ref{def:formal_cumulant}.

In Section \ref{sec:induction}, we describe a decomposition of maps of
unitary type, which can be interpreted as a cutting procedure. It
allows us to deduce induction relations -- similar to the topological
recursion of Chekhov, Eynard and Orantin, see \cite{chekhov_matrix_2006, eynard_algebraic_2008} --
on weighted sums $\M^{(g), N}_{V, l}$ of maps of unitary type of a
given genus $g$. This decomposition is reminiscent of a procedure
introduced by Tutte \cite{tutte_enumeration_1968}. In Section
\ref{sec:multimatrix}, we extend the results obtained so far to the
case of integrals over several independent random unitary matrices
$U^{N}_{1}, \ldots, U^{N}_{n}$.

It turns out that the induction relations obtained in Section
\ref{sec:induction} are related to the Dyson-Schwinger lattice. The
Dyson-Schwinger lattice (see \cite{guionnet_asymptotics_2015}) is a
family of equations relating cumulants together, which generalize the
Dyson-Schwinger equation (see Equation (\ref{eq:DS-problem})). This
equation admits under some hypotheses a unique solution
\cite{collins_asymptotics_2009}. Furthermore, in
\cite{guionnet_asymptotics_2015}, the Dyson-Schwinger lattice has been
used to establish the existence of an asymptotic expansion of the
cumulants, when $N\to\infty$. Let us assume Hypotheses \ref{hyp:real}
and \ref{hyp:bound}, and that the joint law of the matrices
$A^{N}_{i}$, $\tr$ admits an asymptotic expansion as $N \to \infty$.
For all $h$, we have an asympotic expansion for the renormalized joint
cumulants $N^{l-2}\W^{N}_{V, l}(P_{1}, \ldots, P_{l})$ (introduced in
Definition \ref{def:moment-cumulant-particular}) when the coefficients
of the potential $V$ are small enough
\begin{equation}\label{eq:expansion-key-thm}
  N^{l-2}\W_{V, l}^{N}(P_{1}, \ldots, P_{l}) = \sum_{g=0}^{h}\frac{\tau^{V}_{l, g}(P_{1}, \ldots, P_{l})}{N^{2g}} + o(N^{-2h}),
\end{equation}
where the coefficients $\tau^{V}_{l, g}(P_{1}, \ldots, P_{l})$ are
uniquely defined by some induction relations.

In Section \ref{sec:dyson-schw-equat}, we use the same techniques to
express the terms of this expansion in terms of maps of unitary type.
We thus obtain a topological expansion: the coefficient of
$\frac{1}{N^{2g}}$ in the expansion is a generating series of weighted
unitary type maps of genus $g$.

We thus improve on the result of \cite[Theorem
25]{guionnet_asymptotics_2015} by relaxing the hypotheses, showing
that the convergence is uniform in $g$ and $l$, and by giving a
combinatorial interpretation to the coefficients
$\tau^{V}_{l, g}(P_{1}, \ldots, P_{l})$.
\begin{theorem}[Main theorem]\label{thm:main}
  Assume that for all $N \geq 1$, $\Tr(V)$ is real for all
  $U_{1}, \ldots, U_{n} \in \Unit(N)^{n}$ and that
  $\sum_{N \geq 1}\sup_{1 \leq i \leq p}\|A_{i}^{N}\| < \infty$.

  There exists $\epsilon > 0$ such that if
  \begin{equation*}
    \|\bm{z}\|_{\infty} < \epsilon,
  \end{equation*}
  then for all $l \geq 1$, $g \geq 0$, and
  $\bm{P} = (P_{1}, \ldots, P_{l})$, we have the asymptotic expansion
  as $N \to \infty$
  \begin{equation*}
    N^{l-2}\W^{N}_{V, l}(P_{1}, \ldots, P_{l}) = \sum_{h=0}^{g}\frac{1}{N^{2h}}\M^{(h), N}_{V, l}(P_{1}, \ldots, P_{l}) + \order{N^{-2g-2}}.
  \end{equation*}
\end{theorem}
Notice that we do not require the trace $\Tr$ to have an asymptotic
expansion as in \cite[Theorem 25]{guionnet_asymptotics_2015}.

An interesting particular case described in Section \ref{sec:hurwitz}
is when all the polynomial involved are \textbf{alternated}, see
Definition \ref{def:alternated_poly}, that is if they can be written as
\begin{equation*}
  \begin{split}
    P = B_{1}^{N}U^{N}C_{1}^{N}(U^{N})^{*}\cdots B_{m}^{N}U^{N}C_{m}^{N}(U^{N})^{*},
  \end{split}
\end{equation*}
where $B_{i}^{N}$ and $C^{N}_{i}$ for $i = 1, \ldots, m$ are square $N \times N$
matrices. This is the case of the HCIZ integral in particular. In that
case, our sums of maps are related to the triple monotine Hurwitz
numbers, which count some ramified coverings of the sphere with at
most three nonsimple ramification points. We thus generalize the link
between the (double) monotone Hurwitz numbers and the HCIZ integral,
which had already been studied in \cite{goulden_monotone_2014}. See also
\cite{collins_tensor_2023} for a study of the HCIZ integral in the
tensor setting.

In Section \ref{sec:Weingarten}, we give definitions and recall
important consequences of the Weingarten calculus. In Section
\ref{sec:maps}, we introduce the maps of unitary types and show that
they describe the topological expansion of cumulants with respect to
the Haar mesure. When the polynomial are alternated, these maps are
related to the triple Hurwitz numbers. In Section \ref{sec:induction},
we give a decomposition of maps of unitary type and deduce induction
relations on sums of maps of a given genus and with prescribed
vertices, in the spirit of the work of Tutte
\cite{tutte_enumeration_1968}. In Section \ref{sec:dyson-schw-equat},
we study the Dyson-Schwinger equation and give the proof of the main
result.

\paragraph{Acknowledgement.}
I thank my PhD advisors, Alice Guionnet and Grégory Miermont for
countless discussions, tips and advice during this project. This
project was supported by ERC Project LDRAM : ERC-2019-ADG Project
884584. I thank two anonymous referee for their many helpful
recommandations.

\section{Weingarten calculus}\label{sec:Weingarten}

In this section, we first give a few definitions and introduce
notation pertaining to moments and cumulants of traces of random
matrices. Then, we give a short review of the Weingarten calculus.
This allows us to give expression for the expectation of a product of
traces of monomials in the matrices
$U^{N}, (U^{N})^{*}, A_{i}^{N}, (A^{*})^{N}$.

\subsection{Moments and cumulants}
\label{sec:moments-cumulants}

Let us consider $l \geq 1$ non-commutative polynomials
$P_{1}, P_{2}, \ldots, P_{l}$ in the variables $u, u^{-1}$, and
$ a_{i}, a_{i}^{*}$ for $1 \leq i \leq p$, with $p \in \N$. We define
the involution $*$ such that $u^{*} = u^{-1}$, for $1 \leq i \leq p$,
$(a_{i})^{*} = a_{i}^{*}$, and for any letters $X_{1}, \ldots, X_{k}$
in $\{u, u^{*}, a_{i}, a_{i}^{*} \colon 1 \leq i \leq p\}$ and
$z \in \C$, we have
$(zX_{1}\cdots X_{k})^{*} = z^{*}X_{k}^{*}\cdots X_{1}^{*}$. We denote
the unital $*$-algebra generated by such polynomials by
\begin{equation*}
  \begin{split}
    \algA = \C\langle{u, u^{-1}, a_{i}, a_{i}^{*}; 1 \leq i \leq p}\rangle.
  \end{split}
\end{equation*}
The unital $*$-algebra generated by the non-commutative polynomials in
$a_{1}, a_{1}^{*}, \ldots, a_{p}, a_{p}^{*}$ only is denoted by
$\algB$. We will evaluate all polynomials $P_{i}$ at the matrices
$U^{N}, (U^{N})^{*}, A_{1}^{N}, (A_{1}^{N})^{*}, \ldots, A_{p}^{N}, (A_{p}^{N})^{*}$
and will omit writing this explicitly in the sequel, e.g. writing
$\Tr(P)$ to mean
$\Tr(P(U^{N}, (U^{N})^{*}, A_{1}^{N}, \ldots, A_{p}^{N}))$. Notice
that there is no relation between the formal variables $u$ and
$u^{-1}$, or $a_{i}$ and $a_{i}^{*}$ for $i \in \N^{*}$ (except for
those involving $*$). We will denote by $\tr = \frac{1}{N}\Tr$ the
normalized trace.

In this article, we study the random variables
$\Tr(P_{1}), \ldots, \Tr(P_{l})$, seen as functions of $U^{N}$, under the
measure $\mu_{V}^{N}$ (see (\ref{eq:measure})). We will be interested in computing
the joint moments and cumulants of these random variables. To state
the definition of the cumulants, we introduce some notation about
partitions. We denote by $\partition(I)$ the set of partitions of a
finite set $I$. In particular, for $n\in \N^{*}$, we denote the set
$\{1, 2, \ldots, n\}$ by $[n]$. We denote the cardinality of a finite set $I$
by $|I|$. Given a partition $\pi \in \partition(I)$, $|\pi|$ is the number
of blocks of $\pi$.

\begin{definition}\label{def:moment-cumulant}
  Let $k\in \N^{*}$. The \textbf{joint moment} of the complex random
  variables $X_{1}, \ldots, X_{k}$ is
  \begin{equation*}
    \begin{split}
      m_{k}(X_{1}, \ldots, X_{k}) = \E\left[X_{1}X_{2}\cdots X_{k}\right]\,.
    \end{split}
  \end{equation*}

  The \textbf{joint cumulant} of the complex random variables
  $X_{1}, \ldots, X_{k}$ is $c_{k}(X_{1}, \ldots, X_{k})$, defined
  recursively by
  \begin{equation*}
    \begin{split}
      c_{k}(X_{1}, \ldots, X_{k}) = m_{k}(X_{1}, \ldots, X_{k}) - \sum_{\substack{\pi \in \partition([k])\\|\pi| \geq 2}}\prod_{B \in \pi}c_{|B|}(X_{i} \colon i \in B)\,.
    \end{split}
  \end{equation*}
\end{definition}
Notice that both the joint moments and cumulants are symmetric, multilinear
functions. The symmetry, which can be proven inductively, makes the expression
$c_{|B|}(X_{i} \colon i \in B)$ above unambiguous.

\begin{remark}\label{rem:other-def-cumulants}
  The cumulants can also be defined, see \cite{speed_cumulants_1983},
  as the coefficients of the series of the logarithm of the
  exponential generating series of the moments
  \begin{equation*}
    \begin{split}
      \sum_{n\geq0}\frac{z^{n}}{n!}c_{n}(X_{1}, \ldots, X_{1}) = \ln \sum_{n \geq 0}\frac{z^{n}}{n!}m_{n}(X_{1}, \ldots, X_{1})\,.
    \end{split}
  \end{equation*}
\end{remark}

\begin{definition}\label{def:moment-cumulant-particular}
  For $(P_{1}, \ldots, P_{l}) \in \algA^{l}$, we write the joint moments of the traces of
  $P_{i}$'s under $\mu^{N}_{V}$ as
  \begin{equation*}
    \begin{split}
      \alpha^{N}_{V, l}(P_{1}, \ldots, P_{l}) &= m_{l}(\Tr(P_{1}), \ldots, \Tr(P_{l})) = \int_{\Unit(N)}\Tr(P_{1})\cdots\Tr(P_{l})\dd\mu_{V}^{N}\,.
    \end{split}
  \end{equation*}
  We write the joint cumulants under $\mu^{V}_{N}$ as
  \begin{equation*}
    \begin{split}
      \W^{N}_{V, l}(P_{1}, \ldots, P_{l}) &= c_{l}(\Tr(P_{1}), \ldots, \Tr(P_{l}))\,, \\
    \end{split}
  \end{equation*}
  and introduce the renormalized cumulants
  \begin{equation*}
    \begin{split}
      \tilde{\W}^{N}_{V, l}(P_{1}, \ldots, P_{l}) &= N^{l-2}c_{l}(\Tr(P_{1}), \ldots, \Tr(P_{l}))\,.\\
    \end{split}
  \end{equation*}
\end{definition}

In Section \ref{sec:dyson-schw-equat}, we will discuss an asymptotic expansion (as $N \to \infty$)
for the joint cumulants. For now, we study the moments for $N$ fixed.
When $V = 0$, we can compute directly the moments using Weingarten's
formula, see Subsection \ref{sec:Weingarten-formula}. When $V \neq 0$, we can compute the
cumulants using the free energy $F_{V}^{N}$ defined in terms of the
partition function $Z^{N}_{V}$. Recall that
$V = \sum_{i=1}^{k}z_{i}q_{i}$ is the potential, a sum of $k$ polynomials
$q_{1}, \ldots, q_{k} \in \algA$ with complex coefficients $z_{1}, \ldots, z_{k}$.
Note that $V$ does not depend on $N$. We have
\begin{equation*}
  \begin{split}
    Z^{N}_{V} &= \int_{\Unit(N)}\exp(N\Tr(V))\dd U^{N},
  \end{split}
\end{equation*}
and we define the free energy as
\begin{equation}\label{eq:free-energy}
  F_{V}^{N} = \frac{1}{N^{2}}\ln Z^{N}_{V}.
\end{equation}
The free energy is always well defined when $\Tr V$ is real.

In the expression of the partition function, we can develop the
exponential as a series and exchange the sum and the integral:
\begin{equation*}
  \begin{split}
    Z^{N}_{V} &= \int_{\Unit(N)}\sum_{n_{1}, \ldots, n_{k} \geq 0}\prod_{i=1}^{k} \frac{(Nz_{i}\Tr(q_{i}))^{n_{i}}}{n_{i}!}\dd U^{N}\\
              &= \sum_{n_{1}, \ldots, n_{k} \geq 0}\prod_{i=1}^{k}\frac{(Nz_{i})^{n_{i}}}{n_{i}!}\int_{\Unit(N)}\Tr(q_{1})^{n_{1}}\cdots \Tr(q_{k})^{n_{k}}\dd U^{N}\,.
  \end{split}
\end{equation*}
In the second line, we used Hypothesis \ref{hyp:bound}, which implies
that $|\Tr(q_{i})|\leq N$, and the fact that we are integrating with
respect to the Haar measure on the compact group $\Unit(N)$ to
exchange the sum and the integral. Notice that this expression is
valid for all $\bm{z}$, even if $\Tr V$ is not real.

We introduce the notation $\bm{z}=(z_{1}, \ldots, z_{k})$, and for
$\bm{n} = (n_{1}, \ldots, n_{k}) \in \N^{k}$,
$\bm{z}^{\bm{n}} = \prod_{i=1}^{k}z_{i}^{n_{i}}$ and
$\bm{n}! = \prod_{i=1}^{k}n_{i}!$. Then,
\begin{equation*}
  \begin{split}
    Z^{N}_{V}&= \sum_{n\geq 0}N^{n}\sum_{\substack{\bm{n} \in \N^{k}\\n_{1} + \cdots + n_{k} = n}}\frac{\bm{z}^{\bm{n}}}{\bm{n}!}\alpha_{0, n}^{N}(\underbrace{q_{1}, \ldots, q_{1}}_{n_{1} \text{times}}, \ldots, \underbrace{q_{k}, \ldots, q_{k}}_{n_{k} \text{times}}),\\
  \end{split}
\end{equation*}
and therefore the partition function is a generating series of the
moments with respect to the Haar measure (i.e. with $V = 0$).

Similarly, the free energy is a generating series of the renormalized cumulants
for $V = 0$ (see \cite[Theorem 1.3.3, 4.]{bona_handbook_2015})
\begin{equation*}
  \begin{split}
    F^{N}_{V} &= \sum_{n \geq 1}\sum_{\substack{\bm{n} \in \N^{k}\\n_{1} + \cdots +n_{k} = n}}\frac{(N\bm{z})^{\bm{n}}}{\bm{n}!}\frac{1}{N^{2}}\W_{0, n}^{N}(\underbrace{q_{1}, \ldots, q_{1}}_{n_{1}\text{
    times }}, \ldots, \underbrace{q_{k}, \ldots, q_{k}}_{n_{k}\text{ times
    }})\\
              &= \sum_{n \geq 1}\sum_{\substack{\bm{n} \in \N^{k}\\n_{1} + \cdots + n_{k} = n}}\frac{\bm{z}^{\bm{n}}}{\bm{n}!}\tilde{\W}_{0, n}^{N}(\underbrace{q_{1}, \ldots, q_{1}}_{n_{1}\text{
    times }}, \ldots, \underbrace{q_{k}, \ldots, q_{k}}_{n_{k}\text{ times
    }})\\
  \end{split}
\end{equation*}
Notice that the free energy a priori exists only for $\bm{z}$
sufficiently small. Indeed, $Z^{N}_{V}$ is defined for all $\bm{z}$
but is nonzero on a open neighborhood of $0$ which depends on $N$. In
particular, the radius of convergence of $F_{V}^{N}$ a priori depends
on $N$.

Notice that by modifying the potential $V$ and differentiating, we have
\begin{equation*}
  \begin{split}
    \frac{\partial}{\partial t}\Big|_{t = 0}F^{N}_{V + tP} &= \frac{1}{N}\int_{\Unit(N)}\Tr(P)\dd\mu^{N}_{V}(U^{N}) = \frac{1}{N}\alpha_{V, 1}^{N}(P) = \tilde{\W}^{N}_{V, 1}(P).
  \end{split}
\end{equation*}

In general, we can prove by induction the following lemma, which is a
consequence of the definition of cumulants given in Remark
\ref{rem:other-def-cumulants}.
\begin{lemma}\label{lem:diff-free-energy}
  The renormalized joint cumulants are given by
  \begin{equation*}
    \begin{split}
      \tilde{\W}^{N}_{V, l}(P_{1}, \ldots, P_{l}) &= \frac{\partial^{l}}{\partial t_{1}\partial t_{2}\cdots \partial t_{l}}\Big|_{t_{1} = \ldots = t_{l} = 0}F^{N}_{V + \sum_{i}t_{i}P_{i}}.
    \end{split}
  \end{equation*}
\end{lemma}

Lemma \ref{lem:diff-free-energy} implies that for a fixed $N$, there exists a
neighborhood $U_{0} \in \C^{k}$ of $0$ such that for $\bm{z} \in U_{0}$,
\begin{equation}\label{eq:expr-cumulant}
  \begin{split}
    \tilde{\W}^{N}_{V, l}(P_{1}, \ldots, P_{l})
    &= \sum_{n \geq 0}\sum_{\substack{\bm{n} \in \N^{k}\\n_{1} + \cdots + n_{k} = n}}\frac{\bm{z}^{\bm{n}}}{\bm{n}!}\tilde{\W}_{0, n}^{N}(\underbrace{q_{1}, \ldots, q_{1}}_{n_{1}\text{
    times }}, \ldots, \underbrace{q_{k}, \ldots, q_{k}}_{n_{k}\text{ times
    }}, P_{1}, \ldots, P_{l}).
  \end{split}
\end{equation}

In the next subsections, we compute the moments with respect to the
Haar measure. From these moments and Definition
\ref{def:moment-cumulant}, we can compute the cumulants with respect
to the Haar measure. The expression (\ref{eq:expr-cumulant}) motivates
the introduction in Section \ref{sec:formal-topological-exp} of a
formal sum. The first terms of this sum are shown to give the
asymptotic expansion of the cumulants in Theorem \ref{thm:main}.

\subsection{The Weingarten formula}\label{sec:Weingarten-formula}
To compute the moments with respect to the Haar measure, the key tool
is Weingarten's formula, first obtained in
\cite{weingarten_asymptotic_1978}, which expresses the average of
coefficients of a unitary matrix in terms of the Weingarten function
defined below (Definition \ref{def:Weingarten-function}). See
\cite{collins_weingarten_2022} for a review on the Weingarten
calculus.

The Weingarten formula involves a sum over permutations. Let us fix
some notation pertaining to permutations. For $I$ a finite set, we
denote by $\Sym(I)$ the set of permutations on this set. In
particular, $\Sym_{n} = \Sym([n])$ is the set of permutations on
$[n] = \{1, 2, \ldots, n\}$. A permutation $\sigma$ admits a decomposition in
disjoint cycles. The set of cycles of $\sigma$ is denoted by $\Cyc(\sigma)$ and
its number of cycles is denoted by $c(\sigma)$. A cycle $c \in \Cyc(\sigma)$ is
written $\cycle{u_{1}, u_{2}, \ldots, u_{k}}$ with distinct
$u_{1}, \ldots, u_{k}$. It is the permutation whose support is
$\{u_{1}, \ldots, u_{k}\}$ and such that $c(u_{i}) = u_{i+1}$ with the
convention $u_{k+1}=u_{1}$.

We also introduce the modified traces $\Tr_{\sigma}(\bm{M})$ and
$\tr_{\sigma}(\bm{M})$ for $\sigma \in \Sym(I)$ and $\bm{M} = (M_{i}, i\in I)$ a
tuple of matrices, defined by
\begin{equation}\label{eq:not_trace_permutation}
  \begin{split}
    \Tr_{\sigma}(\bm{M}) &= \prod_{c\in\Cyc(\sigma)}\Tr(\overrightarrow{\prod}_{i\in c}M_{i}),\\
    \tr_{\sigma}(\bm{M}) &= \prod_{c\in\Cyc(\sigma)}\tr(\overrightarrow{\prod}_{i\in c}M_{i}) = N^{-c(\sigma)}\Tr_{\sigma}(\bm{M}),
  \end{split}
\end{equation}
where, if $c = \cycle{i_{1}, \ldots, i_{k}}$ is a cycle of the
permutation $\sigma$, the notation
$\overrightarrow{\prod}_{i\in c}M_{i}$, stands for the non-commutative
product $M_{i_{1}}M_{i_{2}}\cdots M_{i_{k}}$. Notice that such a
non-commutative product is defined up to circular permutation. The
trace property ensures that the quantity $\Tr_{\sigma}(\bm{M})$ is
well defined.

\begin{definition}\label{def:Weingarten-function}
  Let $q \leq N$ be an integer. The \textbf{Weingarten function}
  $\Weingarten_{N}\colon \Sym_{q} \to \C$ is defined for all $\pi\in\Sym_{q}$ by
  \[
    \Weingarten_{N}(\pi) = \int_{\Unit(N)}(U^{N})_{11}\cdots (U^{N})_{qq}\overline{(U^{N})_{1\pi(1)}\cdots (U^{N})_{q\pi(q)}}\dd U^{N}.
  \]
\end{definition}
This function can also be defined for all $q \in \N^{*}$ using
characters of the symmetric group (see \cite{collins_integration_2006}). The invariance of the
Haar measure by multiplication by permutation matrices implies that
the Weingarten function is invariant by conjugation, i.e. for all
$\sigma, \pi \in \Sym_{q}$ we have
\begin{equation*}
  \Weingarten_{N}(\sigma\pi\sigma^{-1}) = \Weingarten_{N}(\pi).
\end{equation*}

With our definition of the Weingarten function, Weigarten's formula is
valid in the case $q \leq N$. It actually holds for all $q \geq 1$ with an
appropriate definition of the Weingarten function.
\begin{theorem}{(Weingarten's formula, see \cite{collins_moments_2003} and \cite{collins_integration_2006})}
  Let $U^{N}$ be a Haar-distributed unitary matrix of size $N\times N$ and
  $\bm{i}=(i_{1}, i_{2}, \ldots, i_{q}), \bm{j}=(j_{1}, j_{2}, \ldots, j_{q}), \bm{i'}=(i'_{1}, i'_{2}, \ldots, i'_{q'})$
  and $\bm{j} = (j'_{1}, j'_{2}, \ldots, j'_{q'})$ be elements of $[N]^{q}$
  or $[N]^{q'}$ for $q, q' \geq 1$.
\begin{equation}
  \int_{\Unit(N)}(U^{N})_{i_{1}j_{1}}\cdots (U^{N})_{i_{q}j_{q}}\overline{(U^{N})_{i'_{1}j'_{1}}\cdots (U^{N})_{i'_{q'}j'_{q'}}}\dd U^{N} = \delta_{q, q'}\sum_{\rho, \sigma\in\Sym_{q}}\prod_{k=1}^{q}\delta_{i_{k},i'_{\sigma(k)}}\delta_{j_{k}, j'_{\rho(k)}}\Weingarten_{N}(\sigma\rho^{-1})\,.
\end{equation}
\end{theorem}

Before giving the expression for the moments with respect to the Haar
measure $\alpha_{0, l}^{N}(P_{1}, \ldots, P_{l})$, let us make a
simplifying assumptions on our polynomials $P_{i}$.

We introduce the set $\wordsB$ of words in the letters
$a_{1}, a_{1}^{*}, \ldots, a_{p}, a^{*}_{p}$. We assume that for all
$i$, $P_{i}$ can be written uniquely as
\begin{equation}\label{eq:wordsX}
  M_{i, 1}u^{\epsilon_{i, 1}}M_{i, 2}u^{\epsilon_{i, 2}}\cdots M_{i, d_{i}}u^{\epsilon_{i, d_{i}}},
\end{equation}
where $M_{i, j}$ is either the empty word or an element of $\wordsB$,
$d_{i} \geq 1$, and
$\bm{\epsilon_{i}} = (\epsilon_{i, 1}, \ldots, \epsilon_{i, d_{i}}) \in \{\pm 1\}^{d_{i}}$. We write
$\wordsA$ the set of such monomials. We have $\wordsB \subset \wordsA$.
Notice that $\algA$ is generated by the elements of $\wordsA$ up to
cyclic permutation of the factors in a monomial.

The integer $d_{i}$, that we will sometime write $\deg P_{i}$, is the
\textbf{degree} of the monomial $P_{i}$. Notice that there is no
relation between the formal variables, in particular between $u$ and
$u^{-1}$ (except for those involving $*$). Therefore, the degree of
\eqref{eq:wordsX} is defined by counting the total number of letter
$u$ or $u^{*}$ in a word. In particular, $\deg(uu^{-1}) = 2$.

\begin{definition}\label{def:L_prime}
  With $(P_{1}, \ldots, P_{l}) \in \wordsA^{l}$, and with the notation
  \eqref{eq:wordsX}, we set
  \begin{itemize}
    \item $\bm{P} = (P_{1}, \ldots, P_{l})$,
    \item $\bm{M}_{\bm{P}} = (M_{i})_{i \in [\sum_{i}\deg P_{i}]} = (M_{1, 1}, \ldots, M_{1, d_{1}}, \ldots, M_{l, 1}, \ldots, M_{l, d_{l}})$,
    \item
          $\bm{\epsilon}_{\bm{P}}  = (\epsilon(i))_{i \in [\sum_{i}\deg P_{i}]} = (\epsilon_{1, 1}, \ldots, \epsilon_{1, d_{1}}, \ldots, \epsilon_{l, 1}, \ldots, \epsilon_{l, d_{l}})$.
  \end{itemize}
  Notice that we change the indices of the monomials $M_{i, j}$ and of
  $\bm{\epsilon}_{i, j}$, by setting for all
  $1 \leq i \leq k, 1 \leq j \leq d_{i}$,
  $M_{d_{1} + \cdots +d_{i-1} + j} = M_{i, j}$ and
  $\epsilon(d_{1} + \cdots d_{i-1} + j) = \epsilon_{i, j}$.

  We set $\deg \bm{P} = \sum_{i}\deg P_{i}$.

  Furthermore, we define the permutation
  \begin{equation}\label{eq:not_gamma}
      \gamma_{\bm{P}} = \cycle{1, \ldots, d_{1}}\cycle{d_{1} + 1, \ldots, d_{2}}\cdots\cycle{d_{l-1}+1, \ldots, d_{l}}\,.
  \end{equation}
\end{definition}
In the sequel, we shall consider $\bm{\epsilon}_{\bm{P}}$ as a function, but
sometime using vector notation for convenience. In particular, we
consider the sets.
$\bm{\epsilon}_{\bm{P}}^{-1}(+1) = \{i \in [\deg \bm{P}]\colon \bm{\epsilon}_{\bm{P}}(i) = +1\}$
and
$\bm{\epsilon}_{\bm{P}}^{-1}(-1) = \{i \in [\deg \bm{P}]\colon \bm{\epsilon}_{\bm{P}}(i) = -1\}$.

\begin{remark}\label{rem:relabelling}
  The permutation $\gamma_{\bm{P}}$ defined by \eqref{eq:not_gamma}
  gives a choice of labelling for the letters $u$ and $u^{*}$ in the
  monomials we consider. Notice that this choice is arbitrary. The
  cycle notation is convenient here as we are interested in traces of
  such monomials. The ordering of the letters in the words needs only
  to be specified up to cyclic permutation.

  For any permutation $\sigma \in \Sym_{\deg \bm{P}}$, we can replace
  $\gamma_{\bm{P}}, \bm{M}_{\bm{P}}=(M_{i}), \bm{\epsilon}_{\bm{P}}=(\epsilon(i))$ by
  $\gamma' = \sigma^{-1}\gamma_{\bm{P}}\sigma, \bm{M}' = (M'_{i}) = (M_{\sigma(i)}), \bm{\epsilon}' = (\epsilon'(i)) = (\epsilon(\sigma(i)))$.
  This new data describes the same polynomials. By this we means that
  if we write $\gamma' = c_{1}'\cdots c_{l}'$ the decomposition in disjoint
  cycles of $\gamma'$, we have
  \begin{equation*}
    \prod_{i=1}^{l}\Tr(P_{i}) = \prod_{i=1}^{l}\Tr \left( \overrightarrow{\prod}_{j \in c_{i}}M_{j}u^{\bm{\epsilon}'(j)} \right)\,.
  \end{equation*}
  Notice that the non-commutative product is only defined up to the
  cyclic permutation of the factors. The cyclic property of the trace
  ensures that the quantity on the right-hand side is well defined.
\end{remark}

We can assume all the polynomials are of the form (\ref{eq:wordsX})
without loss of generality as $\alpha_{V, l}^{N}$ is multilinear and
satisfies the trace property
\begin{equation*}
  \begin{split}
    \alpha_{V, l}^{N}(P_{1}, \ldots, P_{i-1}, P_{i}Q, P_{i+1}, \ldots, P_{l}) = \alpha_{V, l}^{N}(P_{1}, \ldots, P_{i-1}, QP_{i}, P_{i+1}, \ldots, P_{l})\,,
  \end{split}
\end{equation*}
as $\Tr(P_{i}Q) = \Tr(QP_{i})$. Furthermore, if there exists $i$ such
that $P_{i}$ contains no letter $u$ nor $u^{-1}$, we can factor the
term $\Tr(P_{i})$ out of the moment.

The formula for the moments with respect to the Haar measure involves
permutations belonging to the set
$\Sym^{(\epsilon)}(I) \subset \Sym(I)$ of permutations (introduced in
\cite{mingo_second_2007}), for $\epsilon \in \{\pm 1\}^{I}$.
\begin{definition}\label{def:permutations-epsilon}
  Let $\epsilon \in \{\pm 1\}^{I}$. The set
  $\Sym^{(\epsilon)}(I) \subset \Sym(I)$ is the set of permutations
  $\pi \in \Sym(I)$ such that
  \begin{equation*}
    \begin{split}
      \pi\left(\epsilon^{-1}(+1)\right) = \epsilon^{-1}(-1)\,.
    \end{split}
  \end{equation*}
  Furthermore, we define
  $\pi^{(\epsilon)} = \pi^{2}|_{\epsilon^{-1}(+1)}\in\Sym(\epsilon^{-1}(+1))$.
\end{definition}
Notice that the set $\Sym^{(\epsilon)}(I)$ is empty if
$|\epsilon^{-1}(+1)| \neq |\epsilon^{-1}(-1)|$.

\begin{ex}
  For instance, if $\epsilon = (+1, +1, -1, +1, -1, -1)$, then
  $\pi = \cycle{1, 3, 4, 6}\cycle{2, 5} \in \Sym_{6}^{(\epsilon)}$, and
  $\pi^{(\epsilon)} = \cycle{1, 4}\cycle{2}$.
\end{ex}

The notation of Definitions \ref{def:L_prime} and \ref{def:permutations-epsilon} allows us to express the
moments in a compact way.
\begin{prop}[{\cite[Proposition 3.4]{mingo_second_2007}}]\label{prop:comput_moment}
  Let $\bm{P} = (P_{1}, \ldots, P_{l}) \in \wordsA^{l}$. We have
  \begin{equation}\label{eq:moments_V=0}
    \begin{split}
      \alpha^{N}_{0, l}(\bm{P}) &= \alpha^{N}_{0, l}(P_{1}, \ldots, P_{l}) = \sum_{\pi\in\Sym^{(\epsilon_{\bm{P}})}_{\deg \bm{P}}}\Tr_{\gamma_{\bm{P}}\pi^{-1}}(\bm{M}_{\bm{P}})\Weingarten_{N}(\pi^{(\epsilon_{\bm{P}})}).
    \end{split}
  \end{equation}
\end{prop}

\subsection{Expansion of the Weingarten function}
We wish to express the moments and cumulants uniquely in terms of combinatorial
objects and traces. To this end, we now present a result of Novak
\cite{novak_jucys-murphy_2010} expressing the Weingarten function in terms of
walks on the Cayley graph of $\Sym_{n}$ generated by the transpositions.

\begin{definition}\label{def:value_permutation}
  The \textbf{value} of a transposition
  $\tau = \cycle{i, j}\in\Sym(I)$, where $I$ is a finite subset of
  $\N^{*}$, is $\val(\tau) = \max \{i, j\}$.
\end{definition}

\begin{definition}\label{def:monotone-walk}
  Let $\rho$ and $\sigma$ be in $\Sym(I)$, with $I$ a finite subset of
  $\N^{*}$.

  A \textbf{(weakly) monotone walk with $r$ steps on $\Sym(I)$} from $\rho$ to
  $\sigma$ is a tuple $(\tau_{1}, \ldots, \tau_{r})$ of transpositions of
  $\Sym(I)$ such that
  \begin{itemize}
    \item $\tau_{r}\cdots\tau_{1}\rho = \sigma$, and
    \item $\val(\tau_{1}) \leq \cdots \leq \val(\tau_{r})$.
  \end{itemize}
  We denote the set of such walks by $\mwset^{r}(\rho, \sigma)$, and we define
  $\mwalks^{r}(\rho, \sigma)$ as the cardinality of the set
  $\mwset^{r}(\rho, \sigma)$.
\end{definition}
In this Definition, we use the arrow notation $\mwalks$ and $\mwset$
to emphasize the monotonicity property, as in
\cite{goulden_monotone_2014}.

\begin{prop}[{\cite[Theorem 3.1]{novak_jucys-murphy_2010}}]\label{prop:dvpt_weingarten}
  Let $\pi \in \Sym_{q}$ with $N \geq q$.
  We have
  \begin{equation*}
    \begin{split}
      \Weingarten_{N}(\pi) = \sum_{r \geq 0}\frac{(-1)^{r}}{N^{r+q}}\mwalks^{r}(\Id, \pi),
    \end{split}
  \end{equation*}
  and the series is absolutely convergent.
\end{prop}

Propositions \ref{prop:dvpt_weingarten} and \ref{prop:comput_moment} imply the following result (recall notation
from Definition \ref{def:L_prime}).
\begin{corol}\label{corol:expansion_moment}
  Let $N \geq 1$ be an integer, $\bm{P} = (P_{1}, \ldots, P_{l}) \in \wordsA^{l}$
  with $m=\deg \bm{P}/2 \leq N$. The moments admits the expansion
  \begin{equation*}
    \begin{split}
      \alpha^{N}_{0, l}(P_{1}, \ldots, P_{l}) = \sum_{r \geq 0}\frac{(-1)^{r}}{N^{r+m}}\sum_{\pi\in\Sym^{(\epsilon_{P})}_{2m}}\Tr_{\gamma_{\bm{P}}\pi^{-1}}(\bm{M}_{\bm{P}})\mwalks^{r}(\Id, \pi^{(\epsilon_{P})}).
    \end{split}
  \end{equation*}
  Moreover, the series is absolutely convergent, uniformly on
  $\bm{M_{\bm{P}}}$.
\end{corol}
Notice that if $\deg{P}$ is odd, there are a different number of
occurences of $u$ and of $u^{*}$, and such moments are $0$.
\begin{proof}
  Starting from the expression for the moment of Proposition
  \ref{prop:comput_moment}, we use the expansion for the Weingarten
  function of Proposition \ref{prop:dvpt_weingarten}. Notice that this
  second result can only be used if $m=\deg \bm{P}/2 \leq N$, where
  $m$ is the total number of letter $u$ in the monomials
  $P_{1}, \ldots, P_{l}$. One of the sums is finite, we can exchange
  the sums and get the wanted expression.

  Finally, as the matrices $(A_{i})$ have their operator norm bounded
  by 1, we can crudely bound the trace by
  \begin{equation*}
    |\Tr_{\gamma_{\bm{P}}\pi^{-1}}(\bm{M}_{\bm{P}})| \leq N^{c(\gamma_{\bm{P}}\pi^{-1})} \leq N^{2m}\,.
  \end{equation*}
  This implies that the convergence is uniform in $\bm{M}_{\bm{P}}$.
\end{proof}

\section{Maps and maps of unitary type}\label{sec:maps}
In this section, we introduce combinatorial objects, the so-called
maps of unitary type, that will be convenient to express the moments
$\alpha_{0, l}^{N}$, and then the cumulants $\W^{N}_{0, l}$. These maps are
particular cases of the maps appearing in the Gaussian case.

\subsection{Maps}\label{ssec:maps}
First, we give a few definitions regarding maps. See
\cite{gamkrelidze_graphs_2004} for more details on maps.

\begin{definition}\label{def:embedded-graph}
  An \textbf{embedded graph} is a pair $(\Gamma, \surf)$, where
  $\surf$ is a compact topological surface and $\Gamma$ is a graph
  (with possibly loops and multiple edges) embedded in $\surf$, so
  that we write $\Gamma \subset \surf$, such that
  \begin{itemize}
    \item the vertices of $\Gamma$ are distinct points on the surface $\surf$,
    \item the edges of $\Gamma$ are simple paths on $\surf$ that can
          intersect only at their endpoints,
    \item the complement $\surf \setminus \Gamma$ of the graph is a
          disjoint union of simply connected open sets. Each of these
          connected components is called a \textbf{face}.
  \end{itemize}
  The genus of an embedded graph is the genus of the surface $\surf$.

  An embedded graph will be said to be oriented if $\Gamma$ is an
  oriented graph.
\end{definition}
We shall sometimes refer to $\Gamma$ and $\surf$ as the underlying graph
and surface of an embedded graph. Notice that the genus of the surface
is well-defined: it follows from the fact that the faces are
homeomorphic to disks, see \cite[Section
1.3.2.2.]{gamkrelidze_graphs_2004}.

\begin{remark}
  In this article, the embedded graphs are in general disconnected. We
  will specify it when the maps we consider are connected.
\end{remark}

\begin{definition}
  Two (oriented or unoriented) embedded graphs
  $(\Gamma_{1}, \surf_{1})$ and $(\Gamma_{2}, \surf_{2})$ are said to
  be isomorphic if there is a orientation-preserving homeomorphism
  $h \colon \surf_{1} \to \surf_{2}$ such that $h|_{\Gamma}$ is an
  isomorphism of (oriented or unoriented) graphs.
\end{definition}

\begin{definition}
  A \textbf{map} (or oriented map) is an equivalence class of embedded graphs
  (or oriented embedded graphs) up to isomorphism.
\end{definition}
As the genus of a surface is a topological invariant, the genus of a map is the
genus of any of its representative.

\begin{definition}
  A face will be said to be \textbf{incident} to a vertex or an edge if the vertex or the edge belongs to the boundary of the face.
\end{definition}

It will be convenient to regard each edge of a map as being made of
two half-edges. As a part of an embedded graph they can be described
as follows. Each edge $e = \{v, v'\}$ (with possibly $v= v'$) can be
parametrized by $\gamma_{e}\colon [0, 1] \to \surf$, with $\gamma_{e}(0) = v$ and
$\gamma_{e}(1) = v'$. The two half-edges that compose $e$ are
$h = \gamma_{e}([0, 1/2])$ and $ h' = \gamma_{e}([1/2, 1])$. We will say that
$h$ (respectively $h'$) is connected to $v$ (resp. $v'$). As we will
be concerned only with combinatorial data, the choice of
parametrization $\gamma_{e}$ will be unimportant. When going from the
vertex of the half-edge to the other end of the half-edge (connected
to another half-edge), we can distinguish a left side and a right side
(see Figure \ref{fig:half-edge-left}). Notice that the left and right side are defined
relative to the position fo the incident vertex, and does not depend
on the eventual orientation.

\begin{figure}[htbp]
  \centering
  \includegraphics[width=0.5\textwidth]{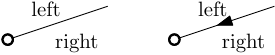}
  \caption{\label{fig:half-edge-left}The left and right side of a half-edge, and of an ingoing half-edge.}
\end{figure}

Defining a left side of an half edge is crucial for the labelling
procedure. Indeed, we regard each half-edge as being incident to the
face on its left. A face will thus be described by the labels of the
half-edges incident to it.

We label the half-edges of a map $\carte$ from $1$ to $2m$, where $m$
is the number of edges of $\carte$. By convention, we write each label
at the left of its half-edge. See Figure \ref{fig:usual_map}. In an oriented map with
labelled half-edges, the edges can be represented as an ordered pair
of two half-edges. The first half-edge is connected to the first
vertex of the edge and is said to be \textbf{outgoing}. The second half-edge is
connected to the second vertex of the edge and is said to be
\textbf{ingoing}.

\begin{figure}[htbp]
  \centering
  \includegraphics[width=0.3\textwidth]{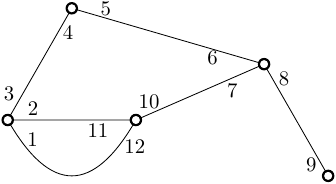}
  \caption{\label{fig:usual_map}A map with labelled half-edges.}
\end{figure}

These labels allow us to define three permutations that encode the
labelled map, see \cite[Section 1.3.3.]{gamkrelidze_graphs_2004}.
\begin{definition}\label{def:usual_maps_permutational_model}
  Let $\carte$ be a map with $2m \geq 2$ labelled half-edges. We define the three
  permutations $\sigma_{\carte}, \alpha_{\carte}, \varphi_{\carte}\in\Sym_{2m}$ as
  follows.
  \begin{itemize}
    \item Let $i \in [2m]$. The half-edge labelled by $i$ is attached to a vertex
          $v_{i}$. Starting from the half-edge $i$ and turning in the
          clockwise direction around $v_{i}$, the next half-edge we
          encounter is labelled $j$ (possibly $i = j$). We set
          $\sigma_{\carte}(i) = j$.
    \item Let $i \in [2m]$. The half-edge labelled by $i$ is attached to another
          half-edge labelled $j$. We set $\alpha_{\carte}(i) = j$.
    \item Let $i \in [2m]$. The half-edge labelled $i$ has a face $f_{i}$ to its
          left. Starting from the half-edge $i$, we turn in the counterclockwise
          direction around the face $f_{i}$. The next half-edge we encounter
          with $f_{i}$ to its left is labelled $j$. We set
          $\varphi_{\carte}(i) = j$.
  \end{itemize}
  The three permutations $\sigma_{\carte}, \alpha_{\carte}, \varphi_{\carte}$
  constitute the permutational model of $\carte$. If the map has no
  edges then we do not define any permutational data.
\end{definition}
The permutation $\sigma_{\carte}$ describes how the half-edges are arranged
around a vertex (we call this data ``the local structure of the
map''), and $\alpha_{\carte}$ describes how to attach them. The permutation
$\alpha_{\carte}$ only depends on the underlying graph of the map. Notice
that $\alpha_{\carte}$ belongs to the set of involutions without fixed
points
\begin{equation*}
  \begin{split}
    \Inv_{2m} = \{\alpha\in\Sym_{2m}\colon \alpha^{2} = \Id, \forall i \in [2m], \alpha(i) \neq i\}.
  \end{split}
\end{equation*}
Notice that we chose different conventions than in
\cite{gamkrelidze_graphs_2004}, resulting in our permutation $\sigma$
being the inverse of theirs.

\begin{ex}
  The map $\carte$ of Figure \ref{fig:usual_map} is described by
  \begin{equation*}
    \begin{split}
      \sigma_{\carte}&= \cycle{1, 3, 2}\cycle{4, 5}\cycle{6, 8, 7}\cycle{9}\cycle{10, 12, 11}\\
      \alpha_{\carte}&= \cycle{1, 12}\cycle{2, 11}\cycle{3, 4}\cycle{5, 6}\cycle{7, 10}\cycle{8, 9}\\
      \varphi_{\carte}&= \cycle{1, 11}\cycle{2, 10, 6, 4}\cycle{3, 5, 8, 9, 7, 12}.
    \end{split}
  \end{equation*}
\end{ex}

For an oriented map, we must also describe the orientation of each half-edge.
\begin{definition}
  Let $\carte$ be an oriented map with $2m$ labelled half-edges. We
  define the function $\bm{\epsilon}_{\carte} \in \{\pm 1\}^{[2m]}$ as follows. For
  all $i\in [2m]$, we set $\epsilon(i) = +1$ if the half-edge labelled $i$ is
  outgoing and $\epsilon(i) = -1$ if the half-edge labelled $i$ is ingoing.

  Such an $\bm{\epsilon}$ belongs to the set
  $\mathcal{E}_{2m} = \{\bm{\epsilon}\in\{\pm 1\}^{2m}\colon \sum_{i=1}^{2m}\epsilon(i) = 0\}$.
\end{definition}
In the case of an oriented map, $\alpha$ is in the set $\Inv_{2m}^{(\epsilon)}$ of
the permutations of $\Inv_{2m}$ such that for all $i \in [2m]$,
$\epsilon(\alpha(i)) = -\epsilon(i)$.

\begin{lemma}{{\cite[Proposition 1.3.16]{gamkrelidze_graphs_2004}}}\label{lem:psi_map}
  Let $\carte$ be a map with labelled half-edges. We have
  \begin{equation*}
    \begin{split}
      \varphi\alpha = \sigma.
    \end{split}
  \end{equation*}
\end{lemma}

Conversely, we can reconstruct a map from two permutations
$\sigma \in \Sym_{2m}, \alpha \in \Inv_{2m}$. The following theorem is
essentially a restatement of a result obtained in
\cite{edmonds_combinatorial_1960-1}.
\begin{theorem}\label{thm:usual_map_permutational_model}
  Let $m\geq 1$, $\sigma \in \Sym_{2m}$ and $\mathfrak{C}(m, \sigma)$ be the
  set of maps with labelled half-edges $\carte$ such that
  $\sigma_{\carte} = \sigma$. Then, the mapping
  \begin{equation*}
    \begin{split}
      \mathfrak{C}(m, \sigma) &\to \Inv_{2m}\\
      \carte &\mapsto \alpha_{\carte},
    \end{split}
  \end{equation*}
  is a bijection.
\end{theorem}
This theorem shows that once the local structure of the map (and a
labelling of the half-edges) is fixed, the map only depends on the
underlying graph. We have the corresponding result for oriented maps.
\begin{theorem}\label{thm:perm-model-oriented-map}
  Let $m\geq 1$, $\sigma \in \Sym_{2m}, \bm{\epsilon} \in \mathcal{E}_{2m}$ and
  $\mathfrak{C}(m, \bm{\epsilon}, \sigma)$ be the set of oriented maps
  with $2m$ labelled half-edge $\carte$ such that $\sigma_{\carte} = \sigma$ and
  $\bm{\epsilon}_{\carte} = \bm{\epsilon}$. Then,
  \begin{equation*}
    \begin{split}
      \mathfrak{C}(m, \epsilon, \sigma) &\to \Inv^{(\bm{\epsilon})}_{2m}\\
      \carte &\mapsto \alpha_{\carte},
    \end{split}
  \end{equation*}
  is a bijection.
\end{theorem}

\subsection{Maps of unitary type}\label{sec:maps_unitary}
We have just seen how to describe a map with permutations. We now define a
particular type of map, which we call map of unitary type, whose edge structure
is described by a permutation $\pi \in \Sym_{2m}^{(\epsilon)}$ for some
$\epsilon$ and $m \geq 1$ and a monotone walk
$(\tau_{1}, \ldots, \tau_{r}) \in \mwset^{r}(\Id, \pi^{(\epsilon)})$.

\begin{definition}\label{def:alternated_vertex}
  A vertex in an oriented map will be said to be \textbf{alternated} if when
  going around this vertex the half-edges connected to it are alternatively
  ingoing and outgoing.
\end{definition} 

\begin{definition}\label{def:maps_unitary_type}
  Let $I$ be a finite subset of $\N^{*}$ and $r \in \N$. A \textbf{map of
    unitary type} with labels in $I$ with $r$ black vertices is an oriented map
  with vertices colored in white or black such that
  \begin{enumerate}
    \item\label{it:unit-bv} there are $r$ black vertices, which are alternated of degree 4 and
          numbered from $1$ to $r$;
    \item\label{it:unit-wv} there are $|I|$ half-edges that are connected to white vertices. We
          call these half-edges \textbf{white half-edges}. Each element of $I$
          labels exactly one white half-edge;
    \item\label{it:unit-order} if an oriented edge connects the black vertex numbered $k$
          to the black vertex numbered $l$, with the orientation from
          $k$ to $l$, then $k < l$.
  \end{enumerate}
\end{definition}
See Figure \ref{fig:unitary-map} for an example.

\begin{remark}
  There is a correspondence between a tuple
  $\bm{P} = (P_{1}, \ldots, P_{l})$ of monomials and a family of maps of
  unitary type. The number of white vertices is $l$, each of them
  corresponds to a monomial. The white outgoing half-edges correspond
  to occurences of $u$, the white ingoing half-edges correspond to
  occurences of $u^{*}$. The black vertices correspond to steps in a
  walk as defined in Definition \ref{def:monotone-walk}. Note however that the
  monotonicity condition of the walk correspond to the increasing
  condition defined in Definition \ref{def:monotone-unitary-map}. This link will be described in
  more details in Section \ref{sec:expr-moments}.
\end{remark}

\begin{remark}
  The map has oriented edges so there are as many ingoing as outgoing
  half-edges, of any color. The black vertices are alternated and of
  edgree $4$ so there are as many ingoing black half-edges as black
  outgoing half-edges. Thus, there are as many white ingoing
  half-edges as white outgoing half-edges.
\end{remark}

\begin{remark}\label{rem:face-incident_white}
  Notice that condition 3. in Definition \ref{def:maps_unitary_type}
  implies that each face is incident to at least one white vertex.
  Indeed, if it were not the case, there would be a face incident to
  only black vertices, numbered $n_{1} < n_{2} < \ldots < n_{k}$, with
  $n_{k} < n_{1}$, a contradiction.
\end{remark}

\begin{figure}[ht]
  \centering
  \includegraphics[width=0.5\textwidth]{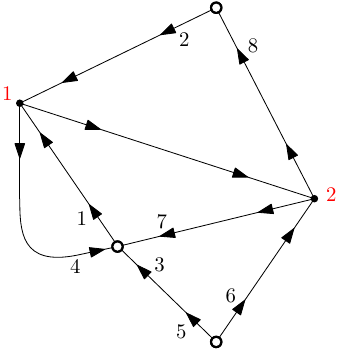}
  \caption{\label{fig:unitary-map} A unitary type map. The numbers in red (1 near the
    black vertex the left, 2 near the black vertex on the right) are
    the numbers of the black vertices, the labels in black are the
    labels of the white half-edges.}
\end{figure}

\begin{remark}
  The maps of unitary type are very similar to the maps introduced in
  \cite{collins_asymptotics_2009} to describe the leading term in the
  asymptotics of the cumulants when $N \to \infty$. In fact, the two
  kinds of maps in genus 0 are related by a surgery that transforms
  black vertices of unitary maps into ``dotted edges'' of the maps
  from \cite{collins_asymptotics_2009}. Here, we consider the
  non-planar cases as well.
\end{remark}

We denote by $w_{k}(\carte)$ the white vertex in the unitary type map $\carte$
connected to the half-edge labelled $k$. We will omit the notation $\carte$ if
there is no ambiguity.

Notice that in a map of unitary type, the half-edges connected to
black vertices are not labelled. We now explain how to label them.
Consider, in a map of unitary type, an unlabelled half-edge which we
denote by $h$. This half-edge has a face $f$ to its left (see Figure
\ref{fig:half-edge-left}). Starting from $h$, we turn around the face
in the clockwise direction until we encounter a labeled half-edge
connected to a white vertex, which is labelled by $i$. We assign to
$h$ the label $i$. See Figures \ref{fig:unitary-map} and
\ref{fig:prop_label}.

\begin{figure}[htbp]
  \centering
    \includegraphics[width=0.5\textwidth]{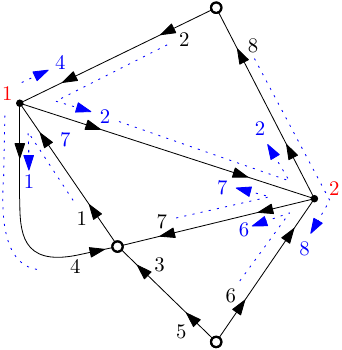}
    \caption{\label{fig:prop_label}Procedure to assign labels to half-edges}
    The newly labelled half-edges are in blue (and follow the dotted
    arrows).
\end{figure}

Notice that by Remark \ref{rem:face-incident_white}, all faces are incident to at least one white vertex,
so all unlabelled half-edges can be labelled by this procedure, in a unique way.

The following Lemma will be used to prove Lemma \ref{lem:pi_sym_eps}.
\begin{lemma}\label{lem:label_he_in-out}
  Let $h$ be a half-edge labelled by $i$. There exists a unique white
  half-edge $h'$ labelled by $i$. If $h$ is ingoing then $h'$ is
  ingoing. If $h$ is outgoing then $h'$ is outgoing.
\end{lemma}
\begin{proof}
  Consider an ingoing half-edge $h$. The existence and uniqueness of
  $h'$ is a consequence of the definition. If $h$ is a white
  half-edge, the statement is obvious. If not, then consider the face $f$
  to its left. Starting from $h$ we turn around $f$ in the clockwise
  direction until we reach a white vertex $w$. All the vertices we
  encounter before $w$ are black. The black vertices are alternated so
  all the half-edges such that $f$ is at their left are ingoing as
  well, and so is the white half-edge $h'$ that we reach, whose label
  is the same as the label of $h$. We proceed similarly for outgoing half-edges.
\end{proof}

The labels for the edges allow us to define the notion of value of a black
vertex.

\begin{definition}
  Consider a black vertex $b$. Let $i$ and $j$ be the labels of the two outgoing
  half-edges at $b$. The \textbf{value} of the black vertex $b$ is $\val(b) = \max(i, j)$.
\end{definition}

\begin{definition}\label{def:monotone-unitary-map}
  A map of unitary type with $r$ black vertices $b_{1}, \ldots, b_{r}$
  numbered respectively $1, \ldots, r$ is \textbf{nondecreasing} if
  \begin{equation*}
    \begin{split}
      \val(b_{1}) \leq \val(b_{2}) \leq \cdots \leq \val(b_{r}).
    \end{split}
  \end{equation*}
\end{definition}

\begin{ex}
  Figure \ref{fig:prop_label} displays an example. The labels of the
  black vertices are in red. The values of the black vertices $s_{1}$
  and $s_{2}$ are $\val(s_{1}) = 2, \val(s_{2}) = 6$.
\end{ex}

\subsection{Permutational model}
Similarly as in Section \ref{ssec:maps}, we define a permutational model for the
maps of unitary type.

\begin{definition}\label{def:permutation-unitary}
  Let $I \subset \N^{*}$ be finite and $r\in\N$. Let $\carte$ be a map of
  unitary type with labels in $I \neq \emptyset$ and $r$ black
  vertices.

  We define $\bm{\epsilon}_{\carte} = (\epsilon(i), i\in I)$ as follows. If the
  white half-edge labelled $i \in I$ is outgoing, we set $\epsilon(i) = +1$,
  else we set $\epsilon(i) = -1$.

  We define $\gamma_{\carte}, \pi_{\carte}, \phi_{\carte} \in \Sym(I)$ and
  $\tau_{\carte} = (\tau_{1}, \ldots, \tau_{r}) \in \Sym(\epsilon_{\carte}^{-1}(+1))^{r}$ as
  follows.
  \begin{itemize}
    \item Let $i \in I$. The white half-edge $h_{i}$, labelled $i$, is
          connected to a white vertex $w_{i}$. Starting from $h_{i}$,
          we turn in the clockwise direction around $w_{i}$. Let $j$
          be the label of the next half-edge connected to $w_{i}$. We
          set $\gamma_{\carte}(i) = j$.
    \item Let $i \in I$. The white half-edge $h_{i}$ labelled $i$ is connected
          to another half-edge $h_{j}$, which is labelled by $j$. We set
          $\pi_{\carte}(i) = j$.
    \item Let $i \in I$. The white half-edge labelled $i$ has a face $f_{i}$ to
          its left. Starting from the half-edge $i$, we turn in the
          counterclockwise direction around the face $f_{i}$. The next white
          half-edge with $f_{i}$ on its left we encounter is labelled $j$. We
          set $\phi_{\carte}(i) = j$.
    \item Let $b_{l}$ be the black vertex numbered $l$. The outgoing half-edges
          that are connected to it are labelled by $i$ and $j$. We set
          $\tau_{l} = \cycle{i, j}$.
  \end{itemize}
\end{definition}
The permutations $\gamma_{\carte}, \pi_{\carte}, \phi_{\carte}$ are the
counterparts for maps of unitary type of the permutations
$\sigma_{\carte}, \alpha_{\carte}, \varphi_{\carte}$ defined in Definition
\ref{def:usual_maps_permutational_model}.

\begin{ex}
  For the map in Figure \ref{fig:prop_label}, we have $r = 2$ and
  \begin{equation*}
    \begin{split}
      \gamma_{\carte}&=\cycle{1, 7, 3, 4}\cycle{5, 6}\cycle{2, 8},\\
      \epsilon_{\carte} &= (+1, +1, -1, -1, +1, +1, -1, -1),\\
      \tau_{1} &= \cycle{1, 2}, ~\tau_{2} = \cycle{2, 6},\\
      \pi_{\carte} &= \cycle{1, 7, 6, 8, 2, 4}\cycle{3, 5},\\
      \phi_{\carte} &=\cycle{1}\cycle{2}\cycle{3, 6}\cycle{4, 8, 5}\cycle{7}.
    \end{split}
  \end{equation*}
\end{ex}

\begin{lemma}\label{lem:pi_sym_eps}
  The permutation $\pi_{\carte}$ belongs to $\Sym^{(\epsilon_{\carte})}(I)$,
  defined in Definition \ref{def:permutations-epsilon}.
\end{lemma}
\begin{proof}
  An edge consists of an outgoing half-edge $h$ attached to an ingoing half-edge
  $h'$. Assume that $h$ is white. Let $i$ be the label of $h$ and $j$ be the
  label of $h'$. We have $\pi(i) = j$. By Lemma \ref{lem:label_he_in-out}, $j$
  is the label of a white ingoing half-edge. Thus, $\epsilon(i) = -1$ and
  $\epsilon(j) = +1$. We proceed similarly if $h'$ is white.
\end{proof}

We have the following counterpart of Lemma \ref{lem:psi_map}
\begin{lemma}\label{lem:phi_map}
  For a unitary type map $\carte$, we have
  $\gamma_{\carte}\pi_{\carte}^{-1} = \phi_{\carte}$.
\end{lemma}
\begin{proof}
  Let $i \in I$ be the label of a white outgoing half-edge, and $f$
  the face at the left of the half-edge. Starting from the half-edge
  labelled $i$, we follow the boundary of the face until we encounter
  a white vertex. The last half-edge we traversed, which was ingoing,
  was labelled by $j$. This half-edge is connected to a outgoing
  half-edge labelled $i$. By definition, we thus have
  $\pi_{\carte}(j)=i$. The next labelled half-edge when going around
  $f$ in the counterclockwise order is the half-edge following the
  half-edge $j$ when turning in the clockwise direction around the
  white vertex. This next half-edge is thus labelled
  $\gamma_{\carte}(j) = \gamma_{\carte}\pi_{\carte}^{-1}(i)$,
  see Figure \ref{fig:lemme_phi}.

  \begin{figure}[htbp]
  \centering
    \includegraphics[width=0.4\textwidth]{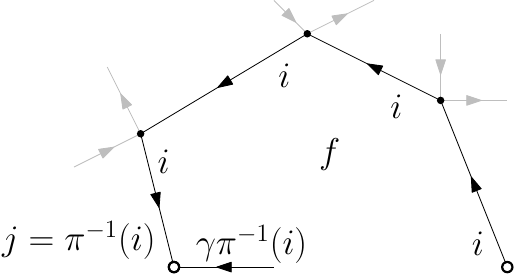}
    \caption{\label{fig:lemme_phi}Proof of Lemma \ref{lem:phi_map}.}
  \end{figure}

  The proof is identical if $i$ if the label of an ingoing half-edge.
\end{proof}

\begin{prop}\label{prop:tau_carte}
  Let $I$ be a finite subset of $\N^{*}$, $r \in \N^{*}$, $\gamma \in \Sym(I)$,
  and $\epsilon \in \{\pm 1\}^{I}$. Let $\carte$ be a unitary type map with set of
  labels $I$ and with $r$ black vertices such that $\gamma_{\carte} = \gamma$
  and $\epsilon_{\carte} = \epsilon$, and let $\tau_{\carte} = (\tau_{1}, \ldots, \tau_{r})$.
  Then, $\tau_{r}\cdots\tau_{1} = \pi_{\carte}^{(\epsilon)}$.
\end{prop}
\begin{proof}
  Let $k \in I$ be the label of a white outgoing half-edge connected to a vertex
  $w_{k} = u_{0}$. Let $f$ be the face at its right. We construct a path
  starting from the half-edge labelled $k$ as follows, see also Figure
  \ref{fig:preuve_tau_carte}. Consider the edge $e_{1} = (u_{0}, u_{1})$ of
  which the half-edge labelled $k$ is part. If $u_{1}$ is white then for all
  $1 \leq j \leq r, \tau_{j}(k)=k = \pi_{\carte}^{(\epsilon)}(k)$.

  If $u_{1}$ is black, we can find vertices $u_{2}, u_{3}, \ldots, u_{p+1}$ such
  that $u_{2}, \ldots, u_{p}$ are black and $u_{p+1}$ is white, and
  $(u_{j}, u_{j+1})$ follows $(u_{j-1}, u_{j})$ when going around the vertex
  $u_{j}$ in the counterclockwise order. Notice that these edges are all part of
  the boundary of $f$.

  Let $n_{1}, n_{2}, \ldots, n_{p}$ be the labels of the black vertices
  $u_{1}, \ldots, u_{p}$, and $k_{j}, 1 \leq j \leq p+1$ be the labels of the
  outgoing half-edges edges (connected to $u_{j-1}$) in
  $(u_{j-1}, u_{j})$. Notice that $1 \leq n_{1}, \ldots, n_{p} \leq r$ as
  black vertices have labels in $[r]$. By construction, we have
  $\tau_{n_{j}}(k_{j-1}) = k_{j}$.

  \begin{figure}[htbp]
    \centering
    \includegraphics[width=0.4\textwidth]{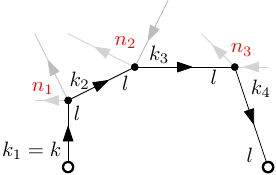}
    \caption{\label{fig:preuve_tau_carte}Chain of edges around the face $f$.}
  \end{figure}

  We have $\tau_{n_{p}}\tau_{n_{p-1}}\cdots\tau_{n_{1}}(k) = k_{p}$, so the labels
  of the black ingoing vertices incident to $f$ are all equal to
  $l = \pi_{\carte}(k)$, by construction of $\pi_{\carte}$. We have
  $\pi_{\carte}(k) = l = \pi_{\carte}^{-1}(k_{p})$. Thus
  $\tau_{n_{p}}\tau_{n_{p-1}}\cdots\tau_{n_{1}}(k) = \pi^{(\epsilon)}(k)$.

  Assume now that
  $\tau_{r}\cdots \tau_{1}(k) \neq \tau_{n_{p}}\tau_{n_{p-1}}\cdots\tau_{n_{1}}(k)$. Let $j$ be
  the minimal index such that there exists $p'$ satisfying
  $n_{p'} \leq j < n_{p'+1}$ (with the convention $n_{p+1} = r+1$) and
  $\tau_{j}\cdots \tau_{1}(k)\neq \tau_{n_{p'}}\cdots\tau_{n_{1}}(k)$. The index $j$ is minimal
  so $j > n_{p'}$ (else we would have a contradiction as
  $\tau_{j-1}\cdots \tau_{1}(k)= \tau_{n_{p'-1}}\cdots\tau_{n_{1}}(k)$). We have
  $k_{p'} = \tau_{j-1}\cdots \tau_{1}(k) = \tau_{n_{p'}}\cdots\tau_{n_{1}}(k)$. By
  construction, all the half-edges labelled by $k_{p'}$ are on the
  boundary of a same face $f'$, and they follow each other. We have
  just seen that there is such an half-edge in the edge between
  $u_{p'}$ and $u_{p'+1}$. The fact that $\tau_{j}(k_{p'}) \neq k_{p'}$
  implies that there is an half-edge labelled $k_{p'}$ that is
  connected to the $j$-th black vertex. However, this edge must be
  before (when going around the face $f'$) or after the edge
  $(u_{p'}, u_{p'+1})$ in the boundary of $f'$. This contradicts the
  fact that if there is an edge going from a black vertex $i$ to a
  black vertex labelled $j$ we have $i < j$, as
  $n_{p'} < j < n_{p'+1}$.
\end{proof}

\begin{definition}
  We denote by $\mathfrak{C}^{r}(I, \epsilon, \gamma)$ the set of nondecreasing
  unitary type maps $\carte$ with set of labels $I$ and with $r$ black vertices
  such that $\gamma_{\carte} = \gamma$ and $\epsilon_{\carte} = \epsilon$.

  Similarly, we denote by $\mathfrak{C}(g, I, \epsilon, \gamma)$ the set of
  nondecreasing unitary type maps $\carte$ with set of labels $I$ and
  with genus $g$ such that $\gamma_{\carte} = \gamma$ and $\epsilon_{\carte} = \epsilon$.
\end{definition}

\begin{theorem}\label{thm:permutational_model}
  Let $I$ be a finite subset of the positive integers, $r\in \N^{*}$,
  $\epsilon \in \{\pm 1\}^{I}$ and $\gamma \in \Sym(I)$.

  The mapping
  \begin{equation*}
    \begin{split}
      \mathfrak{C}^{r}(I, \epsilon, \gamma) &\to \bigcup_{\pi\in\Sym^{(\epsilon)}(I)}\{\pi\}\times\mwset^{r}(\Id, \pi^{(\epsilon)})\\
      \carte &\mapsto (\pi_{\carte}, \tau_{\carte})
    \end{split}
  \end{equation*}
  is a bijection.
\end{theorem}
\begin{proof}
  Lemma \ref{lem:pi_sym_eps} and Proposition \ref{prop:tau_carte} show that this
  map has values in
  $\bigcup_{\pi\in\Sym^{(\epsilon)}(I)}\{\pi\}\times\mwset^{r}(\Id, \pi^{(\epsilon)})$.

  We now construct an inverse mapping. To do so, we explicitely construct
  a map corresponding to permutations $\pi$ and
  $\bm{\tau} = (\tau_{1}, \ldots, \tau_{r})$. By Theorem
  \ref{thm:usual_map_permutational_model}, it suffices to construct
  from $\pi$ and $\bm{\tau}$ the incidence relation of the underlying
  graph.

  To this end, we introduce the set whose elements represent the
  half-edges
  $\tilde{I} = \{h_{i}\colon i \in I\} \cup \bigcup_{j=1}^{r}\{h_{j, 1}, h_{j, 2}, h_{j, 3}, h_{j, 4}\}$.
  We can split this set into the set of ingoing and outgoing edges
  $\tilde{I} = \tilde{I}_{\text{in}}\cup \tilde{I}_{\text{out}}$. We have
  $\tilde{I}_{\text{out}} = \{h_{i}\colon i\in I, \epsilon(i) = +1\}\cup\bigcup_{j=1}^{r}\{h_{j, 2}, h_{j, 4}\}$.
  The elements $h_{j, k}$ represent the half-edges of the black
  vertices of the map we are going to construct, and the elements
  $h_{i}$ represent the half-edges of the white vertices. We are going
  to define a labelling function $L\colon \tilde{I}\to I$. We set for all
  $i \in I$, $L(h_{i}) = i$. The function $L$ is constructed by
  induction. At the initial step, it is only defined for the white
  half-edges. We then define it for the black half-edges of the black
  vertix $i$ at step $i$.

  To construct a map, we use Theorem
  \ref{thm:usual_map_permutational_model}. We define two permutations
  $\sigma, \alpha \in \Sym(\tilde{I})$ as follows. We define
  $\tilde{\gamma} \in \Sym(\tilde{I})$ by
  $\tilde{\gamma}(h_{i}) = h_{\gamma(i)}$ and the identity otherwise. We
  set
  \begin{equation*}
    \begin{split}
      \sigma = \tilde{\gamma}\cycle{h_{1, 1}, h_{1, 2}, h_{1, 3}, h_{1, 4}}\cdots\cycle{h_{r, 1}, h_{r, 2}, h_{r, 3}, h_{r, 4}}\,.
    \end{split}
  \end{equation*}

  The permutation $\alpha$ is given by the following algorithm. Let
  $\pi\in\Sym^{(\epsilon)}(I)$, and
  $\tau = (\tau_{1}, \ldots, \tau_{r})\in\mwset^{r}(\Id, \pi^{(\epsilon)})$. We consider first
  the permutation $\tau_{1}=(i_{1}, j_{1})$, with $i_{1} < j_{1}$. We set
  $\alpha_{1} = \cycle{h_{i_{1}}, h_{1, 1}}\cycle{h_{j_{1},} h_{1, 3}}$. We
  set $L(h_{1, 2}) = j_{1}$ and $L(h_{1, 4}) = i_{1}$. In terms of
  maps, this procedure corresponds to connecting two edges to a same
  black vertex, see Figure \ref{fig:construction-proof}.

  \begin{figure}[htbp]
    \centering
    \includegraphics[width=0.4\textwidth]{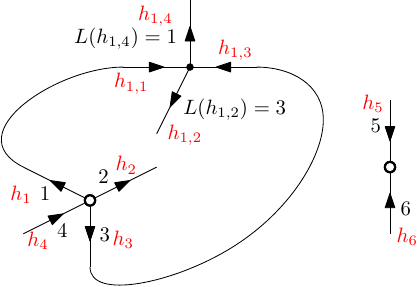}
    \caption{\label{fig:construction-proof}First step of the construction of the
      permutation $\alpha$ from one transposition $\tau_{1}=\cycle{1, 3}$,
      represented as a map.}
    \small{The name of the half-edges are in red
      ($h_{1}, \ldots, h_{6}$ near the white vertices and
      $h_{1, 1}, \ldots, h_{1, 4}$ near the black vertices), and the
      labels are in black. Here,
      $\alpha_{1} = \cycle{h_{1}, h_{1, 1}}\cycle{h_{3}, h_{1, 3}}$.}
  \end{figure}

  We proceed similarly to construct the black vertices labelled
  $2, 3, \ldots, r$ from the transpositions
  $\tau_{2}, \ldots, \tau_{r}$. At the $k$-th step, we consider the
  transposition $\tau_{k} = \cycle{i_{k}, j_{k}}$, with
  $i_{k} < j_{k}$. There is only one half-edge $h$ (respectively $h'$)
  in $\tilde{I}_{\text{out}}$ such that $L(h) = i_{k}$ and
  $\alpha_{k-1}(h) = h$ (respectively $L(h') = j_{k}$ and
  $\alpha_{k-1}(h') = h'$). We set
  $\alpha_{k} = \alpha_{k-1}\cycle{h, h_{k, 1}}\cycle{h', h_{k, 3}}$.

  Finally, we connect each remaining outgoing half-edge labelled $i$
  to the ingoing half-edge $\pi^{-1}(i)$. For all $i \in I$, there is a
  unique $h$ such that $\alpha_{r}(h) = h$ and $L(h) = i$. We set
  $\alpha_{r+1,i} = \cycle{h, h_{\pi^{-1}(i)}}$ and define
  $\alpha = \alpha_{r}\prod_{i \in I}\alpha_{r+1, i}$. We define
  $\tilde{\epsilon} \in \{\pm 1\}^{\tilde{I}}$ by $\tilde{\epsilon}(i) = +1$ if
  $i \in \tilde{I}_{\text{out}}$ and $\tilde{\epsilon}(i) = -1$ otherwise.

  Theorem \ref{thm:usual_map_permutational_model} implies that given $\sigma, \alpha$ and $\tilde{\epsilon}$, we construct
  a unique map $\tilde{\carte}$. By construction, the resulting map is
  of unitary type : the vertices attached to the half-edges $h_{i}$
  are the white vertices and the other are the black vertices. The
  black vertex attached to the half-edges $h_{j, k}$ is numbered $j$.
  The map is constructed such that $\pi_{\carte} = \pi_{\tilde{\carte}}$
  and $\tau_{\carte} = \tau_{\tilde{\carte}}$.

  Furthermore, the map is nondecreasing (recall Definition
  \ref{def:monotone-unitary-map}) as the tuple $\bm{\tau}$ is a monotone
  walk.

  We have constructed a right inverse, so the map
  $\carte \mapsto (\pi_{\carte}, \tau_{\carte})$ is surjective. We now show that
  this map is injective. We show that the incidence relation of a map
  of unitary type $\carte$ is determined by the permutations. Indeed,
  consider a map of unitary type described by $\pi$ and
  $\tau = (\tau_{1}, \ldots, \tau_{r})$, and an outgoing half-edge $h_{i}$ labelled
  $i$. There are four cases.
  \begin{itemize}
    \item If $h_{i}$ is a white half-edge such that for all $j$ we
          have $\tau_{j}(i) = i$, then $h_{i}$ is necessarily attached
          to the white half-edge labelled $\pi^{-1}(i)$.
    \item If $h_{i}$ is a white half-edge and there exists $k$, such that
          $\tau_{k}(i) \neq i$, then $h_{i}$ is necessarily connected to the
          $k'$-th black vertex, where
          $k' = \min\{k\colon \tau_{k}(i)\neq i\}$.
    \item If $h_{i}$ is a half-edge connected to the $k$-th black vertex and for
          all $l > k$ $\tau_{l}(i) = i$, then $h_{i}$ is necessarily attached to a white
          half-edge labelled $\pi^{-1}(i)$
    \item If $h_{i}$ is a half-edge connected to the $k$-th black vertex and $l$
          is the smallest integer such that $l > k$ and $\tau_{l}(i) \neq i$,
          then $h_{i}$ is necessarily attached to the $l$-black vertex.
  \end{itemize}
  Thus two maps of unitary type in $\mathfrak{C}^{r}(I, \epsilon, \gamma)$
  described by the same permutations $\pi$ and $\tau$ have necessarily the
  same edges, i.e. are identical.
\end{proof}

For a tuple of permutations,
$(\sigma_{1}, \ldots, \sigma_{k})\in\Sym(I)^{k}$, we denote the
subgroup of $\Sym(I)$ they generate by
\begin{equation*}
  \left< \sigma_{1}, \ldots, \sigma_{k} \right>\,.
\end{equation*}

We can associate to the triplet
$(\gamma_{\carte}, \pi_{\carte}, \tau_{\carte})$ the group
\begin{equation}\label{eq:group}
  \begin{split}
    G(\carte) = \langle{\gamma_{\carte}, \pi_{\carte}, \tau_{1}, \ldots, \tau_{r}}\rangle,
  \end{split}
\end{equation}
where $\tau_{\carte} = (\tau_{1}, \ldots, \tau_{r})$.

\begin{prop}\label{prop:connectedness}
  A unitary type map $\carte$ with set of labels $I$ is connected if
  and only if the group $G(\carte)$ defined by (\ref{eq:group}) acts transitively
  on $I$.
\end{prop}
\begin{proof}
  First, assume that $\carte$ is connected. Let $i, j\in I$. There is a
  path $\rho$ (made up by vertices and edges) connecting the white
  vertices $w_{i}$ and $w_{j}$. First, let us assume that $\rho$ contains
  only black vertices, except for its boundary which is made up of
  $w_{i}$ and $w_{j}$. The path encounters the black vertices
  $n_{1}, \ldots, n_{p}$, the labels on the left of the edges that
  constitute $\rho$ are $k_{1}, \ldots, k_{p+1}$. The first and last edges are
  connected to $w_{i}$ and $w_{j}$ so $k_{1} = \gamma_{\carte}^{m_{1}}(i)$ and
  $k_{p+1} = \gamma_{\carte}^{m_{2}}\pi_{\carte}^{m_{3}}(j)$ for some integers
  $m_{1}, m_{2}, m_{3}$.

  Let $1 \leq i \leq p$. If $k_{i} = k_{i+1}$, we set $\sigma_{i} = \Id$, and if
  $\tau_{n_{i}}(k_{i}) = k_{i+1}$, we set $\sigma_{i} = \tau_{n_{i}}$, see
  Figure \ref{fig:preuve_connectedness}. Those are the only two possibilities as
  the half-edges connected to a black vertex labeled $k$, with
  $\tau_{k} = \cycle{u, v}$ can only be labeled by $u$ or $v$.

  \begin{figure}[htbp]
    \centering
    \includegraphics[width=0.4\textwidth]{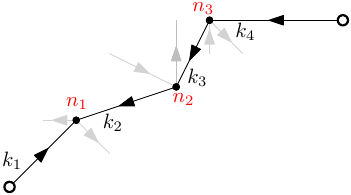}
    \caption{\label{fig:preuve_connectedness}Three situations for $\sigma_{i}$.}
    \small{We set $\sigma_{1}=\tau_{n_{1}}$, $\sigma_{2}=\Id$, and
      $\sigma_{3}=\tau_{n_{3}}$. The black vertices are labelled
      $n_{1}, n_{2}, n_{3}$ in red. In grey are the two half-edges
      that do not play a role for each black vertex.}
  \end{figure}

  Thus, we have proved that there is
  $\sigma_{\rho} = \pi_{\carte}^{-m_{3}}\gamma_{\carte}^{-m_{2}}\sigma_{p}\cdots \sigma_{1}\gamma_{\carte}^{-m_{1}}\in G(\carte)$,
  such that $\sigma_{\rho}(i) = j$.

  In general, any path connecting $w_{i}$ and $w_{j}$ can be written as the
  concatenation of paths with only black vertices in their interiors, we can
  thus construct by composition a permutation in $G(\carte)$ that sends $i$ to
  $j$. Thus $G(\carte)$ is transitive.

  Conversely, if $G(\carte)$ is transitive, for any $k, l\in I$, there
  exists $\sigma \in G(\carte)$ such that $\sigma(k)=l$. We can write
  $\sigma = \sigma_{p}\cdots \sigma_{1}$, with for all $i$, $\sigma_{i}$ is one of
  $\gamma_{\carte}, \pi_{\carte}^{-1}, \tau_{1}, \ldots, \tau_{r}$. We use this to
  construct a path connecting $v_{k}$ to $v_{l}$. For all $i$, we
  attach to $\sigma_{i}$ a path $\rho_{i}$ starting from a half-edge labelled
  $k_{i}$. We set $k_{1} = k$, and we will show that $k_{p+1} = l$.
  \begin{itemize}
    \item If $\sigma_{i} = \gamma_{\carte}$, $\rho_{i}$ is the empty path, and $k_{i+1}=\gamma_{\carte}(k_{i})$.
    \item If $\sigma_{i} = \pi_{\carte}^{-1}$, $\rho_{i}$ is the path connecting
          the half-edge $k_{i}$ to the half-edge $\pi^{-1}(k_{i})$. Such a path
          exists by the propagation of labels procedure. We set $k_{i+1} = \pi^{-1}(k_{i})$.
    \item If $\sigma_{i} = \tau_{n_{i}}$, for some $n_{i}$, and
          $\tau_{n_{i}}(k_{i}) = k_{i}$, then $\rho_{i}$ is the empty path and
          $k_{i+1}=k_{i}$.
    \item If $\sigma_{i} = \tau_{n_{i}}$, for some $n_{i}$, and
          $\tau_{n_{i}}(k_{i}) \neq k_{i}$, then we set
          $k_{i+1} = \tau_{n_{i}}(k_{i})$. Both $k_{i}$ and $k_{i+1}$ are labels
          of outgoing half-edges. We set $\rho_{i}$ to be the path that starts
          from the half-edge $k_{i}$, follows the half-edges labelled $k_{i}$
          until it reaches the black vertex $n_{i}$, and then follows the
          half-edges labelled $k_{i+1}$ until the half-edge $k_{i+1}$, and the
          vertex $w_{k_{i+1}}$.
  \end{itemize}

  We have constructed a path going from the half-edge $i$ to the half-edge
  $k_{p+1}=\sigma(i)=j$, as wanted.
\end{proof}

\subsection{Expression of the moments in terms of maps of unitary type}
\label{sec:expr-moments}
Theorem \ref{thm:permutational_model} allows us to rewrite the expression for the moments given
in Corollary \ref{corol:expansion_moment} (see Definition \ref{def:L_prime} for relevant notation).
\begin{corol}\label{corol:moments_maps}
  Let $N \geq 1$ be an integer,
  $\bm{P} = (P_{1}, \ldots, P_{l}) \in \wordsA^{l}$ be monomials with
  $m=\frac{1}{2}\deg \bm{P} \leq N$. The moments under the Haar
  measure $\mu^{N}_{0}$ (see Definition
  \ref{def:moment-cumulant-particular}) admit the following expansion
  \begin{equation*}
    \begin{split}
      \alpha^{N}_{0, l}(P_{1}, \ldots, P_{l})
      = \sum_{r \geq 0}\frac{(-1)^{r}}{N^{r+m}}\sum_{\carte \in \mathfrak{C}^{r}([2m], \bm{\epsilon}_{\bm{P}}, \gamma_{\bm{P}})}\Tr_{\phi_{\carte}}(\bm{M}_{\bm{P}}).
    \end{split}
  \end{equation*}
  Furthermore, the series is absolutely convergent.
\end{corol}

The weights $\Tr_{\phi_{\carte}}(\bm{M}_{\bm{P}})$ can be interpreted as product
of weights given by the faces of the map $\carte$, see Figure
\ref{fig:avec_matrices}.

\begin{figure}[htbp]
  \centering
  \includegraphics[width=0.4\textwidth]{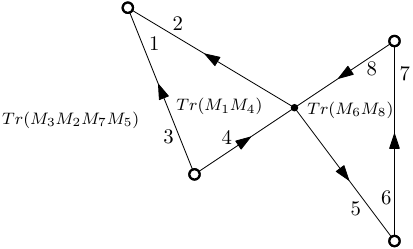}
  \caption{\label{fig:avec_matrices}A map with its weights}
  This weighted map gives (up to a sign) a contribution from the sum
  $\alpha^{(0), N}_{0, 4}(M_{1}u^{-1}M_{2}u^{-1}, M_{3}uM_{4}u, M_{5}u^{-1}M_{6}u, M_{7}u^{-1}M_{8}u)$.
\end{figure}

\begin{proof}[Proof of Corollary \ref{corol:moments_maps}]
  Recall the expression of Corollary \ref{corol:expansion_moment}:
  \begin{equation*}
    \begin{split}
      \alpha^{N}_{0, l}(P_{1}, \ldots, P_{l}) = \sum_{r \geq 0}\frac{(-1)^{r}}{N^{r+m}}\sum_{\pi\in\Sym^{(\epsilon_{\bm{P}})}_{2m}}\Tr_{\gamma_{\bm{P}}\pi^{-1}}(\bm{M}_{\bm{P}})\mwalks^{r}(\Id, \pi^{(\epsilon_{\bm{P}})}).
    \end{split}
  \end{equation*}
  By definition of $\mwalks^{r}(\Id, \pi^{(\epsilon_{\bm{P}})})$, we
  can rewrite this as
  \begin{equation*}
    \begin{split}
      \alpha^{N}_{0, l}(P_{1}, \ldots, P_{l})
      &= \sum_{r \geq 0}\frac{(-1)^{r}}{N^{r+m}}\sum_{\substack{\pi\in\Sym^{(\epsilon_{\bm{P}})}_{2m}\\(\tau_{1}, \ldots, \tau_{r})\in \mwset^{r}(\Id, \pi^{(\epsilon_{\bm{P}})})}}\Tr_{\gamma_{\bm{P}}\pi^{-1}}(\bm{M}_{\bm{P}})\\
      &= \sum_{r \geq 0}\frac{(-1)^{r}}{N^{r+m}}\sum_{\carte \in \mathfrak{C}^{r}([2m], \bm{\epsilon}_{\bm{P}},  \gamma_{\bm{P}})}\Tr_{\gamma_{\bm{P}}\pi_{\carte}^{-1}}(\bm{M}_{\bm{P}}),\\
    \end{split}
  \end{equation*}
  where we used Theorem \ref{thm:permutational_model} in the last line.

  We get the result by using Lemma \ref{lem:phi_map}, which gives
  $\gamma_{\bm{P}}\pi_{\carte}^{-1} = \phi_{\carte}$.
\end{proof}

Definition \ref{def:moment-cumulant} and Corollary
\ref{corol:moments_maps} allow us to express the cumulants in terms of
maps of unitary types. We deduce the following Lemma.
\begin{lemma}\label{lem:cumulants_V=0_first}
  Let $N \geq 1$ be an integer,
  $\bm{P} = (P_{1}, \ldots, P_{l}) \in \wordsA^{l}$ be monomials with
  $m=\frac{1}{2}\deg \bm{P}$. The cumulants admit the expansion
  \begin{equation*}
    \begin{split}
      \W^{N}_{0, l}(P_{1}, \ldots, P_{l})
      = \sum_{r \geq 0}\frac{(-1)^{r}}{N^{r+m}}\sum_{\substack{\carte \in \mathfrak{C}^{r}([2m], \bm{\epsilon}_{\bm{P}}, \gamma_{\bm{P}})\\ \carte \text{ is connected }}}\Tr_{\phi_{\carte}}(\bm{M}_{\bm{P}}).
    \end{split}
  \end{equation*}
  Furthermore, the series is absolutely convergent.
\end{lemma}
\begin{proof}
  We show the formula by induction using Corollary \ref{corol:moments_maps}.
  Notice first that when $l = 1$,
  $\alpha^{N}_{0, 1}(P_{1}) = \W^{N}_{0, 1}(P_{1})$ and the maps in
  $\mathfrak{C}^{r}([2m], \bm{\epsilon}_{\bm{P}}, \gamma_{\bm{P}})$ are
  connected.

  Then, we notice that a map can be decomposed into its connected
  components. This decomposition gives a partition of the set of
  labels of half-edges. Each block contains the labels appearing in
  one connected component. Using Definitions \ref{def:moment-cumulant}
  and Definition \ref{def:moment-cumulant-particular}, we obtain that
  \begin{equation*}
    \begin{split}
      \W^{N}_{0, l}(P_{1}, \ldots, P_{l})
      &= \alpha^{N}_{0, l}(P_{1}, \ldots, P_{l}) - \sum_{\substack{\Pi \in \partition([l])\\|\Pi| \geq 2}}\prod_{B \in \Pi}\W_{0, |B|}^{N}(P_{i}, i \in B)\\
      &= \sum_{r \geq 0}\frac{(-1)^{r}}{N^{r+m}}\sum_{\carte \in \mathfrak{C}^{r}([2m], \bm{\epsilon}_{\bm{P}}, \gamma_{\bm{P}})}\Tr_{\phi_{\carte}}(\bm{M}_{\bm{P}})\\
        &- \sum_{r \geq 0}\frac{(-1)^{r}}{N^{r+m}}\sum_{\substack{\carte \in \mathfrak{C}^{r}([2m], \bm{\epsilon}_{\bm{P}}, \gamma_{\bm{P}})\\ \carte \text{ has at least $2$ connected components }}}\Tr_{\phi_{\carte}}(\bm{M}_{\bm{P}}).\\
    \end{split}
  \end{equation*}
  Hence the result.
\end{proof}

\begin{remark}
  The formulae imply that we can express moments and cumulants with
  respect to the Haar measure as a weighted sum over maps. The maps
  are the nondecreasing maps of unitary type whose local structure
  (i.e. how the half-edges are attached to the vertices, but not how
  the half-edges are attached together) is determined by
  $\gamma_{\bm{P}}^{-1}$ and $\bm{\epsilon}_{\bm{P}}$. To each face is
  associated a weight, which is the trace of a certain word in the
  matrices of $\bm{M}_{\bm{P}}$, times a sign.
\end{remark}

\paragraph{A topological expansion for the Haar measure.}
We now rewrite Lemma \ref{lem:cumulants_V=0_first} as a sum over the
genus $g$ of the maps rather than on the number of black vertices $r$.
We will see that this gives us an expansion in powers of
$\frac{1}{N^{2}}$. We first recall Euler's formula
\begin{equation}\label{eq:Euler}
  \begin{split}
    2 - 2g(\carte) = V(\carte) - E(\carte) + F(\carte),
  \end{split}
\end{equation}
where $V(\carte), E(\carte)$ and $F(\carte)$ are the number of
vertices, edges and faces of a map $\carte$, and $g(\carte)$ is its
genus. In the case of a map of unitary type labelled by a set of $2m$
integers, and with $r$ black vertices, we have
\begin{itemize}
  \item $c(\gamma_{\carte})$ white and $r$ black vertices,
  \item $2m$ white half-edges and $4r$ half-edges out of black
        vertices, for a total of $m + 2r$ edges,
  \item $c(\phi_{\carte})$ faces (see Definition \ref{def:embedded-graph}).
\end{itemize}
Thus, we get
\begin{equation}\label{eq:Euler_for_map}
  \begin{split}
    2 - 2g(\carte) = (c(\gamma_{\carte})  + r) - (m + 2r) + c(\phi_{\carte}) = c(\gamma_{\carte}) + c(\phi_{\carte}) - m - r.
  \end{split}
\end{equation}
A change of variable in the sum of Lemma \ref{lem:cumulants_V=0_first}, and the identities
$l = c(\gamma_{\bm{P}})$ and
$\Tr_{\phi_{\carte}}(\bm{M}_{\bm{P}}) = N^{c(\phi_{\carte})}\tr_{\phi_{\carte}}(\bm{M}_{\bm{P}})$
give the following Proposition.
\begin{prop}\label{prop:cumulants_V=0_second}
  Let $N \geq 1$ be an integer,
  $\bm{P} = (P_{1}, \ldots, P_{l}) \in \wordsA^{l}$ be monomials with
  $m=\frac{1}{2}\deg \bm{P}$. The cumulants admit the expansion
  \begin{equation*}
    \begin{split}
      \W^{N}_{0, l}(P_{1}, \ldots, P_{l})
      = N^{2-l}(-1)^{m+l}\sum_{g \geq 0}\frac{1}{N^{2g}}\sum_{\substack{\carte \in \mathfrak{C}(g, [2m], \bm{\epsilon}_{\bm{P}}, \gamma_{\bm{P}})\\ \carte \text{ is connected }}}(-1)^{c(\phi_{\carte})}\tr_{\phi_{\carte}}(\bm{M}_{\bm{P}}).
    \end{split}
  \end{equation*}
  Furthermore, the series is absolutely convergent.
\end{prop}
Notice that this expansion is in terms of the normalized trace
$\tr = \frac{1}{N} \Tr$. The factors with the trace are bounded by $1$
if we assume $\|A^{N}_{i}\| \leq 1$ for all $1 \leq i \leq N$ and
$N \geq 1$.

\begin{remark}
  The sum in Proposition \ref{prop:cumulants_V=0_second} is in general
  not finite. Indeed, even for $l=1$ and $P_{1} = AUBU^{*}$, the sum
  contains terms of arbitrary genus. They appear for instance because
  of the factorization of the identity $\Id = \cycle{1, 2}^{2k}$, for
  all $k \geq 0$.
\end{remark}

\begin{definition}\label{def:term_2g_cumulant}
  Let $N \geq 1$ be an integer, $\bm{P} = (P_{1}, \ldots, P_{l}) \in \wordsA^{l}$
  be monomials with $m=\frac{1}{2}\deg \bm{P}$. The term of order $2g$ in the
  expansion of the cumulant is denoted by
  \begin{equation*}
    \begin{split}
      \M^{(g), N}_{0, l}(P_{1}, \ldots, P_{l}) = (-1)^{m+l}\sum_{\substack{\carte \in \mathfrak{C}(g, [2m], \bm{\epsilon}_{\bm{P}}, \gamma_{\bm{P}})\\ \carte \text{ is connected }}}(-1)^{c(\phi_{\carte})}\tr_{\phi_{\carte}}(\bm{M}_{\bm{P}}).
    \end{split}
  \end{equation*}
  We extend this definition to all monomials in $\algA$ by setting for
  $P_{1}, \ldots, P_{l} \in \wordsA$ and $M$ a word in
  $a_{1}, a_{1}^{*}, \ldots, a_{p}, a^{*}_{p}$,
  \begin{equation*}
    \begin{split}
      \M^{(g), N}_{0, l}(P_{1}, \ldots, P_{i-1}, P_{i}M, P_{i+1}, \ldots, P_{l}) = \M^{(g), N}_{0, l}(P_{1}, \ldots, P_{i-1}, MP_{i}, P_{i+1}, \ldots, P_{l}),
    \end{split}
  \end{equation*}
  for all $1 \leq i \leq l$.
\end{definition}
The last property is enforced so that $\M^{(g), N}_{0, l}$ has a
property of cyclicity, as does the trace.

\paragraph{Stationnary distribution of the $(A_{i}^{N})_{1\leq i \leq p}$.}

Let us consider a particular choice for the sequence of matrices
$(A_{i}^{N})_{1\leq i \leq p, N \geq 1}$. Fix a family of $p$ matrices of fixed
size $M \times M$, $(A_{i}^{M})_{1 \leq i \leq p}$, and consider the sequence of
matrices $(A^{qM}_{i})_{1 \leq i \leq p, q \geq 1}$, where $A^{qM}_{i}$ is the
block-diagonal matrix with $q$ blocks, whose blocks are $A^{M}_{i}$.
When considering the sums of maps for $N = qM$, the traces
$\tr_{\phi}(\bm{M})$ no longer depend on $q$ or $N = qM$.

In the case of zero potential ($V = 0$), by Proposition
\ref{prop:cumulants_V=0_second}, the renormalized cumulant
$\tilde{\W}^{N}_{0, 1}$ converges with limit
\begin{equation*}
  \begin{split}
    \lim_{q\to\infty} \tilde{\W}^{qM}_{0, 1}(P) = \M^{(0), M}_{0, 1}(P),
  \end{split}
\end{equation*}
for $P \in \algA$. This fact allows us to prove the following Lemma.
\begin{lemma}\label{lem:bound_01}
  Fix $N \in \N^{*}$. Assume that $\|A^{N}_{i}\| \leq 1$ for all $1 \leq i \leq p$.
  Let $P \in \wordsA$ be a monomial. We have for all choices of
  $(A^{N}_{i})_{1 \leq i \leq p}$ that
  \begin{equation*}
    \begin{split}
      |\M^{(0),N}_{0, 1}(P)| \leq 1.
    \end{split}
  \end{equation*}
\end{lemma}
\begin{proof}
  By the previous remark, we have for the choice of with the
  $(A_{i}^{N})_{1\leq i \leq p}$ block diagonal as above,
  \begin{equation*}
    \begin{split}
      |\M^{(0),M}_{0, 1}(P)| = \lim_{N\to\infty}|\tilde{\W}^{N}_{0, 1}(P)| = \lim_{N\to\infty}\E\left[\tr(P)\right] \leq 1,
    \end{split}
  \end{equation*}
  as $\|P\| \leq 1$.
\end{proof}

More generally, with the $(A_{i}^{N})_{1\leq i \leq p}$ block diagonal
as above, we have
\begin{equation*}
  \begin{split}
    \tilde{\W}^{qM}_{0, l}(\bm{P}) &= (-1)^{m+l}\sum_{g \geq 0}\frac{1}{(qM)^{2g}}\sum_{\substack{\carte \in \mathfrak{C}(g, [2m], \bm{\epsilon}_{\bm{P}}, \gamma_{\bm{P}})\\ \carte \text{ is connected }}}(-1)^{c(\phi_{\carte})}\tr_{\phi_{\carte}}(\bm{M}_{\bm{P}})\,,
  \end{split}
\end{equation*}
where $\tr_{\phi_{\carte}}(\bm{M}_{\bm{P}})$ does not depend on $q$.
This implies the following Lemma.
\begin{lemma}\label{lem:prop-M0}
  For all $N \geq 1$, $g \geq 0$, $l \geq 1$, and
  $\bm{P} \in \wordsA^{l}$, we have the following properties.
  \begin{enumerate}[(i)]
    \item (Traciality) For all $Q \in \wordsA$,
    \begin{equation*}
      \begin{split}
        \M^{(g), N}_{0, l}(P_{1}, \ldots, P_{l-1}, P_{l}Q) &= \M^{(g), N}_{0, l}(P_{1}, \ldots, P_{l-1}, QP_{l})\,. \\
      \end{split}
    \end{equation*}
    \item (Symmetry) For all permutation $\sigma \in \Sym_{l}$,
    \begin{equation*}
      \begin{split}
        \M^{(g), N}_{0, l}(P_{1}, \ldots, P_{l}) &= \M^{(g), N}_{0, l}(P_{\sigma(1)}, \ldots, P_{\sigma(l)})\,.
      \end{split}
    \end{equation*}
    \item (Simplification) We have
          \begin{equation*}
            \begin{split}
              \M^{(g), N}_{0, l}(P_{1}, \ldots, P_{l-1}, u^{*}P_{l}u) &= \M^{(g), N}_{0, l}(P_{1}, \ldots, P_{l})\,.
            \end{split}
          \end{equation*}
          \item (Conjugation) We have
          \begin{equation*}
            \begin{split}
              \M^{(g), N}_{0, l}(P_{1}^{*}, \ldots, P_{l}^{*}) &= \overline{\M^{(g), N}_{0, l}(P_{1}, \ldots, P_{l})}\,.
            \end{split}
          \end{equation*}

  \end{enumerate}
\end{lemma}
\begin{proof}
  Consider the series
  \begin{equation*}
    \begin{split}
      G(\hbar) &= (-1)^{m+l}\sum_{g \geq 0}\hbar^{2g}\sum_{\substack{\carte \in \mathfrak{C}(g, [2m], \bm{\epsilon}_{\bm{P}}, \gamma_{\bm{P}})\\ \carte \text{ is connected }}}(-1)^{c(\phi_{\carte})}\tr_{\phi_{\carte}}(\bm{M}_{\bm{P}})\,,
    \end{split}
  \end{equation*}
  where the polynomials in the tuple $\bm{M}_{\bm{P}}$ are evaluated
  at the matrices
  $A_{1}^{M}, (A_{1}^{M})^{*}, \ldots, A_{2m}^{M}, (A_{2m}^{M})^{*}$.
  Proposition \ref{prop:cumulants_V=0_second} implies that
  $G(1/qM) = \tW^{N}_{0, l}(\bm{P})$.

  Thus, as the renormalized cumulant under the Haar measure
  $\tW^{N}_{0, l} = N^{l-2}\W^{N}_{0, l}$ satisfies all four
  properties, and the set $\{1/qM\}_{q \geq 1}$ has an accumulation point, we get
  the result.
\end{proof}

\subsection{Formal topological expansion}\label{sec:formal-topological-exp}
When the potential $V$ is not zero, we expect to have an expansion of
the free energy as in Proposition \ref{prop:cumulants_V=0_second}. Let
us now consider a potential of the form
$V = \sum_{i=1}^{k}z_{i}q_{i}$, with
$\bm{z} = (z_{1}, \ldots, z_{k}) \in \C^{k}$ and
$\bm{q} = (q_{1}, \ldots, q_{k}) \in \wordsA^{k}$.

Proposition \ref{prop:cumulants_V=0_second} motivates the introduction of the
formal series
\begin{equation}
  \begin{split}
    F^{N, f}_{V}
    &= \frac{1}{N^{2}}\sum_{\bm{n} \in \N^{k}}\frac{(-\bm{z})^{\bm{n}}}{\bm{n}!}\tW_{0, \sum_{i}n_{i}}^{N}(\bm{q}_{\bm{n}})\\
    &= \sum_{g \geq 0}\frac{1}{N^{2g}}\sum_{\bm{n} \in \N^{k}}\frac{\bm{z}^{\bm{n}}}{\bm{n}!}(-1)^{\deg \bm{q}_{\bm{n}}}\sum_{\substack{\carte \in \mathfrak{C}(g, [2\deg \bm{q}_{\bm{n}}], \bm{\epsilon}_{\bm{q}_{\bm{n}}}, \gamma(\bm{q}_{\bm{n}}))\\ \carte \text{ is connected }}}(-1)^{c(\phi_{\carte})}\tr_{\phi_{\carte}}(\bm{M}_{\bm{q}_{\bm{n}}})\,,
  \end{split}
\end{equation}
where we use the notation
$\bm{q}_{\bm{n}} = (\underbrace{q_{1}, \ldots, q_{1}}_{n_{1}\text{
    times }}, \ldots, \underbrace{q_{k}, \ldots, q_{k}}_{n_{k}\text{
    times }})$ for $\bm{n} = (n_{1}, \ldots, n_{k})$, as well as
$\bm{z}^{\bm{n}} = \prod_{i=1}^{k}z_{i}^{n_{i}}$ and
$\bm{n}! = \prod_{i=1}^{k}n_{i}!$.

Similarly, we introduce formal series corresponding to the cumulants.
\begin{definition}\label{def:formal_cumulant}
  Let $N \in \N^{*}$, $\bm{P} = (P_{1}, \ldots, P_{l}) \in \wordsA^{l}$ be
  monomials with $m=\frac{1}{2}\deg \bm{P}$. The \textbf{formal cumulant} of
  $\bm{P}$ is the formal series
  \begin{equation*}
    \begin{split}
      \M^{N}_{V, l}(P_{1}, \ldots, P_{l}) &= \sum_{g\geq 0}\frac{1}{N^{2g}}\M^{(g), N}_{V, l}(P_{1}, \ldots, P_{l}),
    \end{split}
  \end{equation*}
  where the $g$-th term is
  \begin{equation*}
    \begin{split}
      \M^{(g), N}_{V, l}(&P_{1}, \ldots, P_{l})\\ &= \sum_{\bm{n}\in\N^{k}}\frac{\bm{z}^{\bm{n}}}{\bm{n}!}(-1)^{\deg \bm{q}_{\bm{n}}+\deg \bm{P}}\sum_{\substack{\carte \in \mathfrak{C}(g, [\deg \bm{q}_{\bm{n}}+\deg \bm{P}], \bm{\epsilon}_{\bm{q}_{\bm{n}}\bm{P}}, \gamma_{\bm{q}_{\bm{n}}\bm{P}})\\ \carte \text{ is connected }}}(-1)^{c(\phi_{\carte})}\tr_{\phi_{\carte}}(\bm{M}_{\bm{q}_{\bm{n}}\bm{P}}),
    \end{split}
  \end{equation*}
  where $\bm{q}_{\bm{n}}\bm{P}$ is the concatenation of the two tuples
  $\bm{q}_{\bm{n}}$ and $\bm{P}$. In particular,
  $\gamma_{\bm{q}_{\bm{n}}\bm{P}}$ is the permutation defined in
  \eqref{eq:not_gamma} associated to the tuple
  $\bm{q}_{\bm{n}}\bm{P}$.
\end{definition}
Notice that the total numbers of $u$ and $u^{*}$ (or white half-edges)
is $\sum_{i}n_{i}\deg q_{i} + 2m$, and the number of white vertices is
$\sum_{i}n_{i}+l$. The case of the formal free energy corresponds to
$m=0$, $l=0$.

At this point, it is not clear whether the series
$\M^{(g), N}_{V, l}(P_{1}, \ldots, P_{l})$ converge. It will be shown in
Section \ref{sec:exist-form-cumul}.

In Section \ref{sec:dyson-schw-equat}, we will show that in the asymptotic regime, the cumulant
$\W^{N}_{V, l}(P_{1}, \ldots, P_{l})$ coincides with the formal cumulant up to an
arbitrary order, for $\bm{z}$ small enough.
\subsection{Alternated polynomials and Hurwitz numbers}\label{sec:hurwitz}
In this section, we consider a particular case, that is we assume that
all polynomials are alternated monomials (see Definition
\ref{def:alternated_poly}). In particular, this covers the case of a
potential of the form $V = zAU^{N}B(U^{N})^{*}$ encountered in the
HCIZ integral. In \cite{goulden_monotone_2014}, the HCIZ integral had
been expressed in terms of monotone double Hurwitz numbers. In the
multimatrix case, results relating the more general tensor HCIZ
integral to the Hurwitz numbers have been obtained in
\cite{collins_tensor_2023}.

Here we consider only the case where we have a single unitary matrix.

\begin{definition}\label{def:alternated_poly}
  A monomial $P \in \algA$ is said to be \textbf{alternated} if it can be
  written
  \begin{equation*}
    \begin{split}
      P = B_{1}uC_{1}u^{-1}\cdots B_{m}uC_{m}u^{-1},
    \end{split}
  \end{equation*}
  with $B_{i}, C_{i}, 1\leq i \leq m$ words in
  $a_{1}, a_{1}^{*}, \ldots, a_{p}, a_{p}^{*}$.
\end{definition}

In this section, we will assume that all the polynomials involved
($P_{1}, \ldots, P_{l}, q_{1}, \ldots, q_{k}$) are alternated monomials. We
write as before $\bm{P} = (P_{1}, \ldots, P_{l})$. We now explain how we
can give different expressions for the renormalized cumulants in this
case. To do so we reason only using the permutational model of the
maps of unitary type.

In this case, $\bm{\epsilon} = (+1, -1, +1, -1, \ldots)$, and we have
$\gamma_{\bm{P}}(\bm{\epsilon}^{-1}(+1)) = \bm{\epsilon}^{-1}(-1)$ and
$\gamma_{\bm{P}}(\bm{\epsilon}^{-1}(-1)) = \bm{\epsilon}^{-1}(+1)$.
Thus, $\gamma_{\bm{P}} \in \Sym^{(\bm{\epsilon})}_{2m}$. We define
$\tilde{\gamma} = \gamma_{\bm{P}}^{2}|_{\bm{\epsilon}^{-1}(+1)}$.

In particular, this implies that for all
$\carte \in \mathfrak{C}(g, [2m], \bm{\epsilon}, \gamma_{\bm{P}})$, we
have $\phi_{\carte}(\bm{\epsilon}^{-1}(+1)) = \bm{\epsilon}^{-1}(+1)$ and
$\phi_{\carte}(\bm{\epsilon}^{-1}(-1)) = \bm{\epsilon}^{-1}(-1)$. That is, we can write
$\phi_{\carte}$ as a product of two permutations, one, $\rho_{\carte}$,
having its support in $\bm{\epsilon}^{-1}(+1)$, and the other, $\sigma_{\carte}$,
having its support in $\bm{\epsilon}^{-1}(-1)$.

Write $\tau_{\carte} = (\tau_{1}, \ldots, \tau_{r})$. We then notice that the group
generated by $\gamma_{\bm{P}}, \phi_{\carte}, \tau_{1}, \ldots \tau_{r}$ is transitive if
and only if the group generated by
$\gamma_{\bm{P}}, \rho_{\carte}, \sigma_{\carte}, \tau_{1}, \ldots \tau_{r}$
is transitive. As conjugating the elements of a group by one of the
elements does not change the gorup, we have
\begin{equation*}
  \langle{\gamma_{\bm{P}}, \sigma_{\carte}, \rho_{\carte}, \tau_{1}, \ldots, \tau_{r}}\rangle = \langle{\gamma_{\bm{P}}, \gamma_{\bm{P}}^{-1}\sigma_{\carte}\gamma_{\bm{P}}, \rho_{\carte}, \tau_{1}, \ldots, \tau_{r}}\rangle\,.
\end{equation*}

Now, we remark that the subgroup of $\Sym_{2m}$ on the right-hand side
is transitive if and only if the subgroup $\langle{\tilde{\gamma}, \gamma_{\bm{P}}^{-1}\sigma\gamma_{\bm{P}}, \rho, \tau_{1}, \ldots, \tau_{r}}\rangle$ of $\Sym(\epsilon^{-1}(+1))$
is transitive.

This remark allows us to rewrite the sum of Definition
\ref{def:term_2g_cumulant},
\begin{equation*}
  \begin{split}
    \W^{(g), N}_{0, l}(P_{1}, \ldots, P_{l})
    &= (-1)^{m}\sum_{\substack{\carte \in \mathfrak{C}(g, [2m], \bm{\epsilon}_{\bm{P}} \gamma_{\bm{P}})\\\carte \text{ connected}}}(-1)^{c(\phi_{\carte})}\tr_{\phi_{\carte}}(\bm{M}_{P})\\
    &= (-1)^{m}\sum_{\substack{\sigma, \rho \in\Sym_{m}\\\tau\in\mwset_{g}(\Id, \rho^{-1}\tilde{\gamma}\sigma)\\\langle{\tilde{\gamma}, \sigma, \rho, \tau_{1}, \ldots, \tau_{r}}\rangle \text{ transitive}}}(-1)^{c(\sigma)+c(\rho)}\tr_{\rho}(\bm{B}_{\bm{P}})\tr_{\sigma}(\bm{C}_{\bm{P}})\,,\\
  \end{split}
\end{equation*}
where we introduced the notation
\begin{equation*}
  \begin{split}
    \bm{B}_{\bm{P}} &= (M_{i}; i \in \bm{\epsilon}^{-1}(+1))\\
    \bm{C}_{\bm{P}} &= (M_{i}; i \in \bm{\epsilon}^{-1}(-1))\,,
  \end{split}
\end{equation*}
if $\bm{M}_{\bm{P}} = (M_{i})_{1 \leq i \leq \deg \bm{P}}$.

\begin{definition}
  Let $\rho, \gamma, \sigma \in \Sym_{m}$. The \textbf{$r$-th monotone triple
    Hurwitz number} associated to $\rho, \gamma, \sigma$, denoted by
  $\mhurwitz^{r}(\rho, \gamma, \sigma)$, is the number of $r$-uple of
  transpositions $(\tau_{1}, \ldots, \tau_{r}) \in \Sym_{m}^{r}$ such that
  \begin{itemize}
    \item $\tau_{r}\cdots \tau_{1} = \rho\gamma\sigma$;
    \item $\val(\tau_{1}) \leq \val(\tau_{2}) \leq \cdots \leq \val(\tau_{r})$;
    \item the group
          $\langle{\gamma, \rho, \sigma, \tau_{1}, \ldots, \tau_{r}}\rangle \subset \Sym_{m}$
          is transitive.
  \end{itemize}
  When $g$ satisfies the Euler equation
  \begin{equation*}
    \begin{split}
      2-2g = c(\gamma) +c(\rho) + c(\sigma) - r - m,
    \end{split}
  \end{equation*}
  we set
  $\mhurwitz_{g}(\gamma, \sigma, \rho) = \mhurwitz^{r}(\gamma, \sigma, \rho)$.
\end{definition}

This gives us the following Proposition.
\begin{prop}
  Let $\bm{P} = (P_{1}, \ldots, P_{l})$ be alternated monomials. We have
  \begin{equation}
    \begin{split}
      \W^{(g), N}_{0, l}(P_{1}, \ldots, P_{l})
      &= (-1)^{m}\sum_{\sigma, \rho \in\Sym_{m}}(-1)^{c(\sigma)+c(\rho)}\tr_{\rho}(\bm{B}_{\bm{P}})\tr_{\sigma}(\bm{C}_{\bm{P}})\mhurwitz_{g}(\rho^{-1}, \tilde{\gamma}, \sigma).\\
    \end{split}
  \end{equation}
\end{prop}

\begin{remark}
  In the case of the HCIZ integral, we have $\tilde{\gamma} = \Id$,
  thus the monotone triple Hurwitz numbers reduce to the monotone
  double Hurwitz numbers.
\end{remark}
\begin{remark}
  Notice that when all the polynomials are alternated, all the white vertices
  in the unitary type maps involved are alternated vertices (see Definition
  \ref{def:alternated_vertex}).
\end{remark}
\begin{remark}
  In this particular case, the maps of unitary type bear similarity
  with the ribbon graphs introduced in \cite{johnson_hurwitz_2012}.
\end{remark}

\section{Tutte-like equations}\label{sec:induction}
We will now state induction relations that applies to the sums of maps
$\M_{0, l}^{(g), N}$ defined in Definition \ref{def:term_2g_cumulant}. They are obtained by a
procedure very similar to the one used by Tutte in \cite{tutte_enumeration_1968}. These
induction relations are the analog of the topological recursion for
matrices of the GUE \cite{eynard_algebraic_2008}. Similar induction relations have been
obtained for maps related to the Gaussian case in \cite{guionnet_second_2007} and \cite{maurel-segala_high_2006}, and for
maps with ``dotted edges'' in the unitary case for $g = 0$ in \cite{collins_asymptotics_2009}.
More precisely, we will prove the following theorem.
\begin{theorem}\label{thm:induction_V=0}
  Let $N \in \N^{*}$,
  $\bm{P} = (P_{1}, \ldots, P_{l}u) \in \wordsA^{l}$ be monomials.
  Then, for $g \geq 0$ and $m = \frac{1}{2}\deg \bm{P}\geq 2$, we have
  the induction relation
  \begin{equation*}
    \begin{split}
      \M^{(g), N}_{0, l}&(P_{1}, \dots, P_{l}u)\\
      = &-\sum_{P_{l} = QuR}\Big[\M^{(g-1), N}_{0, l+1}(P_{1}, \dots, P_{l-1}, Qu, Ru)
      + \sum_{\substack{g_{1}+g_{2}=g\\I\subset [l-1]}}\M^{(g_{1}), N}_{0, |I|+1}(\bm{P}|_{I}, Qu)\M^{(g_{2}), N}_{0, l-|I|}(\bm{P}|_{I^{c}}, Ru)\Big]\\
      &- \sum_{j=1}^{l-1}\sum_{P_{j}=QuR}\M^{(g), N}_{0, l-1}(P_{1}, \dots, P_{j-1}, P_{j+1}, \dots, P_{l-1}, RQuP_{l}u)\\
      &+\sum_{P_{l} = Qu^{*}R}\Big[\M^{(g-1), N}_{0, l+1}(P_{1}, \dots, P_{l-1}, Q, R)
      + \sum_{\substack{g_{1}+g_{2}=g\\I\subset [l-1]}}\M^{(g_{1}), N}_{0, |I|+1}(\bm{P}|_{I}, Q)\M^{(g_{2}), N}_{0, l-|I|}(\bm{P}|_{I^{c}}, R)\Big]\\
      &+ \sum_{j=1}^{l-1}\sum_{P_{j}=Qu^{*}R}\M^{(g), N}_{0, l-1}(P_{1}, \dots, P_{j-1}, P_{j+1}, \dots, P_{l-1}, RQP_{l}),
    \end{split}
  \end{equation*}
  where we use the notation $\bm{P}|_{I} = (P_{i})_{i\in I}$ and we
  set by convention, $\M^{(-1), N}_{0, l} = 0$ and
  $\M^{(g), N}_{0, 0} = 0$.
\end{theorem}

With a similar proof, we can state a similar theorem for
$\bm{P} = (P_{1}, \ldots, P_{l}u^{*})$. Equivalently, this is a
consequence of the invariance by conjugation of the sums of maps, see
Lemma \ref{lem:prop-M0}.

Theorem \ref{thm:induction_V=0} describes how a map of unitary type
can be decomposed into one or several maps. In the Section
\ref{sec:cut_maps}, we will describe the precise procedure used to cut
maps of unitary type into one or more maps of unitary types, or
equivalently to decompose the permutations describing a particular
map. The procedure can be understood as an elaborate way of
contracting an edge in a map. This allows us to give an interpretation
to the terms appearing in the recurrence.

Consider first the terms in the last two lines (the other terms have a
similar interpretation). When contracting an edge, three situations can occur.
\begin{itemize}
  \item The edge is a non-contractible loop, informally it goes around
        a handle of the surface, in that case when contracting it we
        remove the handle and create a new vertex, this corresponds to
        the term
        \begin{equation*}
          \sum_{P_{l} = Qu^{*}R}\M^{(g-1), N}_{0, l+1}(P_{1}, \dots, P_{l-1}, Q, R)\,.
        \end{equation*}
  \item The edge is a contractible loop. When contracting the edge we cut the
        map into two disconnected components. This corresponds to the
        term
        \begin{equation*}
          \sum_{P_{l} = Qu^{*}R}\sum_{\substack{g_{1}+g_{2}=g\\I\subset [l-1]}}\M^{(g_{1}), N}_{0, |I|+1}(\bm{P}|_{I}, Q)\M^{(g_{2}), N}_{0, l-|I|}(\bm{P}|_{I^{c}}, R)\,.
        \end{equation*}
  \item The edge is not a loop. When contracting the edge we just merge
        two vertices. This corresponds to the term
        \begin{equation*}
          \sum_{j=1}^{l-1}\sum_{P_{j}=Qu^{*}R}\M^{(g), N}_{0, l-1}(P_{1}, \dots, P_{j-1}, P_{j+1}, \dots, P_{l-1}, RQP_{l})\,.
        \end{equation*}
\end{itemize}
The other terms have a similar form. They do not correspond to the
contraction of an edge, but rather to the erasure of a black vertex.
This Theorem shows that a sum of maps can be expressed in terms of the
sums of maps with either lower genus, lower overall degree, or lower
number of vertices.

Before describing precisely the procedure used to cuts the maps of
unitary type, and giving the proof of Theorem \ref{thm:induction_V=0}, we show how to
rewrite these equation in a compact way. The function
$\M^{(g), N}_{0, l}$ was defined only on some monomials, i.e. on the
set $\wordsA^{l}$. Recall that $\algA$ is the algebra of
noncommutative polynomials in the variables
$u, u^{*}, a_{1}, a_{1}^{*}, \ldots$. We now extend $\M^{(g), N}_{0, l}$ by
linearity to an linear form on the tensor space
$\algA^{\otimes l} = \algA\otimes_{\C} \cdots \otimes_{\C} \algA$. We will use intechangeably
the notation $\M^{(g), N}_{0, l}(P_{1}, \ldots, P_{l})$
($M^{(g), N}_{0, l}$ as a multilinear function) and
$\M^{(g), N}_{0, l}(P_{1}\otimes \cdots \otimes P_{l})$ ($M^{(g), N}_{0, l}$ as a
linear function on $\algA^{\otimes l}$). We can then rewrite Theorem \ref{thm:induction_V=0}
using the notion of non-commutative derivative.

Let $\bm{P} = (P_{1}, \ldots, P_{l})$ be a $k$-tuple of polynomials,
and $I = \{i_{1}< i_{2} < \cdots < i_{p}\}$ be a non-empty subset of
$[l]$, then we define
\begin{equation*}
  \begin{split}
    \bm{P}_{I} = P_{i_{1}}\otimes P_{i_2} \otimes \cdots \otimes P_{i_{p}}.
  \end{split}
\end{equation*}
We define the operations $\times$ and $\sharp$ as follows. Let
$\bm{P} \in \algA^{l}$, $I \subset [l-1]$, $Q = Q_{1}\otimes Q_{2} \in \algA^{\otimes 2}$ and
$S = S_{1}\otimes S_{2} \in \algA^{\otimes 2}$, then we set
\begin{equation}\label{eq:def-times-tensor}
  Q \times S = (Q_{1}\otimes Q_{2}) \times (S_{1}\otimes S_{2}) = (Q_{1}S_{1}) \otimes (Q_{2}S_{2})\,,
\end{equation}
and
\begin{equation}\label{eq:def-sharp-tensor}
  \bm{P}_{I}\otimes \bm{P}_{I^{c}} \sharp Q = \bm{P}_{I}\otimes Q_{1} \otimes \bm{P}_{I^{c}} \otimes Q_{2}.
\end{equation}

\begin{definition}\label{def:nc_derivative}
  The \textbf{logarithmic non-commutative derivative}
  $\partial\colon\algA \to \algA^{\otimes 2}$ with respect to $u$ of a
  monomial $P\in \algA$ is defined by
  \begin{equation*}
    \begin{split}
      \partial P = \sum_{P=QuR}Qu\otimes R - \sum_{P=Qu^{-1}R}Q\otimes u^{-1}R \in \algA^{\otimes 2}.
    \end{split}
  \end{equation*}
  The definition extends by linearity to any polynomial in $\algA$.
\end{definition}

\begin{remark}
  This derivative was introduced by Voiculescu in \cite[Section 8.1]{voiculescu_analogues_1999}. However, as
  remarked by an anonymous reviewer, this corresponds to the
  derivation on the unitary group invariant under multiplication from
  the right. Near the identity, this derivation corresponds to the
  ordinary gradient $\nabla_{H}$ where $\exp^{H} = U$, hence the name
  logarithmic derivative.
\end{remark}

\begin{remark}\label{rem:interpretation-nc-derivative}
  Consider a monomial $P \in \algA$ evaluated in
  $U, U^{*}, A_{1}, A_{1}^{*}, \ldots$, and denote by
  $\partial_{m, j}$ the derivative with respect to the coefficient
  $(m, j)$ of $U$. We have
  \begin{equation*}
    \sum_{m}U_{m, k}(\partial_{m,j}P)_{i,l} = \sum_{m}\sum_{P = QUR}Q_{i, m}U_{m, k}R_{j, l} = \sum_{P=QUR}(QU)_{i, k}R_{j, l}\,,
  \end{equation*}
  and similarly, if $\overline{\partial}_{m, j}$ is the derivative with
  respect to the coefficient $(m, j)$ of $U^{*}$
  \begin{equation*}
    \sum_{m}U_{j, m}^{*}(\overline{\partial}_{k, m}P)_{i,l} = \sum_{m}\sum_{P = QU^{*}R}Q_{i, k}U^{*}_{j, m}R_{m, l} = \sum_{P=QUR}Q_{i, k}(U^{*}R)_{j, l}\,.
  \end{equation*}

  In coordinates, the non-commutative derivative thus corresponds to
  \begin{equation*}
    (\partial P)_{\substack{i, k\\j, l}} = \sum_{m} \left( U_{m, k}\partial_{m,j}P -  U_{j, m}^{*}\overline{\partial}_{k, m}P\right)_{i, l}\,.
  \end{equation*}
\end{remark}

\begin{definition}\label{def:cyclic_derivative}
  The \textbf{cyclic derivative} $\D\colon \algA \to \algA$ with respect to $u$ of a monomial
  $P\in \algA$ is defined by
  \begin{equation*}
    \begin{split}
      \D P = \sum_{P=QuR}RQu - \sum_{P=Qu^{-1}R}u^{-1}RQ.
    \end{split}
  \end{equation*}
  The definition extends by linearity to any polynomial in $\algA$.
\end{definition}

\begin{remark}
  Let $P$ be a monomial. In coordinates, the cyclic derivative is
  \begin{equation*}
    (\D P)_{i, j} = \sum_{k, m} \left( U_{m, j}\partial_{m,i}P - U^{*}_{m, j}\overline{\partial}_{j, m}P \right)_{k, k} = \sum_{k}(\partial P)_{\substack{j, k\\i, k}}\,.
  \end{equation*}
\end{remark}

To take advantage of the notation just introduced, we shall write
$\M^{(g), N}_{0, l}\otimes \M^{(g'), N}_{0, l'}(P_{1}\otimes \cdots \otimes P_{l+l'})$
with $P_{i} \in \algA$ to mean
\begin{equation*}
  \M^{(g), N}_{0, l}\otimes \M^{(g'), N}_{0, l'}(P_{1} \otimes \cdots \otimes P_{l+l'})
  = \M^{(g), N}_{0, l}(P_{1}, \ldots, P_{l})\M^{(g'), N}_{0, l'}(P_{l+1}, \ldots, P_{l+l'})\,.
\end{equation*}
Theorem \ref{thm:induction_V=0} can then be rewritten as follows.
\begin{corol}\label{corol:induction-general}
  For $m\geq 2$, $g \geq 1$, and $P_{1}, \ldots, P_{l} \in \algA$, we
  have the following equation
  \begin{equation}
    \begin{split}
      \sum_{I \subset [l-1]}&\sum_{g = g_{1}+g_{2}}\M^{(g_{1}), N}_{0, |I|+1}\otimes\M^{(g_{2}), N}_{0, |I^{c}|+1}(\bm{P}_{I}\otimes\bm{P}_{I^{c}}\sharp \partial P_{l})\\
      &= -\M^{(g-1), N}_{0, l+1}(P_{1}\otimes \cdots \otimes P_{l-1} \otimes \partial P_{l}) - \sum_{j=1}^{l-1}\M^{(g), N}_{0, l-1}(P_{1}\otimes \cdots \otimes \check{P_{j}} \otimes \cdots \otimes (\D P_{j})P_{l}),
    \end{split}
  \end{equation}
  where $\check{P_{j}}$ means that the factor $P_{j}$ is omitted.
\end{corol}
\begin{proof}
  We group the terms two by two. We have
  \begin{equation*}
    \begin{split}
      -\sum_{P_{l} = QuR}&\M^{(g-1), N}_{0, l+1}(P_{1}, \dots, P_{l-1}, Qu, Ru)
                           +\sum_{P_{l} = Qu^{*}R}\M^{(g-1), N}_{0, l+1}(P_{1}, \dots, P_{l-1}, Q, R)\\
                         &= - \M^{(g-1), N}_{0, l+1}(P_{1}, \dots, P_{l-1}, \partial (P_{l}u)) + \M^{(g-1), N}_{0, l+1}(P_{1}, \dots, P_{l-1}, P_{l}, 1)\\
                         &= - \M^{(g-1), N}_{0, l+1}(P_{1}, \dots, P_{l-1}, \partial (P_{l}u))\,.
    \end{split}
  \end{equation*}
  where we used the traciality property of Lemma \ref{lem:prop-M0} to
  replace $u^{*}Ru$ by $R$ in the last argument of
  $\M^{(g-1), N}_{0, l+1}$, in the second line. In the third line we
  used the fact that there are no connected map with $l+1$ vertices,
  $l \geq 1$, and one vertex of degree $0$.

  We proceed similarly for the two other terms. We have
  \begin{equation*}
    \begin{split}
      - \sum_{P_{l} = QuR}&\sum_{\substack{g_{1}+g_{2}=g\\I\subset [l-1]}}\M^{(g_{1}), N}_{0, |I|+1}(\bm{P}|_{I}, Qu)\M^{(g_{2}), N}_{0, l-|I|}(\bm{P}|_{I^{c}}, Ru)\\
      +\sum_{P_{l} = Qu^{*}R}&\sum_{\substack{g_{1}+g_{2}=g\\I\subset [l-1]}}\M^{(g_{1}), N}_{0, |I|+1}(\bm{P}|_{I}, Q)\M^{(g_{2}), N}_{0, l-|I|}(\bm{P}|_{I^{c}}, R)\\
                          &= -\sum_{\substack{g_{1}+g_{2}=g\\I\subset [l-1]}}\M^{(g_{1}), N}_{0, |I|+1}\otimes\M^{(g_{2}), N}_{0, l-|I|}(\bm{P}|_{I}\otimes\bm{P}|_{I^{c}} \sharp (\partial Pu))
      + \M^{(g), N}_{0, l}(\bm{P})\,,
    \end{split}
  \end{equation*}
  where we used that there are no maps with one vertex of degree 0 if
  $g \geq 1$ or $l \geq 2$.

  The third term is
  \begin{equation*}
    \begin{split}
      - \sum_{j=1}^{l-1}\sum_{P_{j}=QuR}&\M^{(g), N}_{0, l-1}(P_{1}, \dots, P_{j-1}, P_{j+1}, \dots, P_{l-1}, RQuP_{l}u)\\
      + \sum_{j=1}^{l-1}\sum_{P_{j}=Qu^{*}R}&\M^{(g), N}_{0, l-1}(P_{1}, \dots, P_{j-1}, P_{j+1}, \dots, P_{l-1}, RQP_{l})\\
                                        &= - \sum_{j=1}^{l-1}\M^{(g), N}_{0, l-1}(P_{1}, \dots, P_{j-1}, P_{j+1}, \dots, P_{l-1}, (\D P_{j})P_{l}u)\,.
    \end{split}
  \end{equation*}
  Putting the three terms together, we get the result when $P_{l}$ is
  a monomial that finishes by $u$. The cyclicity property of Lemma \ref{lem:prop-M0}
  implies that it is then the same if $P_{l}$ contains a $u$. Finally,
  if $P_{l}$ does not contains a $u$, then it is either a constant and
  the formula is clear, or it contains a $u^{*}$. In the latter case,
  the conjugation property of Lemma \ref{lem:prop-M0} allows us to recover the
  result.

  Finally, the result extend by linearity to all polynomials, as it is
  linear in each of the $P_{i}$.
\end{proof}

Notice that the formula in Corollary \ref{corol:induction-general} is
valid not only for monomials $P_{l}$ finishing by a $u$ but also for
all polynomials.

\subsection{How to cut maps}\label{sec:cut_maps}

In this section, we fix two integers $m \geq 2$ and $g \geq 0$, a
permutation $\gamma \in \Sym_{2m}$ and
$\bm{\epsilon} \in \mathcal{E}_{2m} = \{\bm{\epsilon} = (\epsilon(i))_{i\in [2m]} \in \{\pm 1\}^{2m}\colon \sum_{i=1}^{2m}\epsilon(i) = 0\}$.
By the cyclicity and the symmetry properties proved in Lemma \ref{lem:prop-M0}, we
can assume that the polynomial $P_{l}$ finishes by the letter $u$. We
thus assume that $\epsilon(2m) = +1$. We consider a map of unitary type,
$\carte \in \mathfrak{C}(g, m, \gamma, \bm{\epsilon})$, with $r$ black vertices. Let
$\surf$ denote the underlying surface of the map $\carte$.

By Theorem \ref{thm:permutational_model}, this map is described by the permutation
$\pi \coloneqq \pi_{\carte}$ and the tuple of transpositions
$\tau_{\carte} = (\tau_{1}, \ldots, \tau_{r}) \in \mwset^{r}(\pi^{(\epsilon)})$, with $r$
related to $g$, $\pi_{\carte}$ and $\gamma_{\carte}$ by Euler's formula, see
(\ref{eq:Euler_for_map}). We will consider two ways of cutting this map, depending on
whether $\tau_{r}(2m) = 2m$ or not.

\begin{remark}\label{rem:vertex_boundary_component}
  We can also see vertices as ``holes'' in the surface, that is, we take the
  underlying surface $\surf$ to be a surface with boundaries. A vertex is then a
  boundary component (homeomorphic to a circle) of the surface $\surf$. An edge is
  then a path connecting two boundary components. See Figure
  \ref{fig:ex_vertex_boundary} for an example.
\end{remark}
\begin{figure}[htbp]
  \centering
  \begin{minipage}{0.45\textwidth}
      \centering
      \includegraphics[width=0.5\textwidth]{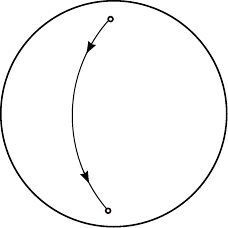}
      \caption*{(a)}
    \end{minipage}
    \begin{minipage}{0.45\textwidth}
      \centering
      \includegraphics[width=0.45\textwidth]{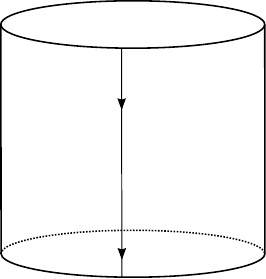}
      \caption*{(b)}
  \end{minipage}\hfill
  \caption{\label{fig:ex_vertex_boundary}(a) An oriented map of genus 0 with
    two vertices (b) The corresponding map where vertices are seen as boundary
    components. The underlying surface is a cylinder.}
\end{figure}

We will see white vertices as boundary components of the underlying
surface, as explained in Remark \ref{rem:vertex_boundary_component}. To describe the cutting
procedure, we introduce the definition of a corner, see Figure \ref{fig:corner}.
\begin{definition}\label{def:corner}
  Let $\carte$ be a map of unitary type. Consider a (white or black)
  vertex $v$ in $\carte$, seen as a boundary component of a surface,
  as explained in Remark \ref{rem:vertex_boundary_component}, and a half-edge $h$.

  The \textbf{corner} at the left (respectively at the right) of $h$ is the
  connected, closed set $C$ such that
  \begin{itemize}
    \item $C$ is a subset of the boundary component $B$ corresponding
          to $v$;
    \item $C$ is in the boundary of the face $f$ at the left (respectively
          right) of the half-edge $h$;
    \item the boundary of $C$ is $\{u, v\}$, where $u$ is the point of
          intersection of $B$ and the half-edge $h$, and $v$ is the
          point of intersection of $B$ and the half-edge $h'$ that
          follows the half-edge $h$ when going around the face $f$ in
          the clockwise (resp. counterclockwise) direction starting
          from the left (resp. the right) of $h$.
  \end{itemize}
\end{definition}

\begin{figure}[ht]
  \centering
  \includegraphics[width=0.3\textwidth]{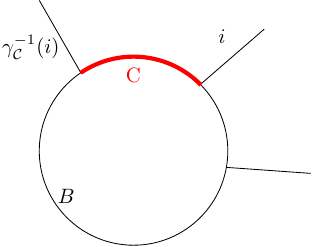}
  \caption{\label{fig:corner} The corner $C$ at the left of $i$, displayed with a red
    thick line. Here, $i$ is a white half-edge.}
\end{figure}

\subsubsection{First case: $\tau_{r}(2m) = 2m$}
In this case, $\val(\tau_{r}) < 2m$, see Definition
\ref{def:value_permutation}. The monotonicity of the walk $\bm{\tau}$
implies that for all $i \in [r]$, $\val(\tau_{i}) < 2m$, and in
particular $\tau_{i}(2m) = 2m$. By Definition
\ref{def:permutation-unitary}, the half-edge $2m$ is not connected to
any black half-edge, and is thus connected to a white-half-edge, say
the $j$-th one. Note that by our assumption that $\epsilon(2m) = +1$,
$\epsilon(j) = -1$. Notice that because $\carte$ is non-decreasing,
$\tau_{r}(2m) = 2m$ implies that for all $i$, $\tau_{i}(2m) = 2m$.

We construct a map $\carte'$ of unitary type from $\carte$ using the
following procedure, depicted in Figure \ref{fig:case_1_cut}.
\begin{enumerate}
  \item We choose a path $\eta$ in the face $f$ at the right of the half-edge
        $2m$. This path is chosen to start from the white vertex
        $w_{2m}$, attached in the corner (see Definition
        \ref{def:corner}) at the right of the half-edge $2m$, and end
        at $w_{j}$, attached in the corner at the left of the
        half-edge $j$. As faces are homeomorphic to disks, there is
        only one way to choose $\eta$ up to homotopy.
  \item We remove the edge containing the half-edges $j$ and $2m$.
  \item We cut the surface along $\eta$. Depending on the cases we connect two
        distinct boundary components of $\surf$, or we connect one boundary
        component to itself.
\end{enumerate}

\begin{remark}
  Notice that if $w_{j}$ and $w_{2m}$ are distinct vertices, this surgery is the
  usual contraction of an edge.
\end{remark}

\begin{remark}
  If the path $\eta$ is a loop, then it may be that the resulting map
  $\carte'$ is disconnected. Furthermore, one of the connected
  component can be a vertex with no edges. In this case, we remove
  this component.
\end{remark}

\begin{figure}[htbp]
  \begin{minipage}{0.33\textwidth}
    \centering
    \includegraphics[width=0.45\textwidth]{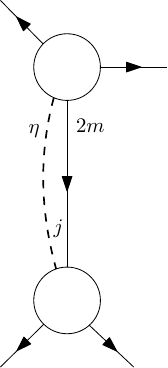}
    \caption*{1.}
  \end{minipage}\hfill
  \begin{minipage}{0.33\textwidth}
      \centering
      \includegraphics[width=0.45\textwidth]{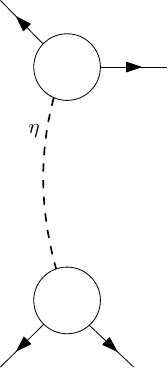}
    \caption*{2.}
  \end{minipage}
  \begin{minipage}{0.33\textwidth}
      \centering
      \includegraphics[width=0.45\textwidth]{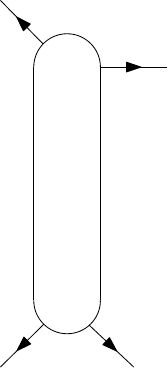}
    \caption*{3.}
  \end{minipage}
  \caption{\label{fig:case_1_cut}First way to cut the map.}
\end{figure}

\begin{lemma}\label{lem:first-cut-map}
  If $\carte$ is a map of unitary type with label set $I$ of size
  $2m$, then $\carte'$ is a map of unitary type, with label set
  $I \setminus \{j, 2m\}$, of size $2m-2$. Furthermore, if $\carte$ is
  non-decreasing, then so is $\carte'$.
\end{lemma}
\begin{proof}
  We first check that $\carte'$ is a map. We only need to check that
  each face of $\carte'$ is homeomorphic to a disk. We only modify the
  faces $f_{\text{left}}$ and $f_{\text{right}}$ at the left and the
  right of the edge $(w_{2m}, w_{j})$. At step 2, when we remove the
  edge, we connect $f_{\text{left}}$ and $f_{\text{right}}$. However,
  at step 3, we cut along a path homotopic to the edge, thus
  separating the two faces. All the faces of $\carte'$ thus remain
  disks.

  We now check that this map is a map of unitary type. The map
  $\carte'$ has $2m-2$ labelled white half-edges (maybe 0 if $m = 1$),
  with label set $I \setminus \{j, 2m\}$ (property \ref{it:unit-wv} in Definition \ref{def:maps_unitary_type} is
  satisfied). The black vertices have not been modified when
  transforming $\carte$ into $\carte'$, so properties \ref{it:unit-bv} and \ref{it:unit-order} in
  Definition \ref{def:maps_unitary_type} are satisfied. As the black vertices are not
  modified, a non-decreasing map remain non-decreasing.
\end{proof}

Let us now compute the permutations that represent $\carte'$. We will
need the notion of the trace of a permutation introduced by Kreweras
\cite{kreweras_sur_1972}. This notion has nothing to do with the notion of trace of a
matrix.
\begin{definition}\label{def:trace-permutation}
  Let $A$ be a finite subset of $\N$. The \textbf{trace} of a permutation
  $\sigma \in \Sym(A)$, on $B \subset A$, denoted by $\Tr(\sigma; B)$, is the
  permutation in $\Sym(B)$ defined for each $x \in B$ by
  \begin{equation*}
    \begin{split}
      \Tr(\sigma; B)(x) = \sigma^{p_{x}}(x),
    \end{split}
  \end{equation*}
  with $p_{x} \geq 1$ the smallest integer so that
  $\sigma^{p}(x) \in B$.
\end{definition}
Computing the trace of a permutation in cycle notation is
straightforward: write the cycle decomposition of $\sigma$, and erase all
elements in the cycles that do not belong to $B$.

Let $I_{j} = [2m-1]\setminus\{j\}$.

\begin{lemma}\label{lem:pi_A}
  Let $\pi' \coloneqq \Tr(\pi_{\carte}; I_{j})$.
  We have $\pi' = \pi_{\carte}\cycle{j, 2m} = \cycle{j, 2m}\pi_{\carte}$, and $\pi_{\carte'} = \pi' = \pi_{\carte}|_{I_{j}}$.
\end{lemma}
\begin{proof}
  We have assumed that the half-edges $2m$ and $j$ are connected to
  form an edge. This imply $\pi_{\carte}(2m) = j$ and $\pi_{\carte}(j) = 2m$. Thus, as $\pi_{\carte}$
  is a permutation, for all $k \in I_{j}$, $\pi_{\carte}(k) \in I_{j}$. This
  means that in the notation of Definition
  \ref{def:trace-permutation}, $p_{k}=1$. We get the first claim, and
  that $\pi' = \pi_{\carte}|_{I_{j}}$.

  When removing the edge at step 2, it is clear that the map we obtain is still
  described by $\pi_{\carte}$, with a cycle removed. When cutting the map at
  step 3, we do not modify the edges further.
\end{proof}

\begin{lemma}\label{lem:rho_A}
  Let $\gamma' \coloneqq \Tr(\gamma\cycle{j, 2m}; I_{j})$.
  We have $\gamma_{\carte'} = \gamma'$.
\end{lemma}
\begin{proof}
  Assume first that $w_{j}$ and $w_{2m}$ are two distinct vertices. Let
  $c = \cycle{u_{1}, \ldots, u_{p}, j}$ and
  $c' = \cycle{u'_{1}, \ldots, u'_{p'}, 2m}$ be the cycles that represent them.
  After cutting the map at step 3, the vertices are replaced by a vertex with
  structure
  $\cycle{u_{1}, \ldots, u_{p}, u'_{1}, \ldots, u'_{p'}} = \Tr(cc'\cycle{j, 2m}; I_{j})$.

  If $w_{j} = w_{2m}$, this vertex is represented by a cycle
  $c = \cycle{u_{1}, \ldots, u_{p}, j, u'_{1}, \ldots, u'_{p'}, 2m}$,
  which we cut using the transposition $\cycle{j, 2m}$. We obtain two
  vertices represented by the two cycles $\Tr(c\cycle{j, 2m}; I_{j})$.
\end{proof}

\begin{lemma}\label{lem:phi_A}
  We have $\phi_{\carte'} = \Tr(\phi_{\carte}; I_{j})$.
\end{lemma}
\begin{proof}
  By Lemmas \ref{lem:phi_map}, \ref{lem:pi_A} and \ref{lem:rho_A},
  \begin{equation*}
    \begin{split}
      \phi_{\carte'} = \gamma'\pi'^{-1} = \Tr(\gamma_{\carte}\cycle{j, 2m}; I_{j})\pi_{\carte}|_{I_{j}}^{-1}\,.
    \end{split}
  \end{equation*}

  Notice first that for any $k \in I$ and $p \in \N$, $(\gamma_{\carte}\pi_{\carte}^{-1})^{p}(k) \in \{j, 2m\}$ if and only if $(\cycle{j, 2m}\gamma_{\carte}\pi_{\carte}^{-1}\cycle{j, 2m})^{p}(k) \in \{j, 2m\}$. This implies that
  \begin{equation*}
    \Tr(\gamma_{\carte}\pi_{\carte}^{-1}; I_{j}) = \Tr(\cycle{j, 2m}\gamma_{\carte}\pi_{\carte}^{-1}\cycle{j, 2m}; I_{j}) = \Tr(\cycle{j, 2m}\gamma_{\carte}^{-1}\pi'^{-1}; I_{j})\,,
  \end{equation*}
  were we used Lemma \ref{lem:pi_A} for the last equality.

  Then, as $\pi'(j) = j$ and $\pi'(2m) = 2m$, we have
  $\Tr(\cycle{j, 2m}\gamma\pi'^{-1}; I_{j}) = \Tr(\cycle{j, 2m}\gamma; I_{j})\pi'^{-1} = \phi_{\carte'}$.
\end{proof}

\begin{lemma}\label{lem:cc_A}
  If the map $\carte$ of unitary type is connected, then the map
  $\carte'$ has one or two connected components. Furthermore, if $j$
  and $2m$ do not belong to the same cycle in $\gamma$, $\carte'$ is
  connected.
\end{lemma}
\begin{proof}
  Assume first that $j$ and $2m$ belong to the same cycle in $\gamma$. This
  means that $w_{j} = w_{2m}$. If we erase the edge containing the half-edges
  $2m$ and $j$, $\carte$ stays connected. However, when we cut the map along the
  path $\eta$, we may separate the map into two connected components. More
  precisely, we separate the map into two connected components if and only if
  $\eta$ is homologically trivial, that is, the boundary of a surface embedded
  in $\surf$.

  If $j$ and $2m$ belong to different cycles, that is
  $w_{j} \neq w_{2m}$, then when removing the edge we may disconnect
  the two vertices but we then merge them. Consequently, the map
  $\carte'$ is connected.
\end{proof}

Using the permutations $\gamma' = \gamma_{\carte'}$ and $\phi_{\carte}$, and Lemma
\ref{lem:cc_A}, we can now compute the genus of $\carte'$. We recall Euler's
formula \eqref{eq:Euler_for_map} for a map of genus $g_{\carte}$ with
$C_{\carte}$ connected components
\begin{equation*}
  2C_{\carte} - 2g_{\carte} = c(\gamma_{\carte}) + c(\phi_{\carte}) - m - r.
\end{equation*}

We now discuss the several cases that can occur. First, if both $j$
and $2m$ are fixed points of $\phi_{\carte}$ then it means that
$\carte$ is reduced to a vertex with $2$ half-edges. We will assume in
what follows that $m \geq 2$.

If $j$ (respectively $2m$) is a fixed point of
$\phi_{\carte} = \gamma_{\carte}\pi_{\carte}^{-1}$, then it means in terms of
map that the face at the left of the white half-edge $j$ (resp. $2m$)
is at the left of the half-edge $j$ (resp. $2m$) only. In term of
permutations, we have $\gamma_{\carte}(2m) = \phi_{\carte}\pi_{\carte}(2m) = j$
(resp. $\gamma_{\carte}(j) = 2m$). Then, $\cycle{j, 2m}\gamma_{\carte}(2m) = 2m$
(resp. $\cycle{j, 2m}\gamma_{\carte}(j) = j$), and when taking the trace on
$I_{j}$, we remove one cycle of $\cycle{j, 2m}\gamma_{\carte}$.
Furthermore, if we remove the disk corresponding to the face at the
left of the half-edge $j$ (resp. $2m$), the resulting map $\carte'$ is
connected.

If $j$ and $2m$ are not fixed points of $\phi_{\carte}$, then none of
the connected component is reduced to a vertex without edges. It
implies that the total number of faces and cylces of the associated
permutation stays the same. We have $c(\phi_{\carte}) = c(\phi_{\carte'})$.
The resulting map $\carte'$ has one or two connected components: by
cutting the map we either remove a handle (and decreased the genus by
one) or disconnected the map.

It gives us one particular case (the degenerate case were one
connected component is reduced to a vertex without edges):
\begin{enumerate}
  \item if $j$ or $2m$ is a fixed point of $\phi_{\carte}$, then
        $c(\phi_{\carte'}) = c(\phi_{\carte})-1$, $c(\gamma') = c(\gamma)$, and
        $\carte'$ is connected. Thus, $g_{\carte'} = g_{\carte}$.
\end{enumerate}
If both $j$ and $2m$ are not fixed points of $\phi_{\carte}$, we have the three
cases.
\begin{enumerate}
  \setcounter{enumi}{1}
  \item If $j$ and $2m$ belong to the same cycle of $\gamma_{\carte}$, and $\carte'$ is
        connected, then $c(\gamma_{\carte'}) = c(\gamma_{\carte})+1$,
        $c(\phi_{\carte'}) = c(\phi_{\carte})$, and $g' = g-1$ (by
        Euler formula).
  \item If $j$ and $2m$ belong to the same cycle of $\gamma_{\carte}$, and
        $\carte'$ has two connected components, then
        $c(\gamma_{\carte'}) = c(\gamma_{\carte})+1$,
        $c(\phi_{\carte'}) = c(\phi_{\carte})$, and $g' = g$.
  \item If $j$ and $2m$ belong to two different cycles of $\gamma_{\carte}$,
        then $c(\gamma_{\carte'}) = c(\gamma_{\carte})-1$,
        $c(\phi_{\carte'}) = c(\phi_{\carte})$, and $g'=g$.
\end{enumerate}

\subsubsection{Second case: $\tau_{r}(2m) \neq 2m$}
In this case, the white half-edge labelled $2m$ is connected to a
black vertex. Let $j = \tau_{r}(2m) \in \epsilon^{-1}(+1)$ (as the support of the
transpositions $\tau_{i}$ is contained in $\epsilon^{-1}(+1)$). In that case,
the last black vertex has an outgoing half-edge labelled by $2m$ by
Lemma \ref{lem:label_he_in-out}. Similarly, the last black vertex has
an outgoing half-edge labelled $j$.

We construct a unitary map $\carte'$ from $\carte$ using the following
procedure, depicted in Figure \ref{fig:case_2_cut}. Notice that the first step is
possible as there are no loop with black edges, as indicated in Remark
\ref{rem:face-incident_white}.
\begin{enumerate}
  \item We choose two paths $\eta_{1}$ and $\eta_{2}$ contained respectively in
        the face $f_{2m}$ at the left of the half-edge $2m$, and
        $f_{j}$ at the left of the half-edge $j$. The path $\eta_{1}$
        (respectively $\eta_{2}$) is chosen to start from the white
        vertex $w_{2m}$ (respectively $w_{j}$), attached in the corner
        at the left of the half-edge $2m$ (respectively $j$), and end
        at the $r$-th black vertex, attached in the corner at the left
        of the half-edge labelled $2m$ (respectively $j$).
  \item We remove the $r$-th black vertex, and attach each ingoing edge to
        the outgoing edge that follows it in the counterclockwise
        order.
  \item We cut the surface $\surf$ along $\eta = \eta_{1}\cup\eta_{2}$.
\end{enumerate}

\begin{figure}[htbp]
  \begin{minipage}{0.33\textwidth}
    \centering
    \includegraphics[width=0.65\textwidth]{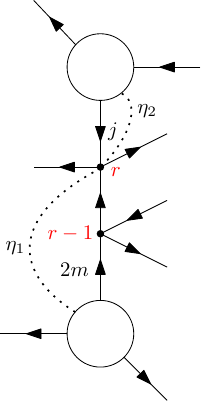}
    \caption*{1.}
  \end{minipage}\hfill
  \begin{minipage}{0.33\textwidth}
      \centering
      \includegraphics[width=0.65\textwidth]{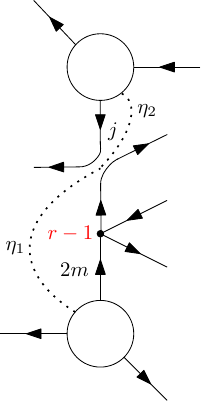}
    \caption*{2.}
  \end{minipage}
  \begin{minipage}{0.33\textwidth}
      \centering
      \includegraphics[width=0.65\textwidth]{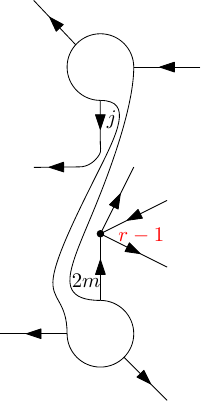}
    \caption*{3.}
  \end{minipage}
  \caption{\label{fig:case_2_cut}Second way to cut the map.}
\end{figure}

\begin{lemma}
  If $\carte$ is a map of unitary type with labelling set $I$ and with
  $r$ black vertices, then $\carte'$ is a map of unitary type with
  labelling set $I$ and with $r-1$ black vertices. Furthermore, if
  $\carte$ is non-decreasing, then so is $\carte'$.
\end{lemma}
\begin{proof}
  As in the proof of Lemma \ref{lem:first-cut-map}, we first prove
  that $\carte'$ is a map. We have to check that the faces are
  homeomorphic to disks. When removing the black vertex, at step 2, we
  may have connected two faces together, or may have connected a face
  to itself, thus creating a ``face'' homeomorphic to an annulus.
  However, when we cut the map, at step 3, we recover one or two faces
  homeomorphic to disks. Thus, $\carte'$ is a map.

  We now show that $\carte'$ is indeed a map of unitary type. The map
  $\carte'$ has $r-1$ black vertices. We removed the last black vertex
  and did not create any new edge linking two black vertices, thus
  properties \ref{it:unit-bv} and \ref{it:unit-order} of Definition
  \ref{def:maps_unitary_type} are satisfied. We did not remove any
  white half-edge so property \ref{it:unit-wv} is satisfied as well.
  Thus, $\carte'$ is of unitary type. Furthermore, as we removed the
  last black vertex and let the other ones unchanged, if $\carte$ is
  non-decreasing (recall Definition \ref{def:monotone-unitary-map}),
  $\carte'$ is non-decreasing as well.
\end{proof}

Let us now compute the permutations that represent $\carte'$.

\begin{lemma}\label{lem:pi_B}
  Let $\pi' = \cycle{j, 2m}\pi_{\carte} = \tau_{r}\pi_{\carte}$. We have
  $\pi_{\carte'} = \pi'$.
\end{lemma}
\begin{proof}
  We only modify the edges during step 2, when we remove the black
  vertex. The outgoing half-edges of the $r$-th black vertex in
  $\carte$ are labelled on the left by $j$ and $2m$. These half-edges
  are part of edges connected at their other end to white vertices,
  because the black vertex we remove is the last. These edges are
  connected respectively to the half edge $\pi_{\carte}^{-1}(j)$ and
  $\pi_{\carte}^{-1}(2m)$.

  Consider the white half-edge labelled by $\pi_{\carte}^{-1}(j)$. After
  the surgery, it is connected to the half-edge labelled $2m$.
  Similarly, the white half-edge labelled $\pi_{\carte}^{-1}(2m)$ is
  attached to the half-edge labelled $j$.

  This corresponds to having $\pi_{\carte'}(\pi_{\carte}^{-1}(j)) = 2m$ and
  $\pi_{\carte'}(\pi_{\carte}^{-1}(2m)) = j$, and $\pi_{\carte'} = \pi_{\carte}$ for all other
  values. We can write this $\pi_{\carte'} = \cycle{j, 2m}\pi_{\carte}$.
\end{proof}

Note that
\begin{equation*}
  \pi'^{(\epsilon)} = (\tau_{r}\pi_{\carte}\tau_{r}\pi_{\carte})|_{\epsilon^{-1}(+1)} = (\tau_{r}\pi_{\carte}^{2})|_{\epsilon^{-1}(+1)} = \tau_{r}\pi_{\carte}^{(\epsilon)} = \tau_{r-1}\cdots\tau_{1}\,.
\end{equation*}
The first and third equalities are Definition
\ref{def:permutations-epsilon}, the second one follows from the fact
that $\pi_{\carte}(\epsilon^{-1}(+1)) = \epsilon^{-1}(-1)$ and the fact that the support of
$\tau_{r}$ is contained in $\epsilon^{-1}(+1)$, the fourth one is a
consequence of Proposition \ref{prop:tau_carte}. This is coherent with
the fact that
\begin{equation}\label{eq:def-tau-carte'}
  \tau_{\carte'} = (\tau_{1}, \ldots, \tau_{r-1})\,.
\end{equation}

\begin{lemma}\label{lem:rho_B}
  Let $\gamma' = \gamma\cycle{j, 2m} = \gamma\tau_{r}$.

  We have $\gamma_{\carte'} = \gamma'$
\end{lemma}
\begin{proof}
  The white vertices are only modified when we cut the map, at step 3.
  The proof is similar to the one of Lemma \ref{lem:rho_A}. We consider the two
  cases of $j$ and $2m$ in the same cycle in $\gamma$ or not, and we
  compute $\gamma_{\carte} = \gamma\tau_{r}$.
\end{proof}

It follows from Lemmas \ref{lem:phi_map}, \ref{lem:pi_B}, and \ref{lem:rho_B}, that
\begin{equation}\label{eq:conjug-phi}
  \phi_{\carte'} = \tau_{r}\phi_{\carte}\tau_{r}\,.
\end{equation}
In particular, $c(\phi_{\carte}) = c(\phi_{\carte'})$.

We can now state the counterpart of Lemma \ref{lem:cc_A}.
\begin{lemma}\label{lem:cc_B}
  If the unitary type map $\carte$ is connected, then the map
  $\carte'$ has one or two connected components. Furthermore, if $j$
  and $2m$ do not belong to the same cycle in $\gamma$, $\carte'$ is
  connected.
\end{lemma}
\begin{proof}
  The proof is almost the same as for Lemma \ref{lem:cc_A}. Alternatively, we
  can prove this lemma using Proposition \ref{prop:connectedness}.

  Let $k \geq 1$ be the number of orbits of the action of
  $G' \coloneq G(\gamma', \pi', \tau')$ on $[2m]$. We notice that
  $G \coloneq G(\gamma, \pi_{\carte}, \tau_{\carte}) = \langle{\gamma', \pi', \tau_{1}, \ldots, \tau_{r}}\rangle$. In particular, $j$
  and $2m$ always belong to the same orbit of $G$. If $j$ and $2m$ do
  not belong to the same orbit of $G'$, $\tau_{r}$ connects two orbits of
  the action of $G'$, and $G$ has $k-1$ orbits. Conversely, if $j$ and
  $2m$ belong to the same orbit of $G'$, the actions of the two groups
  have the same number $k$ of orbits.

  In particular, if $j$ and $2m$ belongs to the same cycle of $\gamma'$ (or
  equivalently to different cycles of $\gamma$), then the two groups
  have the same number of orbits. We assumed that $\carte$ is
  connected so by Proposition \ref{prop:connectedness}, the action of
  $G$ has one orbit. Thus, the action of $G'$ has one orbit and
  $\carte'$ is connected.

  In the other cases, $G$ has $k$ or $k-1$ orbits and necessarily, $k$
  is 1 or 2.
\end{proof}

Using Lemmas \ref{lem:cc_B}, \ref{lem:rho_B} and \eqref{eq:conjug-phi}, we can compute the
genus $g'$ of $\carte'$ using Euler's formula
\eqref{eq:Euler_for_map}. There are three cases.
\begin{itemize}
  \item If $j$ and $2m$ belong to the same cycle in $\gamma$ and $\carte'$ has
        two connected components, then $c(\gamma') = c(\gamma) +1$ and $g' = g$.
  \item If $j$ and $2m$ belong to the same cycle in $\gamma$ and $\carte'$ is
        connected, then $c(\gamma') = c(\gamma) +1$ and $g'=g-1$.
  \item If $j$ and $2m$ belong to two different cycles in $\gamma$, then
        $\carte'$ is connected (Lemma \ref{lem:cc_B}),
        $c(\gamma') = c(\gamma) -1$, and $g=g'$.
\end{itemize}
Note that in these three cases, $m$ is unchanged.

\subsection{Proof of Theorem \ref{thm:induction_V=0}}
We can now turn to the proof of Theorem \ref{thm:induction_V=0}.
\begin{proof}
  Fix $\gamma=\gamma_{\bm{P}} \in \Sym_{2m}$, $\epsilon = \epsilon_{\bm{P}}$, and $\bm{M} = \bm{{M}}_{\bm{P}}$.
  Assume first that $m = \frac{1}{2}\deg \bm{P} \geq 2$. We decompose the sum
  \begin{equation*}
    \begin{split}
      \M^{(g), N}_{0, l}(P_{1}, \ldots, P_{l}u) &= (-1)^{m}\sum_{\substack{\carte \in \mathfrak{C}(g, [2m], \epsilon, \gamma)\\\carte \text{ connected}}}(-1)^{c(\phi_{\carte})}\tr_{\phi_{\carte}}(\bm{M})
    \end{split}
  \end{equation*}
  in two sums, each corresponding to one of the two cases of the previous
  construction.

  We introduce the set $W^{f}_{2m}$, of monotone walks whose last step
  $\tau$ satisfy $\tau(2m) = 2m$, and the set $W^{c}_{2m}$ of monotone walks
  whose last step $\tau$ satisfy $\tau(2m) \neq 2m$. The functions
  $\ind_{W^{f}_{2m}}$ and $\ind_{W^{c}_{2m}}$ are the indicator
  functions of those sets.

  The sum corresponding to the first case, is thus by the previous
  surgery of Section \ref{sec:cut_maps}
  \begin{equation*}
    \begin{split}
      (-1)^{m}\sum_{\substack{\carte \in \mathfrak{C}(g, m, \epsilon, \gamma)\\\carte \text{
      connected}}}(-1)^{c(\phi_{\carte})}\tr_{\phi_{\carte}}(\bm{M})\ind_{W^{f}_{2m}}(\tau_{\carte})
      &= (-1)^{m}\sum_{\substack{\pi\in\Sym^{(\epsilon)}_{2m}\\\tau\in\mwset^{r(g, m, \gamma, \pi)}(\pi^{(\epsilon)})\\G(\gamma, \pi, \tau)\text{ is transitive}}}(-1)^{c(\gamma\pi^{-1})}\tr_{\gamma\pi^{-1}}(\bm{M})\ind_{W^{f}_{2m}}(\tau)\\
      &= (-1)^{m}\sum_{j\in\epsilon^{-1}(-1)}\sum_{\substack{\pi'\in\Sym^{(\epsilon)}(I_{j})\\\pi = \cycle{j, 2m}\pi'\\\tau\in\mwset^{r(g, m, \gamma, \pi)}(\pi'^{(\epsilon)})\\G(\gamma, \pi, \tau)\text{ is transitive}}}(-1)^{c(\gamma\pi^{-1})}\tr_{\gamma\pi^{-1}}(\bm{M}),\\
    \end{split}
  \end{equation*}
  where $r(g, m, \gamma, \pi) = c(\gamma) + c(\gamma^{-1}\pi^{-1}) - m +2g - 2$ according
  to (\ref{eq:Euler_for_map}), and we used the fact that in the first case $\pi$ can be
  rewritten $\pi'\cycle{j, 2m}$ for some $j \in \epsilon^{-1}(-1)$. Notice that
  the global factor $(-1)^{m}$ will account for a sign $-1$, when
  removing an edge and going from $2m$ white half-edges to $2m-2$
  white half-edges.

  We rewrite this as a sum of four terms, corresponding to the
  different ways of computing the genus, as explained in the last
  section. We interpret the new sums as series
  $\M^{N}_{g', l', 0}(Q_{1}, \ldots, Q_{l'})$, with $l'$ (which corresponds
  to the number of vertices in the new map $\carte'$) and $g'$ (the
  genus of the new map $\carte'$) two integers, and $Q_{1}, \ldots, Q_{l'}$
  monomials either in $\wordsA$ or of degree 0. We introduce the
  notation $\bm{Q} = (Q_{1}, \ldots, Q_{l'})$. These monomials are chosen
  so that the combinatorial data $\gamma'$, and $\epsilon'$ described in the last
  section, and the tuple $\bm{M'}$ of appropriate monomials of degree
  0 is such that $\gamma_{\bm{Q}} = \gamma'$, $\bm{\epsilon}_{\bm{q}} = \epsilon'$, and
  $\bm{M}_{\bm{Q}} = \bm{M}'$. The tuple $\bm{M'}$ is chosen
  differently depending on the subcase, but always so that
  $\tr_{\phi_{\carte}}(\bm{M}) = \tr_{\carte'}(\bm{M'})$ (except for
  subcase 1., see below).

  There are four cases. Let us consider first the terms corresponding
  to subcases 1. and 3., which are
  \begin{enumerate}
    \item if $j$ or $2m$ is a fixed point of $\phi_{\carte}$, then
          $c(\phi_{\carte'}) = c(\phi_{\carte})-1$,
          $l' = c(\gamma_{\carte'}) = c(\gamma_{\carte}) = l$, and $\carte'$ is
          connected. Thus, $g' = g$.
      \setcounter{enumi}{2}
      \item If $j$ and $2m$ belong to the same cycle of $\gamma_{\carte}$, and
        $\carte'$ has two connected components, then
        $l' = c(\gamma_{\carte'}) = c(\gamma_{\carte})+1 = l+1$,
        $c(\phi_{\carte'}) = c(\phi_{\carte})$, and $g' = g$.
  \end{enumerate}

  In those two cases, the map $\carte$ is cut into two maps, with total genus
  equal to $g$. The case 1. corresponds to the degenerate case where one of the
  two maps has no edges, and is reduced to a vertex. We associate to it the
  weigth $\tr(M_{j})$ or $\tr(M_{2m})$.

  Together, these cases account for the term
  \begin{equation*}
    \begin{split}
      -\sum_{P_{l}u=Qu^{-1}Ru}\sum_{\substack{g_{1}+g_{2}=g\\I\subset [l-1]}}\M^{(g_{1}), N}_{0, |I|+1}(\bm{P}|_{I}, Q)\M^{(g_{2}), N}_{0, |I^{c}|+1}(\bm{P}|_{I^{c}}, R).
    \end{split}
  \end{equation*}
  The subcase 1. corresponds to the term for which $Q$ or $R$ in the
  sum is reduced to a monomial of degree $0$, and the subcase 3. to
  the other terms. When cutting the map, we obtain two connected
  components, each containing a vertex corresponding to part of
  $P_{l}$. This correspond to the fact that in the argument of the
  series, $P_{l}$ is replaced by two monomials $Q$ and $R$ such that
  $Qu^{-1}Ru$, and one $u$ and one $u^{-1}$ are removed, corresponding
  to the two removed half-edges.

  Similarly, the subcase
  \begin{enumerate}
    \setcounter{enumi}{1}
      \item If $j$ and $2m$ belong to the same cycle of $\gamma_{\carte}$, and $\carte'$ is
        connected, then $l' = c(\gamma_{\carte'}) = c(\gamma_{\carte})+1 = l+1$,
        $c(\phi_{\carte'}) = c(\phi_{\carte})$, and $g' = g-1$ (by
        Euler formula).
  \end{enumerate}
  corresponds to the term
  \begin{equation*}
    \begin{split}
      -\sum_{P_{l}u = Qu^{-1}Ru}\M^{(g-1), N}_{0, l+1}(P_{1}, \ldots, P_{l-1}, Q, R).
    \end{split}
  \end{equation*}

  The subcase
  \begin{enumerate}
    \setcounter{enumi}{3}
    \item If $j$ and $2m$ belong to two different cycles of
          $\gamma_{\carte}$, then
          $l' = c(\gamma_{\carte'}) = c(\gamma_{\carte})-1 = l-1$,
          $c(\phi_{\carte'}) = c(\phi_{\carte})$, and $g'=g$.
  \end{enumerate}
  corresponds to the term
\begin{equation*}
    \begin{split}
      -\sum_{i=1}^{l-1}\sum_{P_{i}=Qu^{-1}R}\M^{(g), N}_{0, l-1}(P_{1}, \ldots, P_{i-1}, P_{i+1}, \ldots, P_{l-1}, RQP_{l}).
    \end{split}
  \end{equation*}
  Here, two vertices are glued together, corresponding to replacing two
  polynomials in the argument of the series by one: $RQP_{l}$.

  We proceed similarly for the terms that correspond to the second
  case where $\tau(2m) \neq 2m$. The corresponding sum is
  \begin{equation*}
    \begin{split}
      (-1)^{m}&\sum_{\substack{\carte \in \mathfrak{C}(g, [2m], \epsilon, \gamma)\\\carte \text{
      connected}}}(-1)^{c(\phi_{\carte})}\tr_{\phi_{\carte}}(\bm{M})\ind_{W^{c}_{2m}}(\tau_{\carte})\\
                                         &= (-1)^{m}\sum_{\substack{\pi\in\Sym^{(\epsilon)}_{2m}\\\bm{\tau}\in\mwset^{r(g, m, \gamma, \pi)}(\pi^{(\epsilon)})\\G(\gamma, \pi, \bm{\tau})\text{ is transitive}}}(-1)^{c(\gamma\pi^{-1})}\tr_{\gamma\pi^{-1}}(\bm{M})\ind_{W^{c}_{2m}}(\bm{\tau})\\
                                         &= (-1)^{m}\sum_{\substack{j\in\epsilon^{-1}(+1)\\j \neq 2m}}\sum_{\substack{\pi'\in\Sym^{(\epsilon)}_{2m}\\(\tau_{1}, \ldots, \tau_{r-1})\in\mwset^{r-1}(\pi'^{(\epsilon)})\\G(\gamma', \pi', \bm{\tau})\text{ is transitive}}}(-1)^{c(\gamma'\pi'^{-1})}\tr_{\gamma'\pi'^{-1}}(\bm{M}_{\cycle{j, 2m}}),\\
    \end{split}
  \end{equation*}
  where $\gamma' = \cycle{j, 2m}\gamma$,
  $\bm{M}_{\cycle{j, 2m}} = (M_{\cycle{j, 2m}(1)}, M_{\cycle{j, 2m}(2)}, \ldots, M_{\cycle{j, 2m}(2m)})$,
  $\bm{\tau} = (\tau_{1}, \ldots, \tau_{r-1}, \cycle{j, 2m})$, and
  $r = r(g, m, \gamma, \cycle{j, 2m}\pi')$. To go from the second to
  the third line, we replaced $\pi$ by $\pi' = \cycle{j, 2m}\pi$.

  Following the construction from last section, we get three kinds of
  terms corresponding to the three subcases from last section. The
  first subcase is
  \begin{enumerate}
    \item If $j$ and $2m$ belong to the same cycle in $\gamma$ and
          $\carte'$ has two connected components, then
          $c(\gamma') = c(\gamma) +1$ and $g' = g$.
  \end{enumerate}
  It corresponds to the sum
  \begin{equation*}
    \begin{split}
      \sum_{P_{l}u=QuRu}\sum_{\substack{g_{1}+g_{2}=g\\I\subset [2m]}}\M^{(g_{1}), N}_{0, |I|+1}(\bm{P}|_{I}, Qu)\M^{(g_{2}), N}_{0, |I^{c}|+1}(\bm{P}|_{I^{c}}, Ru).
    \end{split}
  \end{equation*}

  The second subcase is:
  \begin{enumerate}
    \setcounter{enumi}{1}
  \item If $j$ and $2m$ belong to the same cycle in $\gamma$ and $\carte'$ is
        connected, then $c(\gamma') = c(\gamma) +1$ and $g'=g-1$.
  \end{enumerate}
  corresponds to
  \begin{equation*}
    \begin{split}
      \sum_{P_{l}u=QuRu}\M^{(g-1), N}_{0, l+1}(P_{1}, \ldots, P_{l-1}, Qu, Ru).
    \end{split}
  \end{equation*}

  Finally, the last subcase is:
  \begin{enumerate}
    \setcounter{enumi}{2}
      \item If $j$ and $2m$ belong to two different cycles in $\gamma$, then
        $\carte'$ is connected (Lemma \ref{lem:cc_B}),
        $c(\gamma') = c(\gamma) -1$, and $g=g'$.
  \end{enumerate}
  corresponds to
   \begin{equation*}
    \begin{split}
      \sum_{i=1}^{l-1}\sum_{P_{i}=QuR}\M^{(g), N}_{0, l-1}(P_{1}, \ldots, P_{i-1}, P_{i+1}, \ldots, P_{l-1}, RQuP_{l}u).
    \end{split}
  \end{equation*}
  Putting all the terms together, we get the induction relation of
  Theorem \ref{thm:induction_V=0}.
\end{proof}

\subsection{Induction relation for the series $\mathcal{M}^{(g), N}_{V, l}$}
\label{sec:exist-form-cumul}

We first prove that the series $\M^{(g), N}_{V, l}(\bm{P})$ exist with
a radius of convergence that depend on $g$, $l$, and $V$. To that end,
we show bounds on the series of maps for $V = 0$ that are a
consequence of Theorem \ref{thm:induction_V=0}. Similar bounds have
been obtained in the Gaussian case in \cite[Lemma
4.3]{maurel-segala_high_2006}.
\begin{prop}\label{prop:bounds-V0}
  Assume that for all $N \geq 1$ and all $1 \leq i \leq p$ we have the bound
  $\|A^{N}_{i}\| \leq 1$ on the deterministic matrices matrices
  $(A_{i}^{N})$. Let $\bm{q} = (q_{1}, \ldots, q_{k}) \in \wordsA_{n}^{k}$ be
  monomials, and $\nu = \max_{1\leq i \leq k}\deg q_{i}$. We introduce the
  $n$-th Catalan number $c_{n} = \frac{1}{n + 1}\binom{2n}{n}$.

  There exists constants $A_{k} > 1$, $B_{k} > 1$, and $C_{k} > 1$
  that depend on $k$, and $D_{k, \nu} > 1$ that depends on $k$ and
  $\nu$ such that for all
  $\bm{P} = P_{1}, \ldots, P_{l} \in \wordsA_{n}^{l}$, and all
  $\bm{n}=(n_{1}, \ldots, n_{k})\in\N^{k}$, we have the bound
  \begin{equation}\label{eq:bound-V0-induction}
    \begin{split}
      \frac{1}{\bm{n}!}|\M^{(g), N}_{0, \sum_{i}n_{i}+l}(\underbrace{q_{1}, \ldots, q_{1}}_{n_{1}\text{
      times}}, \ldots, \underbrace{q_{k}, \ldots, q_{k}}_{n_{k}\text{
      times}}, P_{1}, \ldots, P_{l})| \leq A_{k}^{l(2m+\nu\bm{n})}B_{k}^{-l}C^{g(2m+\nu\bm{n})}D^{\bm{n}}\prod_{i}c_{\deg_{P_{i}}}\prod_{j=1}^{k}c_{n_{j}},
    \end{split}
  \end{equation} where $m = \frac{1}{2}\deg \bm{P}$.

  The constants can be chosen to be
  \begin{equation*}
    \begin{split}
      A_{k} &= C_{k} = \sqrt{6}\pi^{1/4}2^{k+3}\\
      B_{k} &= 3\cdot 4^{k+1}\\
      D_{k, \nu} &= 4k(4\ee^{1/\ee})^{\nu}.
    \end{split}
  \end{equation*}
\end{prop}
The proof is given in Appendix \ref{sec:bounds-sum-maps}. The value of the constants can be
improved. These bounds allow us to prove immediately that the series
$\M^{(g), N}_{V, l}$ (see Definition \ref{def:formal_cumulant}) converge.
\begin{corol}
  Let $\bm{P} = (P_{1}, \ldots, P_{l}) \in \wordsA_{n}^{l}$,
  $\bm{q} = (q_{1}, \ldots, q_{k}) \in \wordsA_{n}^{k}$ and
  $\bm{z} = (z_{1}, \ldots, z_{k})$, and let
  $V = \sum_{i=1}^{k}z_{i}q_{i}$ be a potential.

  As a series in $\bm{z}$, $\M^{(g), N}_{V, l}(P_{1}, \ldots, P_{l})$
  converges absolutely with radius of convergence
  $R_{l, g, V}\geq (4 A_{k}^{l+g}D_{k, \nu})^{-1}$.
\end{corol}

We can now turn to the induction relations. The induction relation
from Theorem \ref{thm:induction_V=0} translates to an induction
relation on the series $\M^{(g), N}_{V, l}$.
\begin{prop}\label{prop:induction-V}
  Let $\bm{P} = (P_{1}, \ldots, P_{l}) \in (\wordsA_{n})^{l}$,
  $\bm{q} = (q_{1}, \ldots, q_{k}) \in (\wordsA_{n})^{k}$ and
  $\bm{z} = (z_{1}, \ldots, z_{k})$, and let
  $V = \sum_{i=1}^{k}z_{i}q_{i}$ be a potential. Assume that for all
  $1 \leq i \leq k$, $|z_{i}| < R_{l, g, V}$.

  For all $1 \leq i \leq n$, we have the equation
  \begin{equation}\label{eq:family_DS}
    \begin{split}
      \sum_{\substack{g_{1}+g_{2}=g\\I \subset [l - 1]}}&\M^{(g_{1}), N}_{V, |I|+1}\otimes\M^{(g_{2}), N}_{V, |I^{c}|+1}(\bm{P}_{I}\otimes \bm{P}_{I^{c}}\sharp \partial P_{l})
      + \M^{(g), N}_{V, l}(P_{1}\otimes \cdots \otimes P_{l-1}\otimes (\D V)P_{l})\\
      =&- \M^{(g-1), N}_{V, l+1}(P_{1}\otimes \cdots \otimes P_{l-1}\otimes \partial P_{l})\\
      &- \sum_{j=1}^{l-1}\M^{(g), N}_{V, l-1}(P_{1}\otimes \cdots \otimes P_{j-1}\otimes P_{j+1}\otimes \cdots \otimes P_{l-1}\otimes (\D P_{j})P_{l})\,.\\
    \end{split}
  \end{equation}
\end{prop}
\begin{proof}
  We sum on $\bm{n}\in\N^{k}$ the induction relations of Proposition
  \ref{prop:induction-multimatrix} for
  \begin{equation*}
    \begin{split}
      \M^{(g), N}_{0, \sum_{i}n_{i}+l}(\underbrace{q_{1}, \ldots, q_{1}}_{n_{1}\text{
      times}}, \ldots, \underbrace{q_{k}, \ldots, q_{k}}_{n_{k}\text{
      times}}, P_{1}, \ldots, P_{l}),
    \end{split}
  \end{equation*}
  times $\frac{\bm{z}^{\bm{n}}}{\bm{n}!}$.
\end{proof}

\section{The multimatrix case}\label{sec:multimatrix}
Up to now, we have only considered integrals involving one
Haar-distributed matrix $U^{N}$. The results obtained so far can be
extended in a straightforward way to the case where we have $n \geq 1$
independent Haar-distributed matrices $U^{N}_{1}, \ldots, U^{N}_{n}$. The
polynomials we consider in the sequel are the non-commutative
polynomials in $u_{i}, u^{-1}_{i}$, for $1 \leq i \leq n$ and
$a_{j}, a_{j}^{*}$ for $1 \leq j \leq p$. We denote this $*$-algebra by
$\algA_{n}$. Notice that $\algA = \algA_{1}$.

\subsection{Weingarten calculus}
As previously, we will consider a subset of monomials of $\algA_{n}$,
as the quantity we consider are multilinear functions which are
tracial in each of their arguments. We define $\wordsA_{n}$ the set of
monomials of $\algA_{n}$ of the form
\begin{equation}\label{eq:different_forms}
  \begin{split}
    P &= M_{1}u^{\epsilon_{1}}_{t_{1}}M_{2}u^{\epsilon_{2}}_{t_{2}}\cdots M_{d}u^{\epsilon_{d}}_{t_{d}}\\
  \end{split}
\end{equation}
where $\epsilon \colon [d]\to\{+1, -1\}$, $t\colon [d]\to[n]$, and $\bm{M} = (M_{1}, \ldots, M_{d})$ is a
$d$-uple of monomials $M_{j}\in \algA_{n}$, each of them being empty or
a word in $a_{1}, a_{1}^{*}, \ldots, a_{p}, a_{p}^{*}$.

We define for a tuple $\bm{P} = (P_{1}, \ldots, P_{l})$ the tuples
$\bm{\epsilon}_{\bm{P}}, \bm{t}_{\bm{P}}$, $\bm{M}_{\bm{P}}$ obtained by
concatenating the tuples corresponding to each polynomial
$P_{j}, 1 \leq j \leq l$. We also define for $1 \leq i \leq n$,
$\bm{\epsilon}_{\bm{P}, i} = \bm{\epsilon}|_{\bm{t}_{\bm{P}}^{-1}(i)}$, i.e. the
tuple which encodes the exponents of the variables $u_{i}$ only. We
define the degree with respect to $u_{i}$ of a monomial $P$,
$\deg_{i} P$ as the number of occurrence of $u_{i}$ or $u^{-1}_{i}$ in
$P$. The total degree of a tuple is
$\deg_{i} \bm{P} = \sum_{j=1}^{l}\deg_{i} P_{j}$. We define
$\deg P = \sum_{i = 1}^{n}\deg_{i} P$, and
$\deg \bm{P} = \sum_{j=1}^{l}\deg P_{j}$.

Property \ref{prop:comput_moment} is generalized as follows.
\begin{definition}
  Let $\bm{P} = (P_{1}, \ldots, P_{l}) \in (\wordsA_{n})^{l}$.
  We define the permutation
  \begin{equation*}
    \begin{split}
      \gamma_{\bm{P}} = \cycle{1, \ldots, \deg P_{1}}\cdots\cycle{\sum_{j=1}^{l-1}\deg P_{j} + 1, \ldots, \deg \bm{P}}.
    \end{split}
  \end{equation*}
\end{definition}

\begin{definition}\label{def:moment-multimatrix}
  We introduce the moment with respect to the Haar measure in the
  multimatrix case
  \begin{equation*}
    \begin{split}
      \alpha_{\bm{U}, 0, l}^{N}(P_{1}, \ldots, P_{l}) = \E \left[\Tr(P_{1})\cdots \Tr(P_{l})\right],
    \end{split}
  \end{equation*}
  where the expectation is under the product Haar measure
  $\dd U_{1}^{N}\cdots \dd U_{n}^{N}$.
\end{definition}

\begin{prop}\label{prop:multi_U_comput}
  Let $\bm{P} = (P_{1}, \ldots, P_{l}) \in (\wordsA_{n})^{l}$ and
  $J_{i} = \bm{t}_{\bm{P}}^{-1}(i)$.

  We have
  \begin{equation}
    \begin{split}
      \alpha^{N}_{\bm{U}, 0, l}(P_{1}, \ldots, P_{l}) &= \sum_{\substack{\pi_{1}\in\Sym^{(\bm{\epsilon}_{\bm{P}, 1})}(J_{1})\\\pi_{2}\in\Sym^{(\bm{\epsilon}_{\bm{P}, 2})}(J_{2})\\\cdots\\\pi_{n}\in\Sym^{(\bm{\epsilon}_{\bm{P}, n})}(J_{n})}}\left(\prod_{i=1}^{k}\Weingarten_{N}(\pi_{i}^{(\epsilon_{\bm{P}}^{(i)})})\right)\Tr_{\gamma_{\bm{P}}\pi_{1}^{-1}\cdots \pi_{n}^{-1}}(\bm{M}_{\bm{P}}).
    \end{split}
  \end{equation}
\end{prop}
This Proposition is obtained by applying the following Lemma $n$ times.
\begin{lemma}\label{lem:help_multi_U_comput}
  Let $\bm{M}_{\bm{P}} = (M_{1}, \ldots, M_{\deg \bm{P}})$
  Let $\tilde{\bm{M}} = (\tilde{M}_{i}, 1 \leq i \leq \deg \bm{P})$, defined by
  $\tilde{M}_{i} = M_{i}$ if $t_{i} = 1$, and
  $\tilde{M}_{i} = M_{i}u^{\epsilon_{\bm{P}}(i)}_{t_{i}}$ otherwise. Then, we have the
  expectation with respect to $U^{N}_{1}$ only
  \begin{equation*}
    \begin{split}
      \E_{U^{N}_{1}}\left[\Tr(P_{1})\cdots \Tr(P_{l})\right]
      &= \sum_{\pi_{1}\in\Sym^{(\epsilon_{\bm{P}, 1})}(J_{1})}\Weingarten_{N}(\pi_{1}^{(\epsilon_{\bm{P}, 1})})\Tr_{\gamma_{\bm{P}}\pi_{1}^{-1}}(\tilde{\bm{M}}).
    \end{split}
  \end{equation*}
\end{lemma}
\begin{proof}[Proof of Lemma \ref{lem:help_multi_U_comput}]
  Let $I \subset [l]$ be the subset of indices $i$ such that $P_{i}$ contains a
  letter $u_{1}$ or $u_{1}^{*}$. Denote by $c_{i}$ the cycle in $\gamma$
  that corresponds to $P_{i}$, i.e.
  \begin{equation*}
    c_{i} = \cycle{\sum_{j=1}^{i-1}\deg P_{j} + 1, \cdots, \sum_{j=1}^{i-1}\deg P_{j} + \deg P_{i}}\,.
  \end{equation*}
  We have
  \begin{equation}\label{eq:fact-multiU}
    \E_{U^{N}_{1}}\left[\Tr(P_{1})\cdots \Tr(P_{l})\right] = \left( \prod_{i\notin I} \Tr_{c_{i}}(P_{i})\right)\E \left[ \prod_{i \in I}\Tr(P_{i}) \right]\,.
  \end{equation}
  Furthermore, if we set $S = \bigcup_{i \notin I}\Supp c_{i}$ and
  $\gamma'' = \left( \prod_{i}c_{i} \right)\vert_{S}$, we can rewrite the terms
  \begin{equation*}
    \left( \prod_{i\notin I} \Tr_{c_{i}}(P_{i})\right) = \Tr_{\gamma''}(\tilde{\bm{M}}\vert_{S})\,.
  \end{equation*}

  Considering only the second term in the right side of
  \eqref{eq:fact-multiU} and using the cyclic property of the trace,
  we can assume that the last factor of each polynomial $P_{i}$ is a
  $u_{1}$ or a $u_{1}^{*}$. Let
  $J_{1} = t^{-1}(1) = \{p_{1} < p_{2} < \ldots < p_{q}\}$. We let
  $\gamma' = \Tr(\gamma_{\bm{P}}; J_{1})$. Proposition \ref{prop:comput_moment}, shows that
  \begin{equation*}
    \begin{split}
      \E_{U^{N}_{1}}\left[\Tr(P_{1})\cdots \Tr(P_{l})\right] &= \sum_{\pi_{1}\in\Sym^{(\epsilon_{\bm{P}, 1})}(J_{1})}\Weingarten_{N}(\pi_{1}^{(\epsilon_{\bm{P}, 1})})\Tr_{\gamma'\pi_{1}^{-1}}({\bm{M}'}),
    \end{split}
  \end{equation*}
  where $\bm{M}'= (M_{i}', i\in J_{1})$ defined by
  $M_{i}' = M_{p_{i-1}+1}u^{\epsilon_{\bm{P}}(p_{i-1}+1)}_{t(p_{i-1}+1)}M_{p_{i-1}+2}u^{\epsilon_{\bm{P}}(p_{i-1}+2)}_{t(p_{i-1}+2)}\cdots M_{p_{i}-1}u^{\epsilon_{\bm{P}}(p_{i}-1)}_{t(p_{i}-1)}M_{p_{i}}$
  (with the convention $p_{0} = 0$).

  This is equal to
  \begin{equation*}
    \begin{split}
      \E_{U^{N}_{1}}\left[\Tr(P_{1})\cdots \Tr(P_{l})\right] &= \sum_{\pi_{1}\in\Sym^{(\epsilon_{\bm{P}, 1})}(J_{1})}\Weingarten_{N}(\pi_{1}^{(\epsilon_{\bm{P}, 1})})\Tr_{\gamma_{\bm{P}}\pi_{1}^{-1}}(\tilde{\bm{M}}).
    \end{split}
  \end{equation*}
\end{proof}

\subsection{Multicolored maps of unitary type}
We now generalize the notion of a map of unitary type to address the
multimatrix case.

\begin{definition}
  Let $I$ be a finite subset of $\N^{*}$. A \textbf{multicolored map
    of unitary type} with $n$ colors, with labels in $I$, and with
  $r_{i}$ vertices of color $i$ for $1 \leq i \leq n$, is an oriented
  map with vertices colored in white or in one of $n$ colors, and
  colored half-edges which can be of any of the $n$ colors such that
  \begin{itemize}
    \item there are $r_{i}$ vertices of color $i$ for $1 \leq i \leq n$, which
          are alternated of degree 4 and numbered from $1$ to $r_{i}$;
    \item the half-edges connected to a vertex of color $i$ (which is
          not white) are of color $i$ as well;
    \item there are $|I|$ half-edges that are connected to white
          vertices. Each element of $I$ labels exactly one of these
          half-edges;
    \item each half-edges connected to a white vertex is colored in one of the
          $n$ colors;
    \item each edge is composed of two half-edges of the same color;
    \item if an oriented edge connects the vertex of color $i$ numbered $l_{1}$
          to the vertex of color $i$ numbered $l_{2}$ then $l_{1} < l_{2}$.
  \end{itemize}
\end{definition}

\begin{figure}[htbp]
  \centering
  \includegraphics[width=0.4\textwidth]{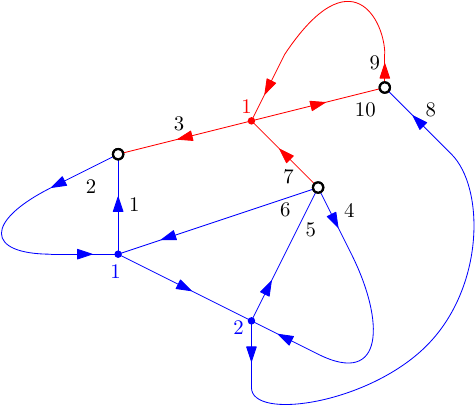}
  \caption{\label{fig:multimatrix}A muticolored map of unitary type with two colors. The
    integers $3, 7, 10, 9$ label red half-edges, and the other
    elements of $[10]$ label blue half-edges.}
\end{figure}

\begin{remark}\label{rem:erase-color}
  Notice that we can erase the colors of a multicolored unitary map to
  obtain a (monocolored) map of unitary type. To do so, we proceed as
  follows:
  \begin{itemize}
    \item each colored half-edge connected to a white vertex becomes a white half-edge;
    \item each colored half-edge connected to a colored vertex becomes a black half-edge;
    \item each colored vertex becomes a black vertex.
  \end{itemize}

  The resulting map is a map of unitary type with $\sum_{i}r_{i}$
  black vertices and labels in $I$.
\end{remark}

As in Section \ref{sec:maps_unitary}, we define for a multicolored map $\carte$ of
unitary type, with $n$ colors, the following permutations. We
construct permutations $\gamma_{\carte}$, $\pi_{\carte}$, and $\phi_{\carte}$,
and the tuple $\bm{\epsilon}_{\carte}$, as for a monocolored map of unitary
type. If the $i$-th labelled half-edge is of color $j$, we set
$\bm{t}_{\carte}(i) = j$. We then define
$J_{\carte,i} = \bm{t}_{\carte}^{-1}(i)$. We set
$\bm{\epsilon}_{\carte, i} = \bm{\epsilon}_{\carte}|_{J_{\carte, i}}$ for all
$1 \leq i \leq k$.

For each color $i \in [n]$, we consider the edges of this color. We then
define as previously a permutation
$\pi_{\carte, i}\in\Sym^{(\bm{\epsilon}_{\carte, i})}(J_{\carte, i})$ describing
these edges and the vertices of color $i$. Finally, if we consider the
vertices of color $i$, we can associate to the $j$-th vertex of color
$i$ the transposition $\tau_{i, j}$ as previously. We set
$\tau_{\carte} = (\tau_{i, j}, i\in [n], j \in [r_{i}])$. Notice that by construction, we
have $\pi_{\carte} = \pi_{\carte, 1}\pi_{\carte, 2}\cdots \pi_{\carte, n}$.

\begin{definition}\label{def:set-colored-maps}
  Let $n\geq 1$ and $g \geq 0$ be integers, and $I$ be a finite subset
  of the positive integers. Let
  $\bm{r} = (r_{1}, \ldots, r_{n}) \in \N^{n}$, $\gamma \in \Sym(I)$,
  $\bm{t}\colon I \to[n]$ and $\bm{\epsilon}\colon I \to \{\pm 1\}$.

  We denote by
  $\mathfrak{C}^{\bm{r}}(I, \bm{\epsilon}, \gamma, \bm{t})$ the set of
  multicolored maps of unitary type $\carte$ with $n$ colors, with
  label set $I$, and with $r_{i}$ vertices of color $i$, for
  $1 \leq i \leq n$, such that $\gamma_{\carte} = \gamma$,
  $\bm{\epsilon}_{\carte} = \bm{\epsilon}$, and
  $\bm{t}_{\carte} = \bm{t}$.

  We denote by $\mathfrak{C}(g, I, \bm{\epsilon}, \gamma, \bm{t})$ the set
  of multicolored maps of unitary type $\carte$ with $n$ colors, with
  label set $I$, and of genus $g$, such that $\gamma_{\carte} = \gamma$,
  $\bm{\epsilon}_{\carte} = \bm{\epsilon}$, and $\bm{t}_{\carte} = \bm{t}$.
\end{definition}

We then have the analog of Theorem \ref{thm:permutational_model}
\begin{theorem}\label{thm:perm_model_multi_U}
  Let $n \geq 1$ be an integer, and $I$ be a finite subset of the
  positive integers. Let $\bm{r} = (r_{1}, \ldots, r_{n}) \in \N^{n}$,
  $\gamma \in \Sym(I)$, $\bm{t}\colon I \to[n]$ and
  $\bm{\epsilon}\colon I \to \{\pm 1\}$. Define $J_{i} = \bm{t}^{-1}(i)$
  and $\bm{\epsilon}_{i} = \bm{\epsilon}|_{J_{i}}$ for $1 \leq i \leq n$.

  The previous construction gives a bijection between $\mathfrak{C}^{\bm{r}}(I,\bm{\epsilon}, \gamma, \bm{t})$ and
  \begin{equation*}
    \begin{split}
      \bigcup_{\pi_{1}\in\Sym^{(\epsilon_{1})}(J_{1}), \ldots, \pi_{n}\in\Sym^{(\epsilon_{n})}(J_{n})}\prod_{i=1}^{n}\{\pi_{i}\}\times\mwset^{r_{i}}(\pi_{i}^{(\epsilon_{i})}).
    \end{split}
  \end{equation*}
\end{theorem}
\begin{proof}
  The proof is very similar to the one of Theorem
  \ref{thm:permutational_model}. By considering each color, we prove
  that the construction does give a map
  $\mathfrak{C}^{\bm{r}}(I, \bm{\epsilon}, \gamma, \bm{t})\to\bigcup_{\pi_{1}\in\Sym^{(\epsilon_{1})}(J_{1}), \ldots, \pi_{n}\in\Sym^{(\epsilon_{n})}(J_{n})}\prod_{i=1}^{n}\{\pi_{i}\}\times\mwset^{r_{i}}(\pi_{i}^{(\epsilon_{i})})$.

  We can construct its inverse exactly as in the proof for the case
  with one unitary matrix, by constructing the edges for the color
  $1$, then for the color $2$, etc. More precisely, we first consider
  the color $1$. At this step, we leave untouched the half-edges of
  color $2, \ldots, n$. We construct the incidence relation for the edges
  of color $1$ using the data of $(\tau_{1, j})_{j \in [r_{1}]}$ and $\pi_{1}$. When
  this is finished, all the half-edges of color $1$ are part of some
  edge.

  We then turn to the half-edges of color $2$. Using the data of
  $(\tau_{2, j})_{j \in [r_{2}]}$ and $\pi_{2}$, we construct the incidence relation
  for the edges of color $2$. We do this until we reach color $n$ and
  obtain a multicolored map of unitary type.
\end{proof}

It follows directly by erasing the colors (see Remark
\ref{rem:erase-color}) and Proposition \ref{prop:connectedness} that we have the following
proposition.
\begin{prop}
  Let $\carte$ be a multicolored map of unitary type with $n$ colors.
  The map $\carte$ is connected if and only if the group
  $\langle{\gamma_{\carte}, \pi_{\carte, 1}, \ldots, \pi_{\carte, n}, \tau_{i, j}, 1\leq i \leq n, 1 \leq j \leq r_{i}}\rangle$
  is transitive, with
  $\tau_{\carte} = (\tau_{i, j}, i\in[n], j \in
  [r_{i}])$.
\end{prop}


Using Proposition \ref{prop:dvpt_weingarten},
\ref{prop:multi_U_comput}, and \ref{thm:perm_model_multi_U}, we can
compute the moments with no potential (for
$N \geq m = \frac{1}{2}\deg \bm{P}$). Let
$\bm{P} = (P_{1}, \ldots, P_{l}) \in (\wordsA_{n})^{l}$, we have
\begin{equation}
  \begin{split}
    \alpha_{\bm{U}, 0, l}^{N}(P_{1}, \ldots, P_{l})
    &= N^{-m}\sum_{\substack{\pi_{1}\in\Sym^{(\epsilon_{\bm{P}, 1})}(J_{1})\\\pi_{2}\in\Sym^{(\epsilon_{\bm{P}, 2})}(J_{2})\\\cdots\\\pi_{n}\in\Sym^{(\epsilon_{\bm{P}, n})}(J_{n})}}\sum_{r_{1}, \ldots, r_{n}\geq 0}\left(\frac{-1}{N}\right)^{\bm{r}}\prod_{i=1}^{n}\mwalks^{r_{i}}(\pi_{i}^{(\epsilon_{i})})\Tr_{\gamma_{\bm{P}}\pi_{1}^{-1}\cdots \pi_{n}^{-1}}(\bm{M}_{\bm{P}})\\
    &= N^{-m}\sum_{\bm{r}\in\N^{k}}\sum_{\substack{\carte \in \mathfrak{C}(\bm{r}}_{[2m], \bm{\epsilon}_{\bm{P}}, \gamma_{\bm{P}}, \bm{t})}}\left(\frac{-1}{N}\right)^{\bm{r}}\Tr_{\gamma\pi_{1}^{-1}\cdots \pi_{n}^{-1}}(\bm{M}_{\bm{P}}),\\
  \end{split}
\end{equation}
where we use the notation $x^{\bm{r}} = x^{\sum_{i}r_{i}}$, for any $x\in \Real$.

We then compute the cumulants for no potential, when $N \geq m$,
\begin{equation}
  \begin{split}
    \W^{N}_{\bm{U^{N}}, 0, l}(P_{1}, \ldots, P_{l})
    &=N^{-m}\sum_{\bm{r}\in\N^{n}}\sum_{\substack{\carte \in \mathfrak{C}^{\bm{r}}([2m], \bm{\epsilon}_{\bm{P}}, \gamma_{\bm{P}}, \bm{t})\\\carte\text{ connected }}}\left(\frac{-1}{N}\right)^{\bm{r}}\Tr_{\phi_{\carte}}(\bm{M}_{\bm{P}}).\\
  \end{split}
\end{equation}

We now rewrite this sum using the genus of the maps rather than the
number of colored vertices. In this context, the Euler formula becomes
\begin{equation*}
  \begin{split}
    2 - 2g_{\carte} = c(\gamma_{\carte}) + c(\phi_{\carte}) - m - \sum_{i=1}^{k}r_{i}.
  \end{split}
\end{equation*}
We thus get the renormalized cumulant $\tilde{\W}^{N}_{\bm{U^{N}}, 0, l}$
\begin{equation}
  \begin{split}
    \tilde{\W}^{N}_{\bm{U^{N}}, 0, l}(P_{1}, \ldots, P_{l})
    &= N^{l-2}\W^{N}_{\bm{U^{N}}, 0, l}(P_{1}, \ldots, P_{l})\\
    &=(-1)^{m+l}\sum_{g \geq 0}\frac{1}{N^{2g}}\sum_{\substack{\carte \in \mathfrak{C}(g, [2m], \bm{\epsilon}_{\bm{P}}, \gamma_{\bm{P}}, \bm{t})\\\carte\text{ connected }}}(-1)^{c(\phi_{\carte})}\tr_{\phi_{\carte}}(\bm{M}_{\bm{P}}).\\
  \end{split}
\end{equation}
Let us recall the relevant notation. Here $l$ is the number of
monomials or the number of white vertices. In particular, we have
$c(\gamma_{\bm{P}}) = l$. The set of maps
$\mathfrak{C}(g, [2m], \bm{\epsilon}_{\bm{P}}, \gamma_{\bm{P}}, \bm{t})$
was introduced in Definition \ref{def:set-colored-maps}. As in Section
\ref{sec:expr-moments}, the coefficients in the $1/N^{2}$ expansion
are sums of maps, with a weight determined by the permutation of the
faces $\phi_{\carte}$ and the tuple of matrices $\bm{M}$.

The term of order $2g$ is then
\begin{equation*}
  \begin{split}
    \M^{(g), N}_{0, l}(P_{1}, \ldots, P_{l})
    &=(-1)^{m+l}\sum_{\substack{\carte \in \mathfrak{C}(g, [2m], \bm{\epsilon}_{\bm{P}}, \gamma_{\bm{P}}, \bm{t})\\\carte\text{ connected }}}(-1)^{c(\phi_{\carte})}\tr_{\phi_{\carte}}(\bm{M}_{\bm{P}}).
  \end{split}
\end{equation*}

We then define the formal cumulant as
\begin{equation*}
  \begin{split}
        \M^{(g), N}_{V, l}(P_{1}, \ldots, P_{l}) &=\sum_{\bm{n}\in\N^{k}}\frac{\bm{z}^{\bm{n}}}{\bm{n}!}\M^{(g), N}_{0, l}(\underbrace{q_{1}, \ldots, q_{1}}_{n_{1}\text{
        times }}, \ldots, \underbrace{q_{k}, \ldots, q_{k}}_{n_{k}\text{ times
      }}, P_{1}, \ldots, P_{l}),
  \end{split}
\end{equation*}
as previously.

\subsection{Induction relation}
We will now deduce from the relations obtained in Section \ref{sec:induction}
similar relations in the multimatrix case.
\begin{prop}\label{prop:induction-multimatrix}
  Let $\bm{P} = (P_{1}, \ldots, P_{l}) \in (\wordsA_{n})^{l}$, $i\in[n]$
  and $g \geq 1$.

  If $\frac{1}{2}\deg_{i} \bm{P} \geq 2$, then we have the equation
  \begin{equation}\label{eq:family_DS_multi_V=0}
    \begin{split}
      \sum_{\substack{g_{1}+g_{2}=g\\I \subset [l - 1]}}&\M^{(g_{1}), N}_{V, |I|+1}\otimes\M^{(g_{2}), N}_{V, |I^{c}|+1}(\bm{P}_{I}\otimes \bm{P}_{I^{c}}\# \partial_{i} P_{l})
                                                        + \M^{(g), N}_{V, l}(P_{1}\otimes \cdots \otimes P_{l-1}\otimes (\D_{i}V)P_{l})\\
                                                        =& -\M^{(g-1), N}_{V, l+1}(P_{1}\otimes \cdots \otimes P_{l-1}\otimes \partial_{i} P_{l})\\
                                                        &-\sum_{j=1}^{l-1}\M^{(g), N}_{V, l-1}(P_{1}\otimes \cdots \otimes P_{j-1}\otimes P_{j+1}\otimes \cdots \otimes P_{l-1}\otimes (\D_{i}P_{j})P_{l}).\\
    \end{split}
  \end{equation}
  Here $\partial_{i}$ and $\D_{i}$ are the non-commutative and cyclic
  derivative with respect to $u_{i}$, for $i \in [n]$.
\end{prop}
This Proposition is proved as in Section \ref{sec:induction}. If no polynomial of
$\bm{P}$ contains a $u_{i}$ then the equation is trivial. Thus, we can
assume by symmetry that $\deg_{i} P_{l} \geq 1$. We cut the maps from the
sum $\W_{0, l}^{(g), N}(\bm{P})$ as in Section \ref{sec:cut_maps}. Notice that in
this construction, we only modify edges of the color $i$ so we can use
the exact same arguments. We thus obtain the wanted equation.

\section{The Dyson-Schwinger equation and the topological expansion}
\label{sec:dyson-schw-equat}

We now work in the multi-matrix setting. All the maps involved will be
multicolored maps. The induction equations obtained in Section \ref{sec:induction} are
related to the Dyson-Schwinger equations for unitary matrices. In this
section, we introduce the Dyson-Schwinger lattice of equations for the
renormalized cumulants $\tW^{N}_{V, l} = N^{l-2}\W^{N}_{V, l}$.
Together with the induction relations derived in Section \ref{sec:induction}, they
allow us to show that the renormalized cumulants $\tW^{N}_{V, l}$
admit an asymptotic topological expansion as $N \to \infty$. The methods used
in this section are heavily inspired from \cite{guionnet_asymptotics_2015}.

\subsection{Scalar product and parametric norms on $\algA_{n}$}
\label{sec:notation}

Following \cite{guionnet_asymptotics_2015}, we introduce some useful notions. The vector space
$\algA$ -- the algebra of noncommutative polynomials -- admits a
countable basis, which is the set $\hat{\wordsA}_{n}$ of all words in
the letters
$u_{1}, u_{1}^{*}, \ldots, u_{n}, u_{n}^{*}, a_{1}, a_{1}^{*}, \ldots, a_{p}, a_{p}^{*}$,
plus the empty word $1$. Notice that this set contains $\wordsA_{n}$,
the set of such words that finishes by a $u_{i}$ or a $u_{i}^{*}$. Let
\(\langle{\cdot \,, \cdot}\rangle\) be the scalar product that makes this basis orthonormal. In
particular, $\algB^{\perp}$ is the algebra generated by the polynomials
with no constant term, i.e. without factors $u_{i}$ or $u^{*}_{i}$.

\begin{definition}\label{def:parametric-norm}
  Let \(\xi \geq 1\). The \(\xi\)-norm is \(\|\cdot\|_{\xi}\) defined by
  \begin{equation*}
    \begin{split}
      \|P\|_{\xi} = \sum_{Q \in \hat{\wordsA}}|\langle{P, Q}\rangle| \xi^{\deg Q},
    \end{split}
  \end{equation*}
  for $P \in \algA$.

  We write $\algB^{\perp}_{\xi}$ the completion of the algebra
  $\algB^{\perp}$ in the $\xi$-norm $\|\cdot\|_{\xi}$.
\end{definition}
This norm is a deformation of the $\ell^{1}$ norm that takes into
account the degree of the basis monomials. The usual $\ell^{1}$ norm
will in many case not be the appropriate norm, as the effect of many
operators we consider in the sequel depends on the degree of the
monomial it is applied to.

\begin{ex}
  This norm is deformation of the $\ell^{1}$ norm, recovered when
  considering the \(1\)-norm. For instance the \(1\)-norm of the
  potential we consider is
  \begin{equation*}
    \begin{split}
      \|V\|_{1} = \sum_{i=1}^{k}|z_{i}|.
    \end{split}
  \end{equation*}
\end{ex}

This notion of norm allows us to define the parametric \(\xi\)-norm of
a linear operator or form.
\begin{definition}
  Let $T$ be an operator on $\algA$ and $\xi, \xi' \geq 1$. Its \((\xi, \xi')\)-norm is
  \begin{equation*}
    \begin{split}
      \|T\|_{\xi, \xi'} = \sup_{P\in \algA}\frac{\|T P\|_{\xi'}}{\|P\|_{\xi}}.
    \end{split}
  \end{equation*}
  When $\xi = \xi'$ we write $\|T\|_{\xi} = \|T\|_{\xi, \xi}$.

  Similarly, let \(\tau\colon \algA \to \C\) be a linear form. Its
  \(\xi\)-norm is
  \begin{equation*}
    \begin{split}
      \|\tau\|_{\xi} = \sup_{P \in \algA}\frac{|\tau(P)|}{\|P\|_{\xi}}.
    \end{split}
  \end{equation*}
\end{definition}

A particularly important sort of linear forms are tracial states.
\begin{definition}\label{def:tracial_state}
  Let $\mathcal{C}$ be a unital $*$-algebra. A \textbf{tracial state}
  on $\mathcal{C}$ is a linear form $\tau \colon \mathcal{C} \to \C$
  such that for any $P, Q \in \mathcal{C}$, we have
  \begin{itemize}
    \item $\tau(\bm{1}) = 1$, where $\bm{1}$ is the empty word;
    \item $\tau(PQ) = \tau(QP)$;
    \item $\tau(P P^{*}) \geq 0$.
  \end{itemize}
\end{definition}

\begin{remark}
  The normalized trace \(\tr\) is a tracial state on $\algB$. Under
  Hypothesis \ref{hyp:bound}, the Cauchy-Schwarz inequality implies that
  $\|\tr\|_{1} \leq 1$. Furthermore, we have $\tr(\bm{1}) = \tr(\Id) = 1$
  where $\Id$ is the identity matrix, and thus
  \begin{equation*}
      \|\tr\|_{1} = 1.
  \end{equation*}
  Assuming furthermore Hypothesis \ref{hyp:real}, we have that $\tW^{N}_{V,1}$ is
  a tracial state on $\algA_{n}$, with $\|\tW^{N}_{V, 1}\|_{1} = 1$.
\end{remark}

\subsection{The Dyson-Schwinger equations for the unitary matrices}
\label{sec:dyson-schw-equat-unitary}

Let \(\sigma\) be a tracial state on \(\algB\). A tracial state \(\mu\)
on \(\algA\) is a solution to the Dyson-Schwinger problem with initial
value \(\sigma\) if for all $P \in \algA$,
\begin{equation}\label{eq:DS-problem}
  \begin{cases}
    \mu\otimes\mu(\partial_{i} P) + \mu(\D_{i} V \cdot P) &= 0, \text{ for } 1 \leq i \leq n\\
    \mu\vert_{\algB} &= \sigma\,,
  \end{cases}
\end{equation}
where $\partial_{i}$ and $\D_{i}$ are the non-commutative derivative
and cyclic derivative with respect to $u_{i}$, see Definitions
\ref{def:nc_derivative} and \ref{def:cyclic_derivative}.

It has been shown in \cite{collins_asymptotics_2009} that there exists a solution to this problem
when \(\Tr V = \Tr V^{*}\) (which implies that $\Tr V$ is real), and
that the solution is unique for a potential $V$ small enough (i.e.
$\sum_{i=1}^{k}|z_{i}| < \epsilon$ for some $\epsilon > 0$). Notice that for all
$N \geq 1$, $\M^{(0), N}_{V, 1}$ is a solution to \eqref{eq:DS-problem} with $\sigma = \tr_{N}$.
In \cite{guionnet_asymptotics_2015}, a family of equations that generalize \eqref{eq:DS-problem} was studied. The
renormalized cumulants \(\tilde{\W}^{N}_{V, l}\) are solution to these
equations. We reproduce them here.
\begin{prop}[{\cite[Proposition 20]{guionnet_asymptotics_2015}}]\label{prop:DS-first}
  Assume Hypothesis \ref{hyp:real}. The renormalized
  cumulants \(\{\tW_{V, l}^{N}\}_{l \geq 1}\) satisfy the equation
  \begin{equation}\label{eq:DS-first}
    \begin{split}
      \sum_{I \subset [l-1]}&\tilde{\W}^{N}_{V, |I|+1}\otimes\tilde{\W}^{N}_{V, |I^{c}|+1}(\bm{P}_{I}\otimes\bm{P}_{I^{c}}\#\partial_{i} P_{l}) + \tilde{\W}^{N}_{V, l}\left(\bm{P}_{[l-1]}\otimes (\D_{i}V\cdot P_{l})\right)\\
      &= - \sum_{j=1}^{l-1}\tilde{\W}^{N}_{V, l-1}(P_{1} \otimes \cdots \otimes \check{P_{j}} \otimes \cdots \otimes P_{l-1}\otimes(\D_{i}P_{j}\cdot P_{l})) - \frac{1}{N^{2}}\tilde{\W}^{N}_{V, l+1}(\bm{P}_{[l-1]}\otimes \partial_{i}P_{l})\,,
    \end{split}
  \end{equation}
  where $\check{P_{j}}$ means that the factor $P_{j}$ is omitted.
\end{prop}
The series of maps $\M^{(g), N}_{V, l}$ satisfy similar equations (see
\eqref{eq:family_DS}).

\subsection{Radius of convergence of the series $\M^{(g), N}_{V, l}$}
\label{sec:radi-conv-seri}

Before giving the proof of Theorem \ref{thm:main}, we show that all the terms
$\M^{(g), N}_{V, l}$ have a radius of convergence greater than some
$R_{V} > 0$. We can apply the gradient trick from \cite{guionnet_asymptotics_2015} that we explain
in Appendix \ref{sec:gradient-trick} to the equations from Proposition \ref{prop:induction-multimatrix}. To do so, we
introduce some notation, motivated in Appendix \ref{sec:gradient-trick}.
\begin{definition}\label{def:operators-from-appendix}
  Let $P \in \hat{\wordsA}_{n}$ be a monomial, we define
  \begin{equation*}
    \begin{split}
      \Delta_{i} P
      &= \sum_{P = P_{1}u_{i}P_{2}}\left(\sum_{P_{2}P_{1} = Q_{1}u_{i}Q_{2}}Q_{1}u_{i}\otimes Q_{2}u_{i} - \sum_{P_{2}P_{1} = Q_{1}u_{i}^{-1}Q_{2}}Q_{1}\otimes Q_{2}\right)\\
      & - \sum_{P = P_{1}u_{i}^{-1}P_{2}}\left(\sum_{P_{2}P_{1} = Q_{1}u_{i}Q_{2}}Q_{1}\otimes Q_{2} - \sum_{P_{2}P_{1} = Q_{1}u_{i}^{-1}Q_{2}}u_{i}^{-1}Q_{1}\otimes u_{i}^{-1}Q_{2}\right)\,.
    \end{split}
  \end{equation*}
  The \textbf{reduced Laplacian} $\Delta\colon \algA_{n} \to \algA_{n}^{\otimes 2}$ is
  then defined by
  \begin{equation*}
    \Delta = \sum_{i=1}^{n}\Delta_{i}\,.
  \end{equation*}

  Let $(P, Q) \in \algA_{n}^{2}$, we define the operator $\opP^{Q}$ by
  \begin{equation*}
    \opP^{Q}P = \sum_{i=1}^{n}(\D_{i} Q)(\D_{i} P)\,.
  \end{equation*}

  We define the operator $D_{i}$, which acts on a monomial $P$ by
  $D_{i}P = \deg_{i}(P)P$, and $D$ by
  \begin{equation*}
    D = \sum_{i=1}^{n}D_{i}\,.
  \end{equation*}
  Furthermore, for an operator \(T\), we introduce its regularization
  $\bar{T} = T D^{-1}$.
\end{definition}

\begin{definition}[{\cite[Definition 13]{guionnet_asymptotics_2015}}]\label{def:master-operator}
  Let $\Pi$ be the orthogonal projection of the polynomials onto
  $\algB^{\perp}$, the algebra of polynomials without a degree 0 term.
  Let $\Pi' = \Id - \Pi$ be the complementary projection of the
  polynomials onto $\algB$.

  Let $\tau$ be a tracial state. We define
  \begin{equation*}
    T_{\tau} = (\Id\otimes\tau + \tau\otimes\Id)\Delta\,.
  \end{equation*}

  The \textbf{master operator} is
  \begin{equation*}
    \Xi_{\tau}^{V} = \Id + \Pi\bar{T}_{\tau} + \bar{\opP}^{V}\,.
  \end{equation*}
  Where $\bar{T}_{\tau}$ and $\bar{\opP}^{V}$ are the regularization
  (see Definition \ref{def:operators-from-appendix}) of $T_{\tau}$ and
  $\opP^{V}$.
\end{definition}

Applying the gradient trick, we obtain for $(g, l) \neq (0, 1)$,
\begin{equation}\label{eq:induction-formal-cumulant-secondary}
  \begin{split}
    \M^{(g), N}_{V, l}(\bm{P}_{[l-1]}\otimes &\Xi^{V}_{\M^{(0), N}_{V, 1}}P_{l})
    = -\M^{(g-1), N}_{V, l+1}(\bm{P}_{[l-1]}\otimes \bar{\Delta} P_{l})\\
                                             &- \sum_{h=1}^{g-1}\M_{V, l}^{(g-h), N}(\bm{P}_{[l-1]}\otimes \M^{(h), N}_{V, 1}\otimes \Id)(\bar{\Delta} P_{l})\\
                                             &-\sum_{\substack{\emptyset \subsetneq I \subsetneq [l-1]\\g_{1} + g_{2} = g}}\M^{(g_{1}), N}_{V, |I|+1}\otimes\M^{(g_{2}), N}_{V, |I^{c}|+1}(\bm{P}_{I}\otimes \bm{P}_{I^{c}}\# \bar{\Delta} P_{l})\\
                                             &-\sum_{j=1}^{l-1}\M^{(g), N}_{V, l-1}(P_{1}, \ldots, \check{P_{j}}, \ldots, P_{l-1}, \bar{\opP}^{P_{j}}P_{l}).
  \end{split}
\end{equation}

\begin{prop}\label{prop:radius-formal-cumulants}
  Fix a potential $V = \sum_{i=1}^{k}z_{i}q_{i}$, with
  $q_{1}, \ldots, q_{k} \in \wordsA_{n}$. Let $g \geq 0$, $l \geq 1$,
  $\bm{P} \in \algA_{n}^{l}$.

  The radius of convergence of $\M^{(g), N}_{V, l}(\bm{P})$ depends
  only on $k$ and $q_{1}, \ldots, q_{k}$, and is greater than
  $R_{V} = \min(\frac{1}{2}(4A_{k}D_{k, \nu})^{-1}, \frac{1}{2k\nu(4A_{k}+\frac{2^{k+2}}{B_{k}})^{\nu}})$,
  where $\nu = \max_{1 \leq i \leq k}\deg q_{i}$ and the constants $A_{k}, B_{k}$, and $D_{k, \nu}$ are those of Proposition \ref{prop:bounds-V0}.
\end{prop}
\begin{proof}
  Let $\bm{P} \in (\wordsA_{n})^{l}$ be monomials. As
  $\M^{(g), N}_{V, l}$ is linear in each polynomial $P_{i}$, the
  result follows from the case where the $P_{i}$ are monomials. Using
  Proposition \ref{prop:bounds-V0}, the series
  $\M^{(0), N}_{V, 1}(\bm{P})$ can be bounded as follows
  \begin{equation*}
    \begin{split}
      |\M^{(0), N}_{V, 1}(\bm{P})|
      &\leq \sum_{\bm{n} \in \N^{k}}\frac{\bm{z}^{\bm{n}}}{\bm{n}!}|\M^{(0), N}_{0, \sum_{i}n_{i}+l}(\underbrace{q_{1}, \ldots, q_{1}}_{n_{1}\text{
      times}}, \ldots, \underbrace{q_{k}, \ldots, q_{k}}_{n_{k}\text{
      times}}, P_{1}, \ldots, P_{l})|\\
      &\leq \frac{(4A_{k})^{\deg \bm{P}}}{B_{k}}\sum_{\bm{n} \in \N^{k}}\bm{z}^{\bm{n}}(4A_{k}D_{k, \nu})^{\bm{n}}\\
      &\leq \frac{(4A_{k})^{\deg \bm{P}}}{B_{k}}\prod_{i=1}^{k}\frac{1}{1 - 4A_{k}D_{k, \nu}z_{i}}\,,\\
    \end{split}
  \end{equation*}
  where Proposition \ref{prop:bounds-V0} is used on the second line.
  Notice that the radius of convergence of the series does not depend
  on $\bm{P}$.

  Assuming that
  $\|\bm{z}\|_{\infty} < \frac{1}{2}(4A_{k}D_{k, \nu})^{-1}$, we get
  \begin{equation*}
    \begin{split}
      \|\M^{(0), N}_{V, 1}\|_{4A_{k}} \leq \frac{2^{k}}{B_{k}}.
    \end{split}
  \end{equation*}

  Define
  \begin{equation*}
    \begin{split}
      K(\xi, V) = \frac{2^{k+1}}{B_{k}}\frac{4A_{k}}{\xi - 4A_{k}} + \|\Pi V\|_{1}\nu\xi^{\nu},
    \end{split}
  \end{equation*}
  where as before $\nu = \max_{1 \leq i \leq k}\deg q_{i}$. Choose
  $\xi = 4A_{k} + \frac{2^{k+2}}{B_{k}}$. Then, assuming than
  \begin{equation*}
    \begin{split}
      \|\Pi V\|_{1} = \sum_{i=1}^{k}|z_{i}| < \frac{1}{2\nu\xi^{\nu}}\,,
    \end{split}
  \end{equation*}
  where $\Pi$ is the projection operator introduced in Definition
  \ref{def:master-operator}, we have $K(\xi, V) < 1$ and
  $\Xi^{V}_{\M^{(0), N}_{V, 1}}$ is an invertible operator
  $\mathcal{B}^{\perp}_{\xi} \to \mathcal{B}^{\perp}_{\xi}$. Note that this is
  satisfied if $\|\bm{z}\|_{\infty} < \frac{1}{2k\nu\xi^{\nu}}$. We thus set
  $R_{V} = \min(\frac{1}{2}(4A_{k}D_{k, \nu})^{-1}, \frac{1}{2k\nu\xi^{\nu}})$.

  We then proceed by induction. Assume that for all
  $(g', l') < (g, l)$ (with the lexicographic order), and for all
  $\bm{P}\in \wordsA_{n}^{l'}$, the series
  $\M^{(g'), N}_{V, l'}(\bm{P})$ has a radius of convergence greater
  than $R_{V}$. Then, the right side of (\ref{eq:induction-formal-cumulant-secondary}) is a holomorphic
  function that is defined on a polydisc of radius $R_{V}$. The left
  side is a holomorphic function defined on a polydisc of radius
  $R_{l, g, V}$ which coincide with the right side. Thus, it can be
  extended to a holomorphic function on a polydisc of radius $R_{V}$.
  The fact that $\Xi^{V}_{\M^{(0), N}_{V, 1}}$ is invertible allows us
  to conclude.
\end{proof}

\subsection{The topological expansion: proof of Theorem \ref{thm:main}}
\label{sec:topol-expans}
We introduce the truncated formal
cumulant (cf. Definition \ref{def:formal_cumulant})
\begin{equation*}
  S^{(g), N}_{V, l} = \sum_{h=0}^{g}\frac{1}{N^{2h}}\M^{(h), N}_{V, l}.
\end{equation*}

We will show that the cumulants $\tW^{N}_{V, l}$ admit a topological
expansion by bounding the errors $\delta^{(g), N}_{V, l}$ defined by
\begin{equation}\label{eq:def-errors}
  \delta^{(g), N}_{V, l} = \tW^{N}_{V, l} - S^{(g), N}_{V, l}\,.
\end{equation}
In particular, we shall set $\delta^{(-1), N}_{V, l} = \tW^{N}_{V, l}$.

We will derive equations on the errors $\delta^{(g), N}_{V, l}$. To make
this clearer, we first consider the case $g=0, l=1$. In that case,
when $\|\bm{z}\|_{\infty} < R_{V}$, we have as a consequence of \eqref{eq:family_DS},
\begin{equation}\label{eq:simple-troncated-DS}
  \frac{1}{2}\left(S^{(0), N}_{V, 1}\otimes \M^{(0), N}_{V, 1} + \M^{(0), N}_{V, 1}\otimes S^{(0), N}_{V, 1}\right)(\partial_{i}P)
  + S^{(0), N}_{1}((\D_{i}V)P)
  = 0\,.
\end{equation}
On the other hand, Proposition \ref{prop:DS-first} implies
\begin{equation}\label{eq:simple-cumulant-DS}
  \tilde{\W}^{N}_{V, 1}\otimes\tilde{\W}^{N}_{V, 1}(\partial_{i} P) + \tilde{\W}^{N}_{V, 1}(\D_{i}V\cdot P)
  = - \frac{1}{N^{2}}\tilde{\W}^{N}_{V, 2}(\partial_{i}P)\,.
\end{equation}
Taking the difference of \eqref{eq:simple-troncated-DS} and \eqref{eq:simple-cumulant-DS}, we get
\begin{equation*}
  \begin{split}
    \frac{1}{2}\big(\delta^{(0), N}_{V, 1}\otimes\tW^{N}_{V, 1} &+ \tW^{N}_{V, 1}\otimes\delta^{(0), N}_{V, 1} + S^{(0), N}_{V, 1}\otimes \delta^{(0), N}_{V, 1} + \delta^{(0), N}_{V, 1}\otimes S^{(0), N}_{V, 1}\big)(\partial_{i}P)\\
                                                                &+ \delta^{(0), N}_{V, 1}((\D_{i}V)P)
                                                                  = -\frac{1}{N^{2}}\tW^{N}_{V, 2}(\partial_{i}P)\,.
  \end{split}
\end{equation*}

We apply the gradient trick as in Appendix \ref{sec:gradient-trick}.
To do so, we replace $P$ by $\D_{i} P$. Making use of Lemma
\ref{lem:use-big-delta} and the fact that
\begin{equation*}
  \frac{1}{2}\left(\delta^{(0), N}_{V, 1}\otimes\tW^{N}_{V, 1} + \tW^{N}_{V, 1}\otimes\delta^{(0), N}_{V, 1}\right)
  ~\text{ and }~\frac{1}{2}\left(S^{(0), N}_{V, 1}\otimes \delta^{(0), N}_{V, 1} + \delta^{(0), N}_{V, 1}\otimes S^{(0), N}_{V, 1}\right)
\end{equation*}
are symmetric, we can make the master operator defined in Definition
\ref{def:master-operator} appear. We get
\begin{equation}\label{eq:simple-gradient-trick}
  \delta^{(0), N}_{V, 1}
  \left(\Xi^{V}_{\tW^{N}_{V, 1}/2 + S^{(0), N}_{V, 1}/2}\right)(P)
  =
  -\delta^{(0), N}_{V, 1}\otimes\delta^{(0), N}_{V, 1}(\bar{\Delta}P) - \frac{1}{N^{2}}\tW^{N}_{V, 2}(\partial_{i}P)\,.
\end{equation}

We now turn to the general case. When $\|\bm{z}\|_{\infty} < R_{V}$, the
truncated formal cumulants satisfy the equation
\begin{equation}\label{eq:eq-truncated}
  \begin{split}
    \sum_{\substack{I \subset [l-1]\\0 \leq f \leq g}}&\frac{1}{2}\left(\frac{1}{N^{2f}}\M^{(f), N}_{V, |I| + 1}\otimes S^{(g -f), N}_{V, |I^{c}| + 1} + \frac{1}{N^{2f}}S^{(g -f), N}_{V, |I| + 1}\otimes \M^{(f), N}_{V, |I^{c}| + 1}\right)(\bm{P}_{I}\otimes\bm{P}_{I^{c}}\# \partial_{i}P_{l})\\
                                                                 &+ S^{(g), N}_{l}(\bm{P}_{[l-1]}\otimes (\D_{i}V)P_{l})\\
                                                                 &= -\frac{1}{N^{2}}S^{(g-1), N}_{V, l}(\bm{P}_{[l-1]}\otimes\partial_{i}P_{l}) -\sum_{j=1}^{l-1}S^{(g), N}_{V, l-1}(P_{1}\otimes \cdots P_{j-1}\otimes P_{j} \otimes P_{l-1}\otimes (\D_{i} P_{j})P_{l})\,.
  \end{split}
\end{equation}
These equations are obtained by summing the equations
\eqref{eq:family_DS_multi_V=0} for different values of $g$, multiplied
by $1/N^{2g}$.

Together with \eqref{eq:DS-first}, these equations imply equation
\eqref{eq:error-intermediate} on the errors defined by \eqref{eq:def-errors}. Before stating the
equation, we explain how this equation is derived. We substract from
\eqref{eq:DS-first} equation \eqref{eq:eq-truncated}. To have all the
terms from \eqref{eq:DS-first} simplify, we must rewrite the most
complicated term
\begin{equation*}
  \sum_{\substack{I \subset [l-1]\\0 \leq f \leq g}}\frac{1}{2}\left(\frac{1}{N^{2f}}\M^{(f), N}_{V, |I| + 1}\otimes S^{(g -f), N}_{V, |I^{c}| + 1} + \frac{1}{N^{2f}}S^{(g -f), N}_{V, |I| + 1}\otimes \M^{(f), N}_{V, |I^{c}| + 1}\right)(\bm{P}_{I}\otimes\bm{P}_{I^{c}}\# \partial_{i}P_{l})\,.
\end{equation*}
We rewrite $S^{(g), N}_{V, l} = \tW^{N}_{V, l} - \delta^{(g), N}_{V, l}$ in the sum
\begin{equation*}
  \begin{split}
    \sum_{0 \leq f \leq g}\frac{1}{N^{2f}}&\M^{(f), N}_{V, |I| + 1}\otimes S^{(g -f), N}_{V, |I^{c}| + 1}
    =\sum_{0 \leq f \leq g}\frac{1}{N^{2f}}\M^{(f), N}_{V, |I| + 1}\otimes \left( \tW^{N}_{V, |I^{c}| + 1} - \delta^{(g -f), N}_{V, |I^{c}| + 1} \right)\\
    &=S^{(g), N}_{V, |I| + 1}\otimes \tW^{N}_{V, |I^{c}| + 1} - \sum_{0 \leq f \leq g}\frac{1}{N^{2f}}\M^{(f), N}_{V, |I| + 1}\otimes \delta^{(g -f), N}_{V, |I^{c}| + 1}\\
    &=\tW^{N}_{V, |I| + 1}\otimes \tW^{N}_{V, |I^{c}| + 1} - \delta^{(g), N}_{V, |I| + 1}\otimes \tW^{N}_{V, |I^{c}| + 1} - \sum_{0 \leq f \leq g}\frac{1}{N^{2f}}\M^{(f), N}_{V, |I| + 1}\otimes \delta^{(g -f), N}_{V, |I^{c}| + 1}\,.
  \end{split}
\end{equation*}
We proceed similarly for
$\sum_{0 \leq f \leq g}\frac{1}{N^{2f}}S^{(g -f), N}_{V, |I^{c}| + 1}\otimes \M^{(f), N}_{V, |I| + 1}$.
After this rewriting and substracting \eqref{eq:DS-first}, we have
\begin{equation}\label{eq:error-intermediate}
  \begin{split}
    \sum_{I \subset [l-1]}&\frac{1}{2}\left(\delta^{(g), N}_{V, |I|+1}\otimes\tW^{N}_{V, |I^{c}|+1} + \tW^{N}_{V, |I|+1}\otimes\delta^{(g), N}_{V, |I^{c}|+1}\right)(\bm{P}_{I}\otimes\bm{P}_{I^{c}}\# \partial_{i} P_{l})\\
    +\sum_{\substack{I \subset [l-1]\\0 \leq f \leq g}}&\frac{1}{2}\left(\frac{1}{N^{2f}}\M^{(f), N}_{V, |I| + 1}\otimes \delta^{(g -f), N}_{V, |I^{c}| + 1} + \frac{1}{N^{2f}}\delta^{(g -f), N}_{V, |I| + 1}\otimes \M^{(f), N}_{V, |I^{c}| + 1}\right)(\bm{P}_{I}\otimes\bm{P}_{I^{c}}\# \partial_{i}P_{l})\\
                                                                 &+ \delta^{(g), N}_{V, l}(\bm{P}_{[l-1]}\otimes (\D_{i}V)P_{l})\\
                                                                 &= -\frac{1}{N^{2}}\delta^{(g-1), N}_{V, l+1}(\bm{P}_{[l-1]}\otimes\partial_{i}P_{l}) -\sum_{j=1}^{l-1}\delta^{(g), N}_{V, l-1}(P_{1}\otimes \cdots P_{j-1}\otimes P_{j} \otimes P_{l-1}\otimes (\D_{i} P_{j})P_{l})\,.
  \end{split}
\end{equation}

Using the gradient trick (see Section \ref{sec:gradient-trick}), these
equations can be rewritten as follows
\begin{equation}\label{eq:error-gradient-trick}
  \begin{split}
    \delta^{(g), N}_{V, l}
    &\left(\bm{P}_{[l-1]}\otimes\Xi^{V}_{\tW^{N}_{V, 1}/2 + \M^{(0), N}_{V, 1}/2}\right)(P_{l})\\
    = &- \frac{1}{N^{2}}\delta^{(g-1)}_{l+1}(\bm{P}_{[l-1]}\otimes \bar{\Delta} P_{l})\\
    &-\frac{1}{2}\left([\tW^{N}_{V, l} + \M^{(0),N}_{V, l}](\bm{P}_{[l-1]}\otimes \Id)\otimes \delta^{(g), N}_{V, 1} + \delta^{(g), N}_{V, 1}\otimes[\tW^{N}_{V, l} + \M^{(0),N}_{V, l}](\bm{P}_{[l-1]}\otimes \Id)\right)(\bar{\Delta}P_{l})\\
    &-\sum_{\emptyset\subsetneq I \subsetneq [l-1]} \frac{1}{2}\left([\tW^{N}_{V, l} + \M^{(0),N}_{V, l}]\otimes \delta^{(g), N}_{V, |I^{c}| +1}+ \delta^{(g), N}_{V, |I| + 1}\otimes [\tW^{N}_{V, l} + \M^{(0),N}_{V, l}]\right)(\bm{P}_{I}\otimes\bm{P}_{I^{c}}\#\bar{\Delta}P_{l})\\
    &-\sum_{\substack{I \subset [l-1]\\1 \leq f \leq g}}\frac{1}{2}\left(\frac{1}{N^{2f}}\M^{(f), N}_{V, |I| + 1}\otimes \delta^{(g -f), N}_{V, |I^{c}| + 1} + \frac{1}{N^{2f}}\delta^{(g -f), N}_{V, |I| + 1}\otimes \M^{(f), N}_{V, |I^{c}| + 1}\right)(\bm{P}_{I}\otimes\bm{P}_{I^{c}}\# P_{l})\\
    &- \sum_{j=1}^{l-1}\delta^{(g)}_{l-1}(P_{1}\otimes \cdots \otimes \check{P_{j}}\otimes  \cdots \otimes P_{l-1}\otimes \bar{\opP}^{P_{j}}P_{l}).
  \end{split}
\end{equation}
The notation $\Id$ in the third line means the identity operator, in
particular terms on the third line must be read as follows:
\begin{equation*}
  [\tW^{N}_{V, l} + \M^{(0),N}_{V, l}](\bm{P}_{[l-1]}\otimes \Id)\otimes \delta^{(g), N}_{V, 1}(P \otimes Q)
  = [\tW^{N}_{V, l} + \M^{(0),N}_{V, l}](\bm{P}_{[l-1]}\otimes P)\otimes \delta^{(g), N}_{V, 1}(Q)\,.
\end{equation*}

The bounds of Proposition \ref{prop:bounds-V0} imply the following
results.
\begin{lemma}\label{lem:bound-first-error}
  Assume that for all $N\geq 1$, $\Tr V$ is real and
  $\|A^{N}_{i}\| \leq 1$ for all $i$ (Hypotheses \ref{hyp:real} and
  \ref{hyp:bound}). There exists $\xi > 1$ and $\epsilon > 0$, such
  that if
  \begin{equation*}
    \begin{split}
      \|\bm{z}\|_{\infty} < \epsilon,
    \end{split}
  \end{equation*}
  then, that for all $g \geq 0$ and $l \geq 1$, we have
  \begin{equation*}
    \begin{split}
      \|\delta^{(0)}_{1}\|_{\xi} \leq \frac{C}{N^{2}}.
    \end{split}
  \end{equation*}
\end{lemma}
\begin{proof}
  Consider the equation \eqref{eq:simple-gradient-trick} for the
  errors with $g=0, l=1$:
  \begin{equation*}
      \delta^{(0)}_{1}\left(\Xi^{V}_{\tW^{N}_{V, 1}/2+\M^{(0), N}_{V, 1}/2}P\right)
      = -\delta^{(0)}_1\otimes\delta^{(0)}_{1}(\bar{\Delta}P)-\frac{1}{N^{2}}\tW^{N}_{V, 2}(\bar{\Delta}P)\,.
  \end{equation*}

  First, notice that the series $\M^{(0), N}_{V, 1}$ satisfies $\|\M^{(0), N}_{V, 1}\|_{4A_{k}} \leq \frac{2^{k}}{B_{k}}$, with $A_{k}$ and $B_{k}$ the constants from Proposition \ref{prop:bounds-V0}. Proposition \ref{prop:bound-T} implies that
  \begin{equation*}
    \begin{split}
      \|\bar{T}_{\M^{(0), N}_{V, 1}}\|_{\xi} \leq \frac{2^{k+1}}{B_{k}}\frac{4A_{k}}{\xi + 4A_{k}}.
    \end{split}
  \end{equation*}

  Let $\xi \geq 32A_{k} \geq 12$ and $0 < \epsilon< R_{V}$, such that
  \begin{equation*}
    \begin{split}
      K(\xi, V) = 2\frac{\xi+1}{\xi(\xi-1)} + \frac{2^{k}}{B_{k}}\frac{4A_{k}}{\xi + 4A_{k}} + \|\Pi V\|_{1}\nu \xi^{\nu} < 1/2.
    \end{split}
  \end{equation*}
  In that case, the operator
  $\Xi^{V}_{\tW^{N}_{V, 1}/2+\M^{(0), N}_{V, 1}/2}\colon \algB^{\perp}_{\xi}\to\algB^{\perp}_{\xi}$
  is invertible.

  We get that
  \begin{equation*}
    \begin{split}
      \|\delta^{(0)}_{1}\|_{\xi}
      &\leq \|\delta^{(0)}_{1}\|_{\xi}\|\bar{T}_{\delta^{(0)}_{1}}\|_{\xi}\|(\Xi^{V}_{\tW^{N}_{V, 1}})^{-1}\|_{\xi} + \frac{C'}{N^{2}}\|\bar{\Delta}\|_{\xi, \xi/2}\|(\Xi^{V}_{\tW^{N}_{V, 1}})^{-1}\|_{\xi},\\
    \end{split}
  \end{equation*}
  where we used that there exists a constant $C' > 0$ such that
  $\|\tW^{N}_{V, 2}\|_{\xi/2} \leq C'$ by \cite[Theorem 22]{guionnet_asymptotics_2015}. Note that for this Theorem 22
  to be applicable, one must show that the sequence of cumulants is
  $\xi$-uniformly bounded in the sense of \cite[Definition 21]{guionnet_asymptotics_2015}. This is shown assuming
  Assumption \ref{hyp:real} in \cite[Corollary 32]{guionnet_asymptotics_2015} for all $\xi \geq 12$.

  The bound on $\M^{(0), N}_{V, 1}$ implies that
  $\|\delta^{(0)}_{1}\|_{4A_{k}} \leq 1 + \frac{2^{k}}{B_{k}} \leq 2$.
  This fact and Proposition \ref{prop:bound-T} give
  \begin{equation*}
    \begin{split}
      \|\bar{T}_{\delta^{(0)}_{1}}\|_{\xi} \leq 2(1 + \frac{2^{k}}{B_{k}})\frac{4A_{k}}{\xi - 4A_{k}} < 1/2.
    \end{split}
  \end{equation*}

  With this result and \cite[Proposition
  19]{guionnet_asymptotics_2015}, we finally get that
  \begin{equation*}
    \begin{split}
      \|\delta^{(0)}_{1}\|_{\xi} &\leq \frac{C'}{N^{2}}\frac{\|(\Xi^{V}_{\tW^{N}_{V, 1}})^{-1}\|_{\xi}}{1 - \|\bar{T}_{\delta^{(0)}_{1}}\|_{\xi}\|(\Xi^{V}_{\tW^{N}_{V, 1}})^{-1}\|_{\xi}} \leq \frac{C'}{N^{2}}\frac{1}{1/2 - K(\xi, V)}.
    \end{split}
  \end{equation*}
\end{proof}

\begin{prop}\label{prop:bound-error}
  Assume that for all $N\geq 1$, $\Tr V$ is real and
  $\|A^{N}_{i}\| \leq 1$ for all $i$ (Hypotheses \ref{hyp:real} and
  \ref{hyp:bound}). There exists $\xi > 1$ and $\epsilon > 0$, such
  that if
  \begin{equation*}
    \begin{split}
      \|\bm{z}\|_{\infty} < \epsilon,
    \end{split}
  \end{equation*}
  then for all $g \geq 0$ and $l \geq 1$, we have
  \begin{equation*}
    \begin{split}
      \|\delta^{(g)}_{l}\|_{2^{l-2}\xi} = \order{N^{-2g-2}}.
    \end{split}
  \end{equation*}
\end{prop}
\begin{proof}
  We proceed by induction on $(g, l)$, with lexicographic order. For
  $l=1, g=0$, the result is given by Lemma \ref{lem:bound-first-error}. Assume now that for
  all $(g', l') < (g, l)$, we have
  \begin{equation*}
    \begin{split}
      \|\delta^{(g')}_{l'}\|_{\xi} = \order{N^{-2g'-2}}.
    \end{split}
  \end{equation*}
  Then, in the equations \eqref{eq:error-gradient-trick} for the
  errors, all the terms on the right side of the equation are of order
  $N^{-2g}$. Note that terms $\delta^{(-1)}_{l} = \tW^{N}_{V, l}$ are
  bounded using \cite[Theorem 22]{guionnet_asymptotics_2015}. This gives the result.
\end{proof}

We can finally prove Theorem \ref{thm:main}.
\begin{proof}[Proof of Theorem \ref{thm:main}.]

  Proposition \ref{prop:bound-error} directly implies Theorem
  \ref{thm:main}, with Hypothesis \ref{hyp:bound-infty} replaced by
  Hypothesis \ref{hyp:bound}. Then, if we only assume Hypothesis
  \ref{hyp:bound-infty}, we set
  \begin{equation*}
    c = \frac{1}{\sup_{N \geq 1}\sup_{1 \leq i \leq p}\|A_{i}^{N}\|}\,.
  \end{equation*}
  We can then replace each matrix $A_{i}^{N}$ by $c A_{i}^{N}$ and
  rescale each coefficient $z_{i}$ of $V$ by an appropriate multiple
  of $c^{-1}$. When the new coefficients of $V$ are small enough, we
  can apply the result obtained with Hypothesis \ref{hyp:bound-infty}
  and obtain the result.
\end{proof}

\begin{remark}
  Notice that for $N$ big enough, the series
  \begin{equation*}
    \begin{split}
      \sum_{h = 0}^{g}\frac{1}{N^{2h}}\M^{(h), N}_{V, l}(P_{1}, \ldots, P_{l})
    \end{split}
  \end{equation*}
  is well defined for all $V$ with $\bm{z}$ small enough, even if
  Hypothesis \ref{hyp:real} is not satisfied. In fact, for any $V$ with $\bm{z}$
  small, provided the cumulants exist, are bounded, and satisfy the
  Dyson-Schwinger equations, the same method applies and the
  asymptotic topological expansion holds.

  The complex asymptotics of the HCIZ and BGW partition functions were
  studied with a different method in \cite{novak_complex_2020}.
\end{remark}

\appendix
\section{Bounds for the sum of maps $\M^{(g), N}_{0, l}$}
\label{sec:bounds-sum-maps}

This appendix gives a detailed proof of Proposition \ref{prop:bounds-V0}.

We assume that $\nu \geq 1$ and that up to cyclic permutation of its
factors we can write $P_{l}$ as $Pu$. If $P_{l}$ has no term $u$, a
similar argument holds with $P_{l} = u^{*}P$. Furthermore, to make
notation less cumbersome, we write
  \begin{equation*}
    \begin{split}
      \M^{(g), N}_{\bm{n}, l}(P_{1}, \ldots, P_{l}) = \M^{(g), N}_{0, \sum_{i}n_{i}+l}(\underbrace{q_{1}, \ldots, q_{1}}_{n_{1}\text{
      times}}, \ldots, \underbrace{q_{k}, \ldots, q_{k}}_{n_{k}\text{
      times}}, P_{1}, \ldots, P_{l}),
    \end{split}
  \end{equation*}
  and omit the indices $k$ and $\nu$ in the constants.

  To prove the result, we do an induction on $\N^{k+3}$, where we
  endow a tuple $(g, l, n_{1}, \ldots, n_{k}, m)$ with the lexicographic
  order. Notice that the result is obvious when
  $n_{1} = \ldots = n_{k} = 0$, $g=0$, $l=1$ when $\deg \bm{P}=1$ (as
  $\M^{(0), N}_{0, 1}(MU^{\pm 1}) = 0$ for all $M \in \wordsB$), and when
  $m = 1 = \frac{1}{2}\deg\bm{P}$ (as
  $\M^{(0), N}_{0, 1}(M_{1}UM_{2}U^{-1}) = \tr M_{1}\tr M_{2}$), as
  soon as $M \geq 1$. In fact, by Lemma \ref{lem:bound_01},
  $|\M^{(0), N}_{\bm{0}, 1}(P)| \leq 1$ for all $P \in \wordsA$.

  Assuming that $m \geq 2$ or $\bm{n} \neq 0$, Theorem
  \ref{thm:induction_V=0} yields
  \begin{equation}\label{eq:induction-for-bound}
    \begin{split}
      \frac{1}{\bm{n}!}\M^{(g), N}_{\bm{n}, l}&(P_{1}, \ldots, P_{l-1}, Pu)
      = -\frac{1}{\bm{n}!}\M^{(g-1), N}_{\bm{n}, l+1}(P_{1}, \ldots, P_{l-1}, (\partial P)\times 1\otimes u)\\
      &-\sum_{\substack{I\subset [l-1]\\\bm{n_{1}} + \bm{n_{2}} = \bm{n}}}\sum_{g_{1} + g_{2} = g}\frac{1}{\bm{n_{1}}!}\frac{1}{\bm{n_{2}}!}\M^{(g_{1}), N}_{\bm{n_{1}}, |I|+1}\otimes\M^{(g_{2}), N}_{\bm{n_{2}}, |I^{c}|+1}(\bm{P}_{I}\otimes \bm{P}_{I^{c}}\# (\partial P)\times 1 \otimes u)\\
      &- \sum_{j=1}^{l-1}\frac{1}{\bm{n}!}\M^{(g), N}_{\bm{n}, l-1}(P_{1}, \ldots, \check{P_{j}}, \ldots, P_{l-1}, (\D P_{j})Pu)\\
      &- \sum_{j=1}^{k}\frac{1}{(\bm{n} - 1_{j})!}\M^{(g), N}_{\bm{n}-1_{j}, l}(P_{1}, \ldots, P_{l-1}, (\D q_{j})Pu),\\
    \end{split}
  \end{equation}
  where $\check{P_{j}}$ means that $P_{j}$ is removed.

  Now assuming that the bound \eqref{eq:bound-V0-induction} holds for
  $(g', l', n_{1}', \ldots, n_{k}', m') < (g, l, n_{1}, \ldots, n_{k}, m)$, we
  get four terms from \eqref{eq:induction-for-bound}.
  \begin{enumerate}
    \item \begin{equation*}
      \begin{split}
        \frac{1}{\bm{n}!}|\M^{(g-1), N}_{\bm{n}, l+1}
        &(P_{1}, \ldots, P_{l-1}, (\partial P)\times 1\otimes u)|\\
        &\leq \sum^{\deg Pu}_{m' = 1}A^{(l+1)(2m + \nu\bm{n})}B^{-l-1}C^{(g-1)(2m + \nu\bm{n})}D^{\bm{n}}c_{m'}c_{\deg Pu - m'}\prod_{i=1}^{l-1}c_{\deg P_{i}}\prod_{j=1}^{k}c_{n_{j}}\\
        &\leq A^{(l+1)(2m + \nu\bm{n})}B^{-l-1}C^{(g-1)(2m + \nu\bm{n})}D^{\bm{n}}(c_{\deg Pu +1} - c_{\deg Pu})\prod_{i=1}^{l-1}c_{\deg P_{i}}\prod_{j=1}^{k}c_{n_{j}}\\
        &\leq \frac{3}{B}\left(\frac{A}{C}\right)^{2m}A^{l(2m + \nu\bm{n})}B^{-l}C^{g(2m + \nu\bm{n})}D^{\bm{n}}\prod_{i=1}^{l}c_{\deg P_{i}}\prod_{j=1}^{k}c_{n_{j}}.\\
      \end{split}
    \end{equation*}
  \end{enumerate}
  In the second line, we expanded the non-commutative derivative (see
  Definition \ref{def:nc_derivative}). In the third line we used the
  recurrence formula for Catalan numbers
  $c_{n+1} = \sum^{n}_{i=0}c_{i}c_{n-i}$ and in the fourth line we used
  that $c_{n+1} \leq 4c_{n}$ for all $n \in \N$. We choose $A$ and $C$ so
  that $A/C \leq 1$ and $B \geq 12$.

  \begin{enumerate}
    \setcounter{enumi}{1}
    \item
          \begin{equation*}
            \begin{split}
              \sum_{\substack{I\subset [l-1]\\\bm{n_{1}} + \bm{n_{2}} = \bm{n}}}
              &\sum_{g_{1} + g_{2} = g}\frac{1}{\bm{n_{1}}!}\frac{1}{\bm{n_{2}}!}|\M^{(g_{1}), N}_{\bm{n_{1}}, |I|+1}\otimes\M^{(g_{2}), N}_{\bm{n_{2}}, |I^{c}|+1}|(\bm{P}_{I}\otimes \bm{P}_{I^{c}}\# (\partial P)\times 1 \otimes u)\\
              & \leq \sum_{m' = 1}^{\deg P}\sum_{I\subset [l-1]}\sum_{g_{1} + g_{2} = g}A^{m_{1}(|I| + 1) + m_{2}(|I^{c}| + 1)}B^{-l-1}C^{g_{1}m_{1}+g_{2}m_{2}}c_{m'}c_{\deg Pu - m'}\\
              &\times \sum_{\bm{n_{1}} + \bm{n_{2}} = \bm{n}}A^{l\nu\bm{n}}C^{g\nu\bm{n}}D^{\bm{n}}\prod_{i=1}^{l-1}c_{\deg P_{i}}\prod^{k}_{i=1}c_{n_{1, i}}c_{n_{2, i}},\\
            \end{split}
          \end{equation*}
          where we used the notation $m_{1} = \sum_{i \in I}\deg P_{i} + m'$ and $m_{2} = \sum_{i\in I^{c}}\deg P_{i} + \deg Pu-m'$. With this notation, we get as soon as $C \geq 2$,
          \begin{equation*}
            \begin{split}
              \sum_{g_{1} + g_{2} = g}C^{g_{1}m_{1}+g_{2}m_{2}}
              &= C^{2mg}\sum_{h=0}^{g}\left(\frac{1}{C^{m_{2}}}\right)^{h}\left(\frac{1}{C^{m_{1}}}\right)^{g-h}\\
              &\leq C^{2mg}\sum_{h=0}^{g}2^{-g}\\
              &\leq C^{2mg}\,.
            \end{split}
          \end{equation*}
          In the second line, we used that $m_{1}, m_{2} \geq 1$.
          Similarly, we have when $A \geq 2$,
          \begin{equation*}
            \begin{split}
              \sum_{I\subset [l-1]}A^{m_{1}(|I|+1)+ m_{2}(|I^{c}| + 1)}
              &= A^{2ml}\sum_{I\subset [l-1]}A^{-m_{1}|I^{c}| - m_{2}|I|}\\
              &\leq A^{2ml}\sum_{i=0}^{l-1}\binom{l-1}{i}\left(\frac{1}{A^{\deg Pu - m'}}\right)^{i}\left(\frac{1}{A^{m'}}\right)^{l-i-i}\\
              &= A^{2ml}(\frac{1}{A^{\deg Pu - m'}} + \frac{1}{A^{m'}})^{l-1}\\
              &\leq A^{2ml}.
            \end{split}
          \end{equation*}
          We finally get
          \begin{equation*}
            \begin{split}
              \sum_{\substack{I\subset [l-1]\\\bm{n_{1}} + \bm{n_{2}} = \bm{n}}}
              &\sum_{g_{1} + g_{2} = g}\frac{1}{\bm{n_{1}}!}\frac{1}{\bm{n_{2}}!}|\M^{(g_{1}), N}_{\bm{n_{1}}, |I|+1}\otimes\M^{(g_{2}), N}_{\bm{n_{2}}, |I^{c}|+1}|(\bm{P}_{I}\otimes \bm{P}_{I^{c}}\# (\partial P)\times 1 \otimes u)\\
              &\leq \frac{6\cdot 4^{k}}{B} A^{l(2m + \nu\bm{n})} B^{-l}C^{g(2m+\nu\bm{n})}D^{\bm{n}}\prod_{i=1}^{l}c_{\deg P_{i}}\prod_{i=1}^{k}c_{n_{i}}\\
            \end{split}
          \end{equation*}
        \end{enumerate}
        Thus, we choose $B \geq 6 \cdot 4^{k+1}$.

        \begin{enumerate}
          \setcounter{enumi}{2}
          \item \begin{equation*}
            \begin{split}
              \sum_{j=1}^{l-1}\frac{1}{\bm{n}!}
              &|\M^{(g), N}_{\bm{n}, l-1}(P_{1}, \ldots, \check{P_{j}}, \ldots, P_{l-1}, (\D P_{j})Pu)|\\
              &\leq \sum_{j=1}^{l-1}(\deg P_{j})A^{(l-1)(2m+\nu\bm{n})}B^{-l+1}C^{g(2m + \nu\bm{n})}D^{\bm{n}}c_{\deg P_{j} + \deg Pu}\prod_{\substack{i=1\\i\neq j}}^{l-1}c_{\deg P_{i}}\prod_{j=1}^{k}c_{n_{j}}\\
              &\leq \frac{B}{A^{2m+\nu\bm{n}}}\left(\sum_{j=1}^{l-1}(\deg P_{j})\frac{c_{\deg P_{j} + \deg Pu}}{c_{\deg P_{j}}c_{\deg Pu}}\right)A^{l(2m + \nu\bm{n})}B^{-l}C^{g(2m + \nu\bm{n})}D^{\bm{n}}\prod_{i=1}^{l}c_{\deg P_{i}}\prod_{j=1}^{k}c_{n_{j}}.\\
            \end{split}
          \end{equation*}
        \end{enumerate}
        To bound this term, we use the following estimate for the
        Catalan numbers, a consequence of the Stirling bound
        \begin{equation*}
          \begin{split}
            \frac{4^{n}}{(n+1)\sqrt{\pi n}}\exp(\frac{1}{24 n+1} - \frac{1}{24 n}) \leq c_{n} \leq \frac{4^{n}}{(n+1)\sqrt{\pi n}}\exp(\frac{1}{24 n} - \frac{1}{24 n + 2}),
          \end{split}
        \end{equation*}
        which implies
        \begin{equation*}
          \begin{split}
            \frac{4^{n}}{\sqrt{\pi}(n+1)^{3/2}} \leq c_{n} \leq \frac{4^{n}}{\sqrt{\pi}n^{3/2}}.
          \end{split}
        \end{equation*}
        It implies that for $p, q \in \N^{*}$,
        \begin{equation*}
          \begin{split}
            \frac{c_{p+q}}{c_{p}c_{q}} \leq \pi^{1/2}\left(\frac{(p+1)(q+1)}{p+q}\right)^{3/2} \leq \pi^{1/2}(p+1)^{3/2}.
          \end{split}
        \end{equation*}
        Thus,
        \begin{equation*}
          \begin{split}
            \frac{B}{A^{2m+\nu\bm{n}}}\left(\sum_{j=1}^{l-1}(\deg P_{j})\frac{c_{\deg P_{j} + \deg Pu}}{c_{\deg P_{j}}c_{\deg Pu}}\right) &\leq \frac{\pi^{1/2}B}{A^{2m+\nu\bm{n}}}(\deg Pu + 1)^{3/2}(2m - \deg Pu).
          \end{split}
        \end{equation*}
        As we can assume that $m \geq 1$ (else this term could be bounded
        by 0), it suffices to choose $A \leq 2 B^{1/2} \pi^{1/4}2^{3/2}$.
        Notice that for all $n \geq 1$, $(n+1)^{3/2} \leq 2^{3n/2}$.

        \begin{enumerate} \setcounter{enumi}{3}
          \item \begin{equation*}
            \begin{split}
              \sum_{j=1}^{k}\frac{1}{(\bm{n} - 1_{j})!}&\M^{(g), N}_{\bm{n}-1_{j}, l}(P_{1}, \ldots, P_{l-1}, (\D q_{j})Pu)\\
              &\leq \frac{1}{D}\sum_{j=1}^{k}(\deg q_{k})A^{l(2m + \nu\bm{n})}B^{-l}C^{g(2m + \nu\bm{n})}D^{\bm{n}}\frac{c_{\deg Pu + \deg q_{j}}c_{n_{j}-1}}{c_{\deg Pu}c_{n_{j}}}\prod_{i=1}^{l}c_{\deg P_{i}}\prod_{i=1}^{k}c_{n_{i}}\\
              &\leq \frac{1}{D}\sum_{j=1}^{k}4^{\deg q_{k}}(\deg q_{k})A^{l(2m + \nu\bm{n})}B^{-l}C^{g(2m + \nu\bm{n})}D^{\bm{n}}\prod_{i=1}^{l}c_{\deg P_{i}}\prod_{i=1}^{k}c_{n_{i}}\\
              &\leq \frac{1}{D}\sum_{j=1}^{k}(4\ee^{1/\ee})^{\deg q_{k}}A^{l(2m + \nu\bm{n})}B^{-l}C^{g(2m + \nu\bm{n})}D^{\bm{n}}\prod_{i=1}^{l}c_{\deg P_{i}}\prod_{i=1}^{k}c_{n_{i}}.\\
            \end{split}
          \end{equation*}
        \end{enumerate}
        We choose $D = 4k(4\ee^{1/\ee})^{\nu}$ to get the result. Notice
        that we can thus choose
        \begin{equation*}
          \begin{split}
            A &= C = 2^{k+3}\sqrt{6}\pi^{1/4}\\
            B &= 3\cdot 4^{k+1}\\
            D &= 4k(4\ee^{1/\ee})^{\nu}.
          \end{split}
        \end{equation*}

\section{The gradient trick}
\label{sec:gradient-trick}

We use several times the gradient trick, previously introduced in \cite{guionnet_asymptotics_2015}.
The main idea of the gradient trick is to replace the polynomial $P$
(or $P_{l}$) the equations of Proposition \ref{prop:DS-first} (or in the
Dyson-Schwinger problem \eqref{eq:DS-problem}, see Section \ref{sec:dyson-schw-equat}) by its cyclic derivative
$\D_{i} P$. An operator -- the master operator introduced below --
naturally appears in the equations. When the potential $V$ is small
enough, this operator is invertible. The gradient trick was introduced
in \cite{guionnet_asymptotics_2015} to study the Dyson-Schwinger lattice of equations.

\subsection{The trick}
\label{sec:trick}

The gradient trick allows us to simplify quadratic terms. We take as
an example the equation for the sums of maps for $g=0, l=2$
\begin{equation*}
  \begin{split}
      \sum_{I \subset [l - 1]}\M^{(0), N}_{0, |I|+1}\otimes\M^{(0), N}_{0, |I^{c}|+1}(\bm{P}_{I}\otimes \bm{P}_{I^{c}}\# \partial_{i} P_{2})
      &= -\M^{(g), N}_{0, 1}((\D_{i}P_{1})P_{2})\,.\\
  \end{split}
\end{equation*}

We can rewrite it as
\begin{equation*}
  \begin{split}
    \M^{(0), N}_{0, 2}(P_{1}\otimes\Id\otimes\M^{(0), N}_{0, 1} + P_{1}\otimes\M^{(0), N}_{0, 1}\otimes \Id)(\partial_{i}P_{2}) = -\M^{(g), N}_{0, 1}((\D_{i}P_{1})P_{2})\,,
  \end{split}
\end{equation*}
where $\Id$ denote the identity operator. In particular, the notation
$P_{1}\otimes\Id\otimes\M^{(0), N}_{0, 1}(\partial_{i}P_{2})$ must be
understood as follows. For $(Q_{1}, Q_{2}) \in \algA_{n}^{2}$,
\begin{equation*}
  P_{1}\otimes\Id\otimes\M^{(0), N}_{0, 1}(Q_{1}\otimes Q_{2}) = (\M^{(0), N}_{0, 1}(Q_{2}))\cdot (P_{1}\otimes Q_{1}) \in \algA_{n}^{\otimes 2}\,.
\end{equation*}

We now replace $P_{2}$ by its cyclic derivative $\D_{i}P_{2}$, and obtain
\begin{equation*}
  \begin{split}
    \M^{(0), N}_{0, 2}(P_{1}\otimes\Id\otimes\M^{(0), N}_{0, 1} + P_{1}\otimes\M^{(0), N}_{0, 1}\otimes \Id)(\partial_{i}\D_{i}P_{2}) = -\M^{(g), N}_{0, 1}((\D_{i}P_{1})(\D_{i}P_{2}))\,.
  \end{split}
\end{equation*}

\begin{lemma}\label{lem:use-big-delta}
  Let $\mu_{2}\colon \algA_{n}\times \algA_{n} \to \C$ be a bilinear
  form, tracial in each of its variables. For a monomial
  $P\in\wordsA_{n}$, write $\deg_{i}^{+}(P)$ for the number of factors $u_{i}$
  in $P$ and $\deg_{i}^{-}(P)$ for the number of factors $u^{*}_{i}$ in $P$.
  We have for any monomial $P \in \algA_{n}$,
  \begin{equation*}
    \mu_{2}(\partial_{i} \D_{i} P)
    = \deg_{i}^{+}(P) \mu_{2}(P\otimes 1) + \deg_{i}^{-}(P)\mu_{2}(1\otimes P) + \mu_{2}(\Delta_{i} P)\,,
  \end{equation*}
  with the operator $\Delta_{i}$ introduced in Definition
  \ref{def:operators-from-appendix}. In particular, if $\mu_{2}$ is
  symmetric, we get
  \begin{equation*}
    \mu_{2}(\partial_{i} \D_{i} P)
    = \deg_{i}(P)\cdot\mu_{2}(1\otimes P) + \mu_{2}(\Delta_{i} P)\,.
  \end{equation*}
\end{lemma}

This Lemma allows us to rewrite the above expression as
\begin{equation*}
  \M^{(0), N}_{0, 2}\left(\deg_{i}(P_{2})P_{1}\otimes\Id + (P_{1}\otimes\Id\otimes\M^{(0), N}_{0, 1} + P_{1}\otimes\M^{(0), N}_{0, 1}\otimes \Id)(\Delta_{i}P_{2})\right) = -\M^{(g), N}_{0, 1}((\D_{i}P_{1})(\D_{i}P_{2}))\,.
\end{equation*}

Introducing the operator
\begin{equation*}
  \opP^{q}_{i}P = (\D_{i} q)(\D_{i} P),\\
\end{equation*}
for \(P, Q \in \algA_{n}\), we get
\begin{equation*}
  \M^{(0), N}_{0, 2}\left(\deg_{i}(P_{2})P_{1}\otimes\Id + (P_{1}\otimes\Id\otimes\M^{(0), N}_{0, 1} + P_{1}\otimes\M^{(0), N}_{0, 1}\otimes \Id)(\Delta_{i}P_{2})\right) = -\M^{(g), N}_{0, 1}(\opP_{i}^{P_{1}}P_{2})\,.
\end{equation*}
for $1 \leq i \leq n$.

Using the operators defined in Defintion
\ref{def:operators-from-appendix}. The sum of maps
$\M^{(0), N}_{0, 2}$ satisfies
\begin{equation*}
  \M^{(0), N}_{0, 2}\left(P_{1}\otimes\Id + (P_{1}\otimes\Id\otimes\M^{(0), N}_{0, 1} + P_{1}\otimes\M^{(0), N}_{0, 1}\otimes \Id)(\bar{\Delta} P_{2})\right) = -\M^{(g), N}_{0, 1}(\bar{\opP}^{P_{1}}P_{2})
\end{equation*}
for all $P_{1}, P_{2} \in \algA_{n}$. This computation justifies the
introduction of the master operator $\Xi_{\tau}^{V}$ of Defintion
\ref{def:master-operator}.

Thus, we have
\begin{equation*}
  \M^{(0), N}_{0, 2}\left(P_{1}\otimes\Xi^{0}_{\M^{(0), N}_{0, 1}}P_{2}\right) = -\M^{(g), N}_{0, 1}(\bar{\opP}^{P_{1}}P_{2})
\end{equation*}
for all $P_{1}, P_{2} \in \algA_{n}$. This will be called the secondary
form of the equation (\ref{eq:family_DS_multi_V=0}). Notice that in this particular case
$V = 0$. In the sequel, we will derive secondary equation with a
potential.

\subsection{Operator norm estimates}
\label{sec:oper-norm-estim}

We now give some bounds on the norms of the different operators. These
bounds and more were derived in \cite[Section
3.2]{guionnet_asymptotics_2015}. In particular, it was shown that
under some hypotheses the master operator is invertible.

\begin{prop}[{\cite[Section 3.3]{guionnet_asymptotics_2015}}]
  Let $\xi \geq 1$, $V \in \algA$ and $\tau$ a tracial state
  satisfying $\|\tau\| \leq 1$. Introduce
  \begin{equation*}
    K(\xi, V) = 4\frac{\xi + 1}{\xi (\xi - 1)} + \|V\|_{1}\xi^{\deg V}\deg V,
  \end{equation*}
  and assume that $K(\xi, V) < 1$. Then, the operator $\Xi^{V}_{\tau}$
  extends to an operator $\algB^{\perp}_{\xi} \to \algB^{\perp}_{\xi}$
  ($\algB^{\perp}_{\xi}$ is defined in Definition
  \ref{def:parametric-norm}) which is invertible, with inverse
  satisfying
  \begin{equation*}
    \begin{split}
      \|\left(\Xi^{V}_{\tau}\right)^{-1}\|_{\xi} \leq \frac{1}{1 - K(\xi, V)}.
    \end{split}
  \end{equation*}
\end{prop}

We use a slightly modified version of \cite[Proposition
17]{guionnet_asymptotics_2015}.
\begin{prop}\label{prop:bound-T}
  Let $1 \leq \xi_{1} < \xi_{2}$, and $\tau$ a linear form
  $\algA_{n}\to \C$. We have
  \begin{equation*}
    \|\bar{T}_{\tau}\|_{\xi_{2}} \leq 2\|\tau\|_{\xi_{1}}\frac{\xi_{1}}{\xi_{2} - \xi_{1}}.
  \end{equation*}
\end{prop}
\begin{proof}
  We proceed as in \cite{guionnet_asymptotics_2015}. Let $P$ be a
  monomial of degree $d \geq 1$. We have
  \begin{equation*}
    \begin{split}
      T_{\tau}P
      &= \sum_{i=1}^{n}\sum_{P=P_{1}u_{i}P_{2}}\left(\sum_{P_{2}P_{1}u_{i} = Q_{1}u_{i}Q_{2}u_{i}}(Q_{1}u_{i}\tau(Q_{2}u_{i}) + \tau(Q_{1}u_{i})Q_{2}u_{i})\right)\\
      &-\sum_{i=1}^{n}\sum_{P=P_{1}u_{i}P_{2}}\left(\sum_{P_{2}P_{1}u_{i} = Q_{1}u_{i}^{-1}Q_{2}u_{i}}(Q_{1}\tau(Q_{2}) + \tau(Q_{1})Q_{2})\right)\\
      &- \sum_{i=1}^{n}\sum_{P=P_{1}u_{i}^{-1}P_{2}}\left(\sum_{u_{i}^{-1}P_{2}P_{1} = u_{i}^{-1}Q_{1}u_{i}Q_{2}}(Q_{1}\tau(Q_{2}) + \tau(Q_{1})Q_{2})\right)\\
      &+\sum_{i=1}^{n}\sum_{P=P_{1}u_{i}^{-1}P_{2}}\left(\sum_{u_{i}^{-1}P_{2}P_{1} = u_{i}^{-1}Q_{1}u_{i}^{-1}Q_{2}}(u_{i}^{-1}Q_{1}\tau(u_{i}^{-1}Q_{2}) + \tau(u_{i}^{-1}Q_{1})u_{i}^{-1}Q_{2})\right).
    \end{split}
  \end{equation*}
  We now take the norm $\|\cdot\|_{\xi_{2}}$ (recall Definition
  \ref{def:parametric-norm}). Using the triangle inequality and
  $|\tau(P)|\leq \|\tau\|_{\xi_{1}}\xi_{1}^{\deg P}$, we get (as $Q_{1}, Q_{2}$ are monomials)
  \begin{equation*}
    \begin{split}
      \frac{\|T_{\tau}P\|_{\xi_{2}}}{\|\tau\|_{\xi_{1}}}
      &\leq \sum_{i=1}^{n}\sum_{P=P_{1}u_{i}P_{2}}\left(\sum_{P_{2}P_{1}u_{i} = Q_{1}u_{i}Q_{2}u_{i}}(\xi_{2}^{\deg_{i} Q_{1}u_{i}}\xi_{1}^{\deg_{i} Q_{2}u_{i}} + \xi_{1}^{\deg_{i} Q_{1}u_{i}}\xi_{2}^{\deg_{i} Q_{2}u_{i}})\right)\\
      &+\sum_{i=1}^{n}\sum_{P=P_{1}u_{i}P_{2}}\left(\sum_{P_{2}P_{1}u_{i} = Q_{1}u_{i}^{-1}Q_{2}u_{i}}(\xi_{2}^{\deg_{i} Q_{1}}\xi_{1}^{\deg_{i} Q_{2}} + \xi_{1}^{\deg_{i} Q_{1}}\xi_{2}^{\deg_{i} Q_{2}})\right)\\
      &+ \sum_{i=1}^{n}\sum_{P=P_{1}u_{i}^{-1}P_{2}}\left(\sum_{u_{i}^{-1}P_{2}P_{1} = u_{i}^{-1}Q_{1}u_{i}Q_{2}}(\xi_{2}^{\deg_{i} Q_{1}}\xi_{1}^{\deg_{i} Q_{2}} + \xi_{1}^{\deg_{i} Q_{1}}\xi_{2}^{\deg_{i} Q_{2}})\right)\\
      &+\sum_{i=1}^{n}\sum_{P=P_{1}u_{i}^{-1}P_{2}}\left(\sum_{u_{i}^{-1}P_{2}P_{1} = u_{i}^{-1}Q_{1}u_{i}^{-1}Q_{2}}(\xi_{2}^{\deg_{i} u_{i}^{-1}Q_{1}}\xi_{1}^{\deg_{i} u_{i}^{-1}Q_{2}} + \xi_{1}^{\deg_{i} u_{i}^{-1}Q_{1}}\xi_{2}^{\deg_{i} u_{i}^{-1}Q_{2}})\right)\\
      &\leq 2\sum_{i=1}^{n}\deg_{i}P\left(\sum_{k=1}^{\deg_{i} P -1}\xi_{2}^{k}\xi_{1}^{\deg_{i}P-k}\right)\,.
    \end{split}
  \end{equation*}
  The main step of the computation is the second inequality. The
  factor $\deg_{i}P$ accounts for the first sum, on the decompositions
  of $P$ as $P=P_{1}u_{i}P_{2}$ or $P = P_{1}u_{i}^{*}P_{2}$. The sum
  on $k$ accounts from the decompositions of $P_{2}P_{1}$ as
  $P_{2}P_{1} = Q_{1}u_{i}Q_{2}$ or
  $P_{2}P_{1} = Q_{1}u_{i}^{*}Q_{2}$. Finally, there is a factor 2 as
  we count those decompositions at most twice. For instance, we have
  \begin{equation*}
    \begin{split}
      &\sum_{i=1}^{n}\sum_{P=P_{1}u_{i}P_{2}}\left(\sum_{P_{2}P_{1}u_{i} = Q_{1}u_{i}Q_{2}u_{i}}(\xi_{2}^{\deg_{i} Q_{1}u_{i}}\xi_{1}^{\deg_{i} Q_{2}u_{i}} + \xi_{1}^{\deg_{i} Q_{1}u_{i}}\xi_{2}^{\deg_{i} Q_{2}u_{i}})\right)\\
      &+\sum_{i=1}^{n}\sum_{P=P_{1}u_{i}P_{2}}\left(\sum_{P_{2}P_{1}u_{i} = Q_{1}u_{i}^{-1}Q_{2}u_{i}}(\xi_{2}^{\deg_{i} Q_{1}}\xi_{1}^{\deg_{i} Q_{2}} + \xi_{1}^{\deg_{i} Q_{1}}\xi_{2}^{\deg_{i} Q_{2}})\right)\\
      &\leq \sum_{i=1}^{n}\sum_{P=P_{1}u_{i}P_{2}}\sum_{P_{2}P_{1} = Q_{1}u_{i}^{\pm 1}Q_{2}}2\xi_{2}^{\deg_{i} Q_{1}u_{i}^{\pm 1}}\xi_{1}^{\deg_{i} Q_{2}u_{i}^{\pm 1}}\\
      &\leq \sum_{i=1}^{n}\sum_{P=P_{1}u_{i}P_{2}}\sum_{k=1}^{\deg_{i}P-1}2\xi_{2}^{k}\xi_{1}^{\deg_{i}P - 1}\\
      &\leq 2\sum_{i=1}^{n}\deg_{i}^{+}P \sum_{k=1}^{\deg_{i}P-1}\xi_{2}^{k}\xi_{1}^{\deg_{i}P - 1}\,.
    \end{split}
  \end{equation*}
  In the first inequality, we abused notation and wrote $u^{\pm 1}_{i}$
  to mean either of $u_{i}$ or $u_{i}^{*}$. In the last line,
  $\deg_{i}^{+}P$ denotes the number of letter $u_{i}$ in $P$.

  We can then conclude:
  \begin{equation*}
    \begin{split}
      \frac{\|T_{\tau}P\|_{\xi_{2}}}{\|\tau\|_{\xi_{1}}}
      &= 2 \sum_{i=1}^{n} \left( \deg_{i}P \right)\xi_{2}^{\deg_{i}P}\sum_{k=1}^{\deg_{i}P-1} \left( \frac{\xi_{1}}{\xi_{2}} \right)^{\deg_{i}P - k}\\
      &\leq 2 \sum_{i=1}^{n} \left( \deg_{i}P \right)\xi_{2}^{\deg_{i}P}\frac{\xi_{1}}{\xi_{2}}\frac{1}{1 - \xi_{1}/\xi_{2}}\\
      &\leq 2 \sum_{i=1}^{n} \left( \deg_{i}P \right)\xi_{2}^{\deg_{i}P}\frac{\xi_{1}}{\xi_{2} - \xi_{1}}\\
      &\leq 2 d \frac{\xi_{1}}{\xi_{2} - \xi_{1}}\|P\|_{\xi_{2}}\,.
    \end{split}
  \end{equation*}
  In the last line, $d$ is the total degree of $P$.
\end{proof}

\bibliographystyle{alpha}
\bibliography{Article}

\end{document}